\documentclass[a4paper,10pt]{article}

\usepackage{authblk}  

\usepackage{ushort} 

\usepackage{amsmath}
\usepackage{amssymb}
\usepackage{amsthm}
\usepackage{mathrsfs}
\usepackage{dsfont}
\usepackage{mathtools}
\usepackage{euscript}
\usepackage{xcolor}
\usepackage{comment}

\usepackage{hyperref}

\usepackage{color}
\usepackage{graphicx}
\usepackage{bbm}
\usepackage{enumerate}
\usepackage{ifpdf}

\numberwithin{equation}{section}  

\newtheorem{theorem}{Theorem}[section]
\newtheorem{remark}{Remark}[section]
\newtheorem{lemma}[theorem]{Lemma}
\newtheorem{sublemma}[theorem]{Sublemma}
\newtheorem{proposition}[theorem]{Proposition}
\newtheorem{corollary}[theorem]{Corollary}
\newtheorem*{notation}{Notation}

\newcommand{\real}{\mathds{R}}
\newcommand{\entiers}{{\mathds{N}}}
\newcommand{\integers}{\mathds{Z}_+}
\newcommand{\dintegers}{\mathds{Z}^d_+}
\newcommand{\domaine}{\dintegers\backslash\{\vecz\}}
\newcommand{\zentiers}{{\mathds{Z}}}
\newcommand{\rpd}{\mathds{R}_{+}^{d}} 

\newcommand{\qsd}{\mathfrak{m}}  

\newcommand{\proba}{{\mathds{P}}}
\newcommand{\qroba}{{\mathds{Q}}}
\newcommand{\ii}{{\mathrm i}}

\newcommand{\tv}{\mathrm{{\scriptscriptstyle TV}}}

\newcommand{\sK}{{\scriptscriptstyle K}}

\newcommand{\plat}{{\mathfrak N}}

\newcommand{\Oun}{\mathcal{O}(1)} 
\newcommand{\gen}{\EuScript{L}}
\newcommand{\sg}{\EuScript{P}}

\newcommand{\e}{\operatorname{e}}

\newcommand{\un}{\mathds{1}}

\newcommand{\esperance}{{\mathds{E}}}

\newcommand{\vecn}{\ushort{n}}
\newcommand{\vecm}{\ushort{m}}
\newcommand{\vecnt}{{\ushort{n}}_{t}}
\newcommand{\vecu}{\ushort{u}}
\newcommand{\vectheta}{\ushort{\theta}}
\newcommand{\vecv}{\ushort{v}}
\newcommand{\vecw}{\ushort{w}}
\newcommand{\vecz}{{\ushort{0}}}
\newcommand{\vecnf}{\ushort{n}^{*}}
\newcommand{\vecX}{\ushort{F}}         
\newcommand{\vecB}{\ushort{B}}
\newcommand{\vecD}{\ushort{D}}
\newcommand{\vecN}{\ushort{N}}
\newcommand{\vecx}{\ushort{x}}
\newcommand{\vecy}{\ushort{y}}
\newcommand{\vecxf}{\ushort{x}^{*}}
\newcommand{\vecun}{\ushort{1}}
\newcommand{\vece}[1]{\ushort{e}^{(#1)}}
\newcommand{\lyap}{\varphi}

\newcommand{\dd}{{\mathrm d}}
\newcommand{\cour}{\mathcal{H}}

\newcommand{\vg}{\ushort\gamma}
\newcommand{\mb}{\mathfrak{B}}
\newcommand{\md}{\mathfrak{D}}

\newcommand{\cte}[1]{C_{\mathrm{(\ref{#1})}}}
\newcommand{\pcte}[1]{c_{\mathrm{(\ref{#1})}}}
\newcommand{\alcte}[1]{\alpha_{\mathrm{(\ref{#1})}}}
\newcommand{\bcte}[1]{b_{\mathrm{(\ref{#1})}}}
\newcommand{\ceta}[1]{\eta_{\mathrm{(\ref{#1})}}}
\newcommand{\crho}[1]{\rho_{\mathrm{(\ref{#1})}}}
\newcommand{\tcte}[1]{t_{\mathrm{(\ref{#1})}}}
\newcommand{\ttc}{{ t}_{\Delta}}

\newcommand{\rangun}{\EuScript{R}}
\newcommand{\lun}{\ell^{1}(\domaine)}
\newcommand{\linf}{\ell^{\infty}(\domaine)}
\newcommand{\fun}{\mathbf{1}}

\newcommand{\vp}{\mathfrak{u}}

\begin{document}

\title{On time scales and quasi-stationary distributions for multitype birth-and-death processes}

\author[1]{J.-R. Chazottes
\thanks{Email: \texttt{chazottes@cpht.polytechnique.fr}}}

\author[1]{P. Collet
\thanks{Email: \texttt{collet@cpht.polytechnique.fr}}}

\author[2]{S. M\'el\'eard
\thanks{Email: \texttt{sylvie.meleard@polytechnique.edu}}}

\affil[1]{Centre de Physique Th\'eorique, CNRS UMR 7644, F-91128 Palaiseau Cedex (France)}

\affil[2]{Centre de Math\'ematiques Appliqu\'ees, CNRS UMR 7641, F-91128 Palaiseau Cedex (France)}

\date{Dated: \today}

\maketitle

\begin{abstract}
We consider a class of birth-and-death processes describing a population made of 
$d$ sub-populations of different types which interact with one another. The state space is $\dintegers$ (unbounded). We assume that the population goes almost surely to extinction, so that the unique stationary distribution is the Dirac measure at the origin. 
These processes are parametrized by a scaling parameter $K$ which can be thought as the order of magnitude of the total size of the population at time $0$.  For any fixed finite time span, it is well-known that such processes, when renormalized by $K$, are close, in the limit $K\to+\infty$, to the solutions of a certain differential equation in $\rpd$ whose vector field is determined by the birth and death rates.
We consider the case where there is a unique attractive fixed point (off the boundary of the positive orthant) for the vector field (while the origin is repulsive). What is expected is that, for $K$ large, the process will stay in the vicinity of the fixed point for a very
long time before being absorbed at the origin. To precisely describe this behavior, we prove the existence of a quasi-stationary distribution (qsd, for short). In fact, we establish a bound for the total variation distance between the process conditioned to non-extinction before time $t$ and the qsd. This bound is exponentially small in $t$, for $t\gg \log K$. As a by-product, we obtain
an estimate for the mean time to extinction in the qsd. We also quantify how close is the law of the process (not conditioned to non-extinction) either to the Dirac measure at the origin or to the qsd, for times much larger than $\log K$ and much smaller than the mean time to extinction, which is exponentially large as a function of $K$. Let us stress that we are interested in what happens for finite $K$.  We obtain results much beyond what large deviation techniques could provide.

\medskip

\noindent \textbf{Keywords}: Markov jump process, differential equations, competition models, population ecology, mean time to extinction, Lyapunov functions.

\end{abstract}

\newpage

\tableofcontents

\newpage

\section{Introduction}

A fundamental question in population ecology concerns the risk of extinction of populations \cite{OM}.
Stochastic models are well suited to account for the inherently discrete nature of individuals, especially when populations are ``small''. Such models are often referred to as ``individual-based models''.
In contrast, ``large populations'' are traditionally modelled by ordinary differential equations, when the spatial structure, the age-structure, the fluctuations of the environment, etc, are ignored. These ``population-level'' models are supposed to account for the deterministic trends of large populations (the macro\-scale), and are inherently incapable of describing extinction phenomena. 

In the present work we consider birth-and-death processes $(\vecN^{\sK}(t), t\geq 0)$ describing a population made of a finite number of sub-populations of $d$ different types which interact with one another.
At each time $t$, the state of the process is thus given by a vector $\vecn=(n_1,\ldots,n_d)\in \dintegers$, where $n_i$ is the number of individuals of the $i$th sub-population. We assume that these processes depend on a scaling parameter $K>0$ which 
can be thought as the order of magnitude of the total size of the population at time $0$.
More precisely, if  at some time $t$, $\vecN^{\sK}(t)=\vecn$, the rate at which the population is increased (respectively decreased) by one individual of type $j\in\{1,\ldots,d\}$ is $KB_j(\vecn/K)$ (respectively $KD_j(\vecn/K)$).

On the one hand, keeping $K$ fixed and letting $t$ go to $+\infty$, we will show that, under appropriate assumptions, 
the total population goes extinct with probability one. In the context of population ecology, this is a natural assumption to model the truism that ``nothing last forever'', due to the finiteness of ressources. In the terminology of Markov chains, there is an absorbing state, so the stationary distribution (the Dirac measure sitting at this state) is irrelevant as it describes only the state where the population is extinct.

On the other hand,  one can prove that the probability that $\vecN^{\sK}(t)/K$ deviates, over any fixed finite time span, from the solution of the differential equation  
\begin{equation}\label{the-ode}
\frac{\dd \vecx}{\dd t}=\vecB(\vecx)-\vecD(\vecx)
\end{equation}
by more than some prescribed quantity, goes to zero, as $K$ goes to $+\infty$. In the previous equation
$\vecx=(x_1,\ldots,x_d)\in \rpd$, $\vecB(\vecx)=(B_1(\vecx),\ldots,B_d(\vecx))$ and $\vecD(\vecx)=(D_1(\vecx),\ldots,D_d(\vecx))$. Basically, our aim is to describe what happens ``in between'' these two limiting regimes.

Given a differential equation as above, {\em e.g.}, a Lotka-Volterra type equation, one can have repelling fixed points, attracting fixed points (each one with its basin of attraction), limit cycles, ``strange attractors'', etc, see for instance \cite{takeuchi}. In this work we restrict to a simple situation where there is a unique attracting fixed point $\vecxf$ in the interior of $\rpd$ and the origin is a repelling fixed point. The big picture is then intuitively clear: for large (but finite) values of the parameter $K$, one expects that the process will ``feel'' the presence of the deterministic fixed point $\vecxf$ and will stay in the vicinity of the state $\lfloor K\vecxf\rfloor$ for a very long time (``quasi-stationary'' regime), until it is finally absorbed. 
 
Let us informally describe the main results that we obtain.
We firt prove the existence of a unique quasi-stationary distribution (qsd, for short). In fact, we prove a stronger result since
we establish a bound for the total variation distance between the process conditioned to non-extinction before time $t$ and the qsd. This bound is exponentially small in $t$, for $t$ much larger than $\log K$ (see Theorem \ref{Champ-Vill}).
Our second result is an upper bound and a lower bound for the mean time to extinction in the qsd. This mean time is exponential in $K$ (ee Theorem \ref{pertemasse}).
Our third result quantifies how close, in total variation distance, the law of the process not conditioned to non-extinction, is to a convex combination of the Dirac measure at the origin and the qsd (see Theorem \ref{mainincon}).
For $t$ much larger than $\log K$ and much smaller than the mean time to extinction, this distance is very small. Then, for $t$
much larger than $\exp(\Oun K)$, the law of the process not conditioned to non-extinction is very close to the Dirac measure at 
the origin. Our fourth main result show that the spectral gap of this semigroup is larger than $\Oun/\log K$, see
Theorem \ref{ancien}.

We emphasize that we perform a rather fine pathwise analysis of the process. Roughly speaking, we also prove that it takes a time of order one for the process to ``come down from infinity'' and to arrive in a ball of radius of order $K$ and center $\lfloor K \vecxf\rfloor$. This is contained in Sublemma \ref{descente1}.
Afterwards, it takes a time of order $\log K$ to arrive in a ball of radius of order $\sqrt{K}$
and center $\lfloor K \vecxf\rfloor$ (see Lemma \ref{entrec}). Then the process fluctuates around $\lfloor K \vecxf\rfloor$
for a very long time, and is almost distributed according to the qsd.

This work is the natural extension of our work \cite{ccm} on  monotype ({\em i.e.}, $d=1$) birth-and-death processes. 
Therein, we used a precise spectral analysis of a certain self-adjoint operator acting on a suitable ``weighted'' Hilbert space.
We obtained precise estimates, notably for the mean time to extinction, as well as the approximate behavior of the process in terms of a Gaussian distribution. These spectral techniques in Hilbert spaces are lost when $d\geq 2$ since in general the generator cannot be made self-adjoint, as explained in Appendix \ref{selfadjoint}. 
Hence we are forced to follow a different route: we will exploit a theorem proved in \cite{cv}. This abstract theorem gives a necessary and sufficient condition for the exponential convergence, in total variation distance, of the process conditioned on non-extinction toward the quasi-stationary distribution. These conditions are of Doeblin type for submarkovian semigroups. 
In our setting, we have to verify these conditions and a substancial work we have to do is 
to obtain the precise dependence on $K$ of the involved constants.

Let us mention the survey article \cite{AM} which describes how the so-called WKB method can be used to evaluate the mean time and/or probability of population extinction, fixation and switches resulting from either intrinsic (demographic) noise, etc. That article
deals with much more general situations than the one we consider here, but the approach is ``semi-rigorous'' from the mathematical viewpoint.
Let us also mention that there are other papers dealing with quantitative estimates of quasi-stationary distributions
in contexts which are different from ours, namely \cite{bianchi} and \cite{diaconis1,diaconis2}. In particular, the state space is finite in those papers, and different methods are developed. We emphasize that, in the context of stochastic models in population ecology, taking a finite state space is not natural. Indeed, large fluctuations can arise in such a way that we ``go out'' of the state space.

The paper is structured as follows. In Section \ref{setting} we state the hypotheses we make on the vector field $\vecB(\vecx)-\vecD(\vecx)$ and on the birth and death rates.
Section \ref{results} contains our four main results.
In Section \ref{preparatory}, we  construct a Lyapunov function for the generator of the process. We also prove a result (Lemma \ref{disque} ) giving quantitative bounds on the probability of the time the process takes to come down from one level set of the Lyapunov function to a lower one. We expect this quantitative result to be useful in more general situations. 
Section \ref{sec:qsd} is devoted to the proof of the necessary and sufficient conditions required in \cite{cv}. 
More precisely, we prove that the process comes down from infinity and enters a ball centered at $\vecnf$ with a radius of order $\sqrt{K}$. Then we compare the process in this ball with an auxiliary symmetric random walk.
In Section \ref{Champ-Vill} we bound from above and below the parameter of the exponential law of the extinction time
under the qsd.
Section \ref{sec:mixtureandspectrum} is devoted to the proof of a lower bound of the spectal gap of the semigroup associated to the process.

\section{Setting and standing assumptions}\label{setting}

Throughout the paper, we will use the following notations. Elements of $\rpd$ will be denoted
by $\vecx=(x_1,\ldots,x_d)$, and those of $\dintegers$ by $\vecn=(n_1,\ldots,n_d)$. 
For $\vecx\in\rpd$, we will denote by $\|\vecx\|$ its Euclidean norm, by $|\vecx|$ its $\ell^1$-norm, and
by $d(\vecx,\vecy)=\|\vecx-\vecy\|$ the Euclidean distance between $\vecx$ and $\vecy$.
The scalar product in $\real^d$ is denoted by $\langle \cdot\,, \cdot\rangle$. Given $\vecx\in\rpd$ and $r>0$, the 
Euclidean ball of radius $r$ and center $\vecx$ is denoted by $\mathcal{B}(\vecx,r)$.

\subsection{A class of vector fields}

Since we want the process to stay in the positive orthant, we naturally assume the normal component of $\vecD$ of $\rpd$ is zero on the boundary.
We make the following hypotheses on the vector fields $\vecB$, $\vecD$ and  $\vecB-\vecD$.
\begin{itemize}
\item  
The vector fields $\vecB$ and $\vecD$ are locally Lipschitz functions on $\rpd$, and
\begin{equation}
\label{lip}\tag{H0} 
B_j(\vecx)\geq 0, \;  D_j(\vecx)\geq 0, \;\forall j\in\{1,\ldots,d\}, \forall \vecx \in \rpd\,.
\end{equation}
\item
The vector fields $\vecB$ and $\vecD$ vanish only at the origin:
\begin{equation}\label{cond:vanish}\tag{H1}
\vecB(\vecx)= 0 \ \Longleftrightarrow\  \vecD(\vecx)=0 \ \Longleftrightarrow\ \vecx=\vecz.
\end{equation}
The fixed point $\vecz$ of the vector field $\vecB-\vecD$ is linearly unstable.
\item
There exists $\vecxf\in\text{int}({\rpd})$ such that
\begin{equation}\label{cond:fixedpt}\tag{H2}
\vecB(\vecxf)-\vecD(\vecxf)=\vecz.
\end{equation}
\item
There exist $\beta>0$ and $R>L>0$ such that  
\begin{itemize}
\item[(i)]
$\|\vecxf\|<R$ and for all $\vecx\in \rpd$ such that $\|\vecx\|<R$
\begin{equation}\label{cond:lyap}\tag{H3}  
\langle \vecB(\vecx)-\vecD(\vecx),\vecx-\vecxf\rangle \le -\beta\|\vecx\|\|\vecx-\vecxf\|^{2}\, .
\end{equation}
\item[(ii)] $\sum_{j=1}^{d}x^{*}_{j}<L$ and
\begin{equation}\label{cond:hyperplan}\tag{H4}
\mathcal{B}\left(\vecxf, \frac12 \min_{1\leq j\leq d} x^{*}_{j}\right)
\subset
\big\{\vecy\in\rpd : |\vecy|\le L\big\} \subset \mathcal{B}(\vecz, R)\, .
\end{equation}
\end{itemize}
We will denote by $P_L$ the hyperplane defined by 
\begin{equation}\label{def:PL}
\sum_{j=1}^{d}x_{j}=L.
\end{equation}
We refer to Figure \ref{figH3-4} to help the reader visualizing how the different domains defined in Hypotheses
\eqref{cond:lyap} and \eqref{cond:hyperplan} are organized.
\item
Moreover we assume that $L$ is such that
\begin{equation}\label{cond:descente}\tag{H5}
\sup_{s>L}\, \frac{B_{\mathrm{max}}(s)}{D_{\mathrm{min}}(s)}<\frac{1}{2}
\end{equation}
where 
\begin{equation}
\label{domination}
D_{\mathrm{min}}(s) = \inf_{|\vecx|=s} \sum_{j=1}^d D_{j}(\vecx)\quad\text{and}\quad
B_{\mathrm{max}}(s) =\sup_{|\vecx|=s} \sum_{j=1}^d  B_{j}(\vecx).
\end{equation}
\item
We assume that $D_{\mathrm{min}}$ is an eventually monotone function such that 
\begin{equation}\label{cond:integ}\tag{H6}
\int_{1}^{\infty} \frac{\dd s}{D_{\mathrm{min}}(s)} <+\infty\,.
\end{equation}
\item 
There exists $\xi>0$ such that
\begin{equation}\label{cond:mort}\tag{H7}
\inf_{\vecx\in \rpd} \inf_{1\leq j\leq d}\, \frac{D_{j}(\vecx)}{\sup_{1\leq \ell \leq d} x_{\ell}} >\xi>0.
\end{equation}
\item
Finally, we assume that
\begin{equation}\label{cond:nais}\tag{H8}
\inf_{1\leq j\leq d}\partial_{x_{j}} B_{j}(\vecz) > 0.
\end{equation}
(By $\partial_{x_{j}}$ we mean $\frac{\partial}{\partial_{x_{j}}}$.)
\end{itemize}

\begin{figure}[h]
\begin{center}
\ifpdf
\includegraphics[scale=.6,clip]{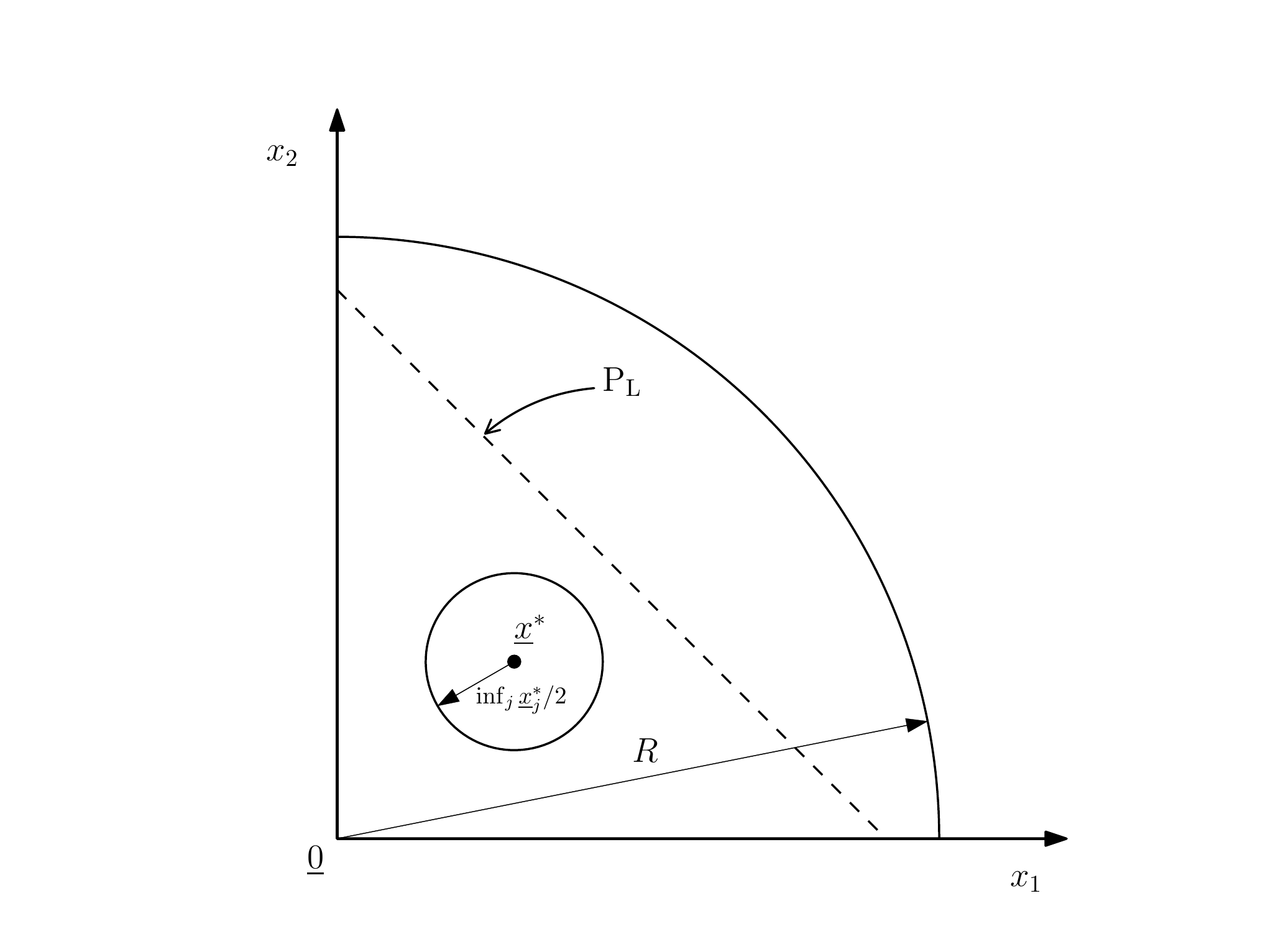}
\else
\includegraphics[scale=.6,clip]{geom_noncrop.eps}
\fi 
\caption{Illustration of Hypotheses \eqref{cond:lyap} and \eqref{cond:hyperplan}}\label{figH3-4}
\end{center}
\end{figure}

We now comment on the different hypotheses.
Notice that, because of the Lipschitz property of the vector field, the polynomial on the right-hand side in \eqref{cond:lyap}
is natural locally around $\vecz$ and $\vecxf$.
Hypothesis \eqref{cond:lyap} implies that the fixed point of $\vecB-\vecD$ is unique in
$\rpd \cap \mathcal{B}(\vecz, R) \backslash \{\vecz\}$.
Any trajectory starting in $\rpd \cap \mathcal{B}(\vecz, R) \backslash \{\vecz\}$ converges to $\vecxf$. The fixed
point $\vecz$ is unstable. In particular, this implies that the faces of $\rpd$ are not globally invariant by the flow.
Notice also that Hypothesis \eqref{cond:descente} implies that there is no fixed point in $\rpd \backslash \mathcal{B}(\vecz, R)$. This hypothesis means that for large populations the death rates dominate the birth rates, this will be used together with Hypothesis \eqref{cond:integ} to prove that the process ``comes down from infinity''. \newline
We will see that Hypothesis \eqref{cond:mort} implies that the jump rate of the process is bounded below away from zero.\newline
 Hypothesis \eqref{cond:nais}  guarantees that the birth rate of the stochastic process is not identically $0$ near the origin.\newline
Finally, notice that Hypotheses \eqref{cond:fixedpt}, \eqref{cond:lyap} (i), \eqref{cond:nais} are open conditions in the $C^{2}$-topology of vector fields. Colloquially, this means that if we slightly perturb the vector field, these hypotheses remain valid with slightly modified constants.

\subsection{An example}

We define $S(\vecx)= \sum_{j=1}^d x_{j}$ and for every $j\in\{1,\ldots,d\}$
\[
B_{j} = \lambda S\,,\;  \ D_{j}= x_{j}(\mu + \kappa S)
\]
where $\lambda>\mu/d >0$ and $\kappa>0$. The non trivial fixed point $\vecxf$ is given by
$x_{j}^* = S^*/d$ where $S^* = (\lambda d - \mu)/\kappa$. We have
\begin{align*}
\langle \vecx - \vecxf, \vecB - \vecD\rangle & = \lambda S (S-S^*)
 -(\mu+\kappa S) \Big(\| \vecx - \vecxf\|^2 +  (S-S^*) \frac{S^*}{d}\Big)\\
&=- \frac{\kappa}{d}\, S (S-S^*)^2 -
(\mu+\kappa S) \left(\| \vecx - \vecxf\|^2-\frac{(S-S^{*})^{2}}{d}\right).
\end{align*}
It is now convenient to use the decomposition
\[
\vecx=\frac{S}{d}\,\vecun+\vecy
\]
where $\vecun$ is the vector with all components equal to $1$,
and $\vecy$ is orthogonal to $\vecun$. We obtain (since $\vecxf=S^{*}\,\vecun/d$)
\[
\langle \vecx - \vecxf, \vecB - \vecD\rangle =
- \frac{\kappa}{d}\, S (S-S^*)^2 -(\mu+\kappa S)\|\vecy\|^{2}.
\]
For $\vecx$ in the positive quadrant we have $\|\vecx\|\le S$, hence
\[
\|\vecy\|\le S.
\]
It is easy to verify that there exists a constant $\Gamma>0$ such that 
for all $S\ge 0$ and all $\|\vecy\|\le S$
\begin{align*}
\|\vecx\| \|\vecx-\vecxf\|^{2} &=\sqrt{\|\vecy\|^{2}+\frac{S^{2}}{d}}
\left(\frac{(S-S^{*})^{2}}{d}+ \|\vecy\|^{2}\right)\\
& \le \Gamma
\left(\,\frac{\kappa}{d}\, S(S-S^*)^2 +(\mu+\kappa S)\|\vecy\|^{2}\right)
\end{align*}
which implies Hypothesis \eqref{cond:lyap} (i) with $\beta=1/\Gamma$. 
Checking the other hypotheses is left to the reader.

Notice that  one can construct many more examples by perturbating (in the $C^2$ sense) this example.

\subsection{The stochastic process and its basic properties}

We consider a birth-and-death process $(\vecN^{\sK}(t),t\geq 0)$ on the $d$-dimensional integer lattice $\dintegers$. 
So, for each $t\geq 0$, $\vecN^{\sK}(t)$ is a vector with $d$ components, that is,
$\vecN^{\sK}(t)=\big((\vecN^{\sK})_1(t),\ldots,(\vecN^{\sK})_d(t)\big)$. 
The birth and death rates of this process are given by  $K B_{j}\left(\frac{\vecn}{K}\right)$ and  $K D_{j}\left(\frac{\vecn}{K}\right)$, $j=1,\ldots,d$.
Given $f:\integers^d\to\real$ with finite support, the generator of the process is given by
\begin{align}
\label{gegene}
\MoveEqLeft \left(\gen\!_{\sK} f\right)(\vecn)= \\
\nonumber
& K \sum_{j=1}^{d} \left[ B_{j}\left(\frac{\vecn}{K}\right)\big(f(\vecn+\vece{j})-f(\vecn)\big)+
D_{j}\left(\frac{\vecn}{K}\right)\big(f(\vecn-\vece{j})-f(\vecn)\big)\right],
\end{align}
where $\vece{j}=(0,\ldots,0,1,0,\ldots,0)$, the $1$ being at the $j$-th position.

\begin{proposition}\label{prop-capitaine-Flam}
For each $K>0$, the process $(\vecN^{\sK}(t), t\geq 0)$ goes to $\vecz$ with probability one.
\end{proposition}

\begin{proof}
For a fixed $K$, the process $\big(\sum_{j=1}^d \langle\vecN^{\sK}(t),\vece{j}\rangle, t\geq 0\big)$ can be stochastically dominated by a monotype birth-and-death process with birth rate $K B_{\mathrm{max}}(m)$ and death rate $K D_{\mathrm{min}}(m)$ with $m\in \integers$ (see \eqref{domination}). 
Hypotheses \eqref{cond:descente} and \eqref{cond:mort} imply that  the process $(\vecN^{\sK}(t), t\geq 0)$ goes almost surely to $\vecz$ (see \cite[Theorem 5.5.5]{m}).
\end{proof}

Under mild assumptions, one-parameter families of pure jump Markov processes can be approximated, in every finite time interval, by the solutions of a differential equation whose vector field is determined by the infinitesimal transition rates.
This is referred to as Kurtz's theorem. In our framework, this result takes the following form.

\begin{proposition}[\cite{EK,K1}]\label{prop-K}
\label{conv-dyn} Let $E\subseteq \rpd$ be  an open bounded subset of $\rpd$. 
Fix a bounded time interval $\left[0,\overline{t}\,\right]$ with $\overline{t}>0$. Let  $\vecx_0\in E$ be such that the trajectory of the solution $\vecx(t)$ of the differential equation
\begin{equation}\label{eq-edo-K}
\frac{\dd \vecx}{\dd t}=\vecB(\vecx)-\vecD(\vecx)
\end{equation}
with initial condition $\vecx_0$ belongs to $E$ for all $t\in [0,\overline{t}\, ]$. If
\[
\lim_{K\to+\infty} \frac{\vecN^{\sK}(0)}{K}=\vecx_0
\]
then, for every $\varepsilon>0$, 
\[
\lim_{K\to+\infty} \mathds{P}\left(\, \sup_{t\leq \overline{t}} \left| \frac{\vecN^{\sK}(t)}{K}-\vecx(t)\right| > \varepsilon\right)=0.
\]
\end{proposition}

According to Propositions \ref{prop-capitaine-Flam} and \ref{prop-K}, we thus have the following picture.
On the one hand, for $K$ fixed, the (total) population dies out with probability one in the limit $t\to+\infty$.
On the other hand, for a fixed finite time span, the number of individuals in the population, when rescaled by $K$,
is very close to the solution of the differential equation \eqref{eq-edo-K} in the limit $K\to+\infty$. The purpose of the present work 
is to describe the process for finite times and for finite $K$. 

\section{Statements of the main results}\label{results}

The hypotheses of Section \ref{setting} are in force in the following four theorems.

We will use the following notations throughout the article.
\begin{notation}
The first entrance time of the process $(\vecN^{\sK}(t),t\geq 0)$ in a subset $A$ of
$\dintegers$ is defined by 
\[
T_A=\inf\{t>0: \vecN^{\sK}(t)\in A\}.
\] 
When $A$ is a singleton, say $\{\vecn\}$, we shall simply write $T_{\vecn}$.
\end{notation}

As usual, $\proba\!_{\vecn}$ will denote the law of the process given that $\vecN^{\sK}(0)=\vecn$, and, for a probability measure $\mu$ on $\dintegers$ and a subset $A$ of
$\dintegers$,
\[
\proba\!_\mu(A)=\sum_{\vecm\in\dintegers} \mu(\vecm)\, \proba\!_{\vecm}(A).
\]

Our first main result is about quantifying the closeness, in total variation distance, of the process condioned to not being
extinct before time $t$, and the quasi-stationary distribution.
Recall that the total variation distance between two probability measures $\mu$ and $\nu$ on $\dintegers$ is
\[
\|\mu -\nu\|_{\tv}
= 
\sup_{A\in \mathscr{P}(\dintegers)} |\mu(A)-\nu(A)|
\]
where $\mathscr{P}(\dintegers)$ is the powerset of $\dintegers$. 

\begin{theorem}
\label{Champ-Vill}
There exist $K_{0} >1$, $0<c<1$ and $0<a<b<+\infty$ such that the following result holds.  For all $K\geq K_{0}$, there exist $t_{0}(K)\in (a \log K, b \log K)$ and a unique probability measure $\qsd_{\sK}$ on $\domaine$ such that for every probability measure $\mu$ on $\domaine$, and for all $t\geq 0$, we have
\[
\|\proba\!_{\mu}\big(\vecN^{\sK}(t)\in \cdot \,|\, t< T_{\vecz}\big) - \qsd_{\sK}(\cdot)\|_{\tv}
\leq 2(1-c)^{\lfloor t/t_{0}(K)\rfloor}.
\]
\end{theorem}

This theorem tells us that for $t\gg \log K$, the process condioned to not being
extinct before time $t$ is very close to the quasi-stationary distribution $\qsd_{\sK}$. As $t$ tends to $+\infty$, we get a convergence of the process conditioned to non-extinction towards the quasi-stationary distribution.

By a general result on quasi-stationary distributions (see for instance \cite{cmm}), one has
\begin{equation}
\label{lambo}
\proba_{\qsd_{\sK}}\big(T_{\vecz}>t\big)=\e^{-\lambda_{0}(K) t},\; t\geq 0,
\end{equation}
where $\lambda_{0}(K)$ is a positive real number called the exponential rate of extinction.
In particular, the mean time to extinction, starting from the quasi-stationary distribution is 
\begin{equation}\label{eq-mte}
\esperance_{\qsd_{\sK}}[T_{\vecz}]=\frac{1}{\lambda_{0}(K)}\,.
\end{equation}

The following theorem shows that the exponential rate of extinction is exponentially small in $K$.

\begin{theorem}\label{pertemasse}
There exists $K_{0}>0$ and two numbers $d_{1}>d_{2}>0$ such that for all $K>K_{0}$
\begin{equation}\label{est-lambo}
\e^{-d_{1} K}\leq  \lambda_{0}(K) \leq \e^{-d_{2} K}.
\end{equation}
\end{theorem}

Hence we get an estimate of the mean time to extinction \eqref{eq-mte}:
\[
\e^{d_{2} K}\leq \esperance_{\qsd_{\sK}}[T_{\vecz}]\leq \e^{d_{1} K}
\]
for all $K>K_0$.
When $d=1$, a more precise estimate was proved in our previous work \cite[Theorem 3.2]{ccm}. 

\begin{remark}
The upper bound in \eqref{est-lambo} could be obtained by a large deviation asymptics for jump processes
(see \cite[Section 4.2]{champagnat}). Theorem \ref{pertemasse} also provides a lower bound. 
In the present paper we are interested, among other things, in the different time scales for large $K$ and not so much in their precise asymptotics. 
\end{remark}

The following theorem provides a quantitative bound for the distance (in total variation) between the law of the process and a convex combination of the quasi-stationary distribution and the Dirac measure at the origin. 

\begin{theorem}\label{mainincon}
Let $c$ and $t_0(K)$ be as in Theorem \ref{Champ-Vill}.
There exist positive constants $\cte{mainincon}$, $\pcte{mainincon}$,
$\ceta{mainincon}$, $K_{0}$, such that for all $t\ge 0$ and all $K>K_{0}$, for each $\vecn\in\domaine$, there exists  a number $p_{\sK}(\vecn)\in(c,1]$ such that 
\begin{align}
\nonumber
\MoveEqLeft \sup_{\vecn\in\domaine}\left\|\, \proba\!_{\vecn}(\vecN^{\sK}(t)\in \cdot) -
 \e^{-\lambda_{0}(K)t} p_{\sK}(\vecn)\,
\qsd_{\sK}(\cdot)-\big(1- \e^{-\lambda_{0}(K)t} p_{\sK}(\vecn)\big)
\delta_{\vecz}(\cdot)\right\|_{\tv}\\
\label{eq:approxmixture}
&  \le 2 \e^{- \ceta{mainincon}K} \e^{-\lambda_{0}(K)t}
+ \, \cte{mainincon}\e^{- \omega(K)  t}
\end{align}
where 
\[
\omega(K) = \frac{- \log(1-c)}{t_{0}(K)}\ge \frac{\pcte{mainincon}}{\log K}\,.
\]
\end{theorem} 

\begin{remark}
Let us give the meaning of inequality \eqref{eq:approxmixture} in two different regimes corresponding to two different time-scales. We assume that $K$ is large enough to have $\e^{- \ceta{mainincon}K}\ll 1$.
First notice that the right-hand side of \eqref{eq:approxmixture} is $\ll 1$ provided that $t\gg \log K$.
Then, for $\log K\ll t\ll 1/\lambda_0(K)$, \eqref{eq:approxmixture} implies that 
\begin{align*}
\MoveEqLeft \sup_{\vecn\in\domaine}\left\|\, \proba\!_{\vecn}(\vecN^{\sK}(t)\in \cdot) -
 p_{\sK}(\vecn)\,
\qsd_{\sK}(\cdot)-\big(1- p_{\sK}(\vecn)\big)
\delta_{\vecz}(\cdot)\right\|_{\tv}\\
&  \le 2 \e^{- \ceta{mainincon}K} \e^{-\lambda_{0}(K)t}
+ \, \cte{mainincon}\e^{- \omega(K)  t} + 2(1-\e^{-\lambda_{0}(K)t})\ll 1.
\end{align*}
This means that, in that time span, the law of the process is close to a mixture of the Dirac measure at the origin and the quasi-stationary distribution with respective weights $1- p_{\sK}(\vecn)$ and $p_{\sK}(\vecn)$. 
For $t\gg 1/\lambda_0(K)$, \eqref{eq:approxmixture} implies  that the law of the process is close to the Dirac measure at the origin.
\end{remark}

Let $(\sg^{\sK}_{t}, t\geq 0)$ be the semigroup of the birth and death process killed at $\vecz$. 
More precisely
\[
\sg^{\sK}_t f(\vecn)=\esperance_{\vecn} \left[ f(\vecN^{\sK}(t))\, \un_{\{t<T_{\vecz}\}}\right]
\]
where $f:\integers^d\backslash \{\vecz\}\to\real$ is any bounded measurable function.
We now  prove that the spectral gap of this semigroup is larger than
$\Oun/\log K$, which is what we obtained in dimension one \cite[Theorem 3.3]{ccm}. 

\begin{theorem}\label{ancien}
The resolvent of $(\sg^{\sK}_{t},t\geq 0)$ in the Banach space $\linf$
is meromorphic in the set $\Re z>-\omega(K) $ with a unique simple pole at 
$-\lambda_{0}(K)$ with residue
the one dimensional  projection $\pi_{\sK}$ given by
\[
\pi_{\sK}(f)=\vp_{\sK} \qsd_{\sK}(f)\,.
\]
The sequence $\big(\vp_{\sK}(\vecn)\big)_{\vecn \in\domaine}$ is such that
$\qsd_{\sK}(\vp_{\sK})=1$,  and, for all $t\ge0$,
\[
\sg^{\sK}_{t}\vp_{\sK}=\e^{-\lambda_{0}(K)t} \vp_{\sK}.
\] 
Moreover, for all $\vecn\in\domaine$,
\[
c\le \vp_{\sK}(\vecn)\le 1+ \e^{-\Oun K}\,,
\]
where $c$ is defined in Theorem \ref{Champ-Vill}.
In particular, the spectral gap $\omega(K) - \lambda_{0}(K)$ is bounded below by
\[
\frac{\pcte{mainincon}}{\log K} - \e^{-d_{2} K}.
\]
\end{theorem}

\begin{remark}
We will see in the proofs that the weights $p_{\sK}(\vecn)$ of Theorem  \ref{mainincon} are equal to $\vp_{\sK}(\vecn)\wedge 1$.
\end{remark}

\section{Some preparatory results}
\label{preparatory}

\subsection{A Lyapunov function}

We first introduce the natural quantity 
\[
\vecnf =\lfloor K \vecxf \rfloor
\]
wich will appear throughout the article. 
Let $\lyap:\integers^d\to \real_+$ defined by
\begin{equation}
\label{lyapou}
\lyap(\vecn)=\e^{\frac{\alpha}{K}\| \vecn-\vecn^*\|^2}
\end{equation}
where $\alpha >0$ is a parameter to be chosen later on.
We now prove that under the previous assumptions and for $\alpha$ small enough, the function $\varphi$ is a Lyapunov function.

\begin{theorem}\label{formule}
There exist $0<\alpha<1/2$,  $K_{0}>0$ and $\cte{formule}>0$ such that for all $K\geq K_{0}$ and for all
$\vecn \in \mathcal{B}(\vecz, R K)$, we have
\[
\gen_{K}\varphi(\vecn) 
\le \bigg(-  \alpha\beta\,  \frac{\|\vecn\|}{K} \frac{\left \|{\vecn} - {\vecnf}\right\|^2}{K}\, + \cte{formule}\bigg)\; \varphi(\vecn)
\]
where $\beta$ and $R$ are defined in \eqref{cond:lyap}.
\end{theorem}

\begin{proof}
We use the elementary fact that for all $x\in\real$ such that $|x|\leq R$ there exists $c_1(R)>0$ such that
\[
0\leq \e^x -1 -x\leq C_1(R)\, x^2.
\]
Then, for all $\vecn \in \mathcal{B}(0,RK)$ we get
\begin{align*}
& \frac{\gen_{K}\varphi(\vecn)}{\varphi(\vecn)} \\
& =K \sum_{j=1}^{d} \left[ B_{j}\left(\frac{\vecn}{K}\right)\left(\frac{\varphi(\vecn+\vece{j})}{\varphi(\vecn)}-1\right)
+ D_{j}\left(\frac{\vecn}{K}\right)\left(\frac{\varphi(\vecn-\vece{j})}{\varphi(\vecn)}-1\right)\right] \\
&= K \sum_{j=1}^{d} \left[\; B_{j}\left(\frac{\vecn}{K}\right)
\left(\exp\Big(\frac{\alpha}{K}(2(n_{j}-n^*_{j}  ) + 1)\Big) -1 \right) \right. \\
& \left. \,\quad\qquad\quad + D_{j}\left(\frac{\vecn}{K}\right)\left(\exp\Big(\frac{\alpha}{K}(- 2(n_{j}-n^*_{j}  ) + 1)\Big) -1\right)\right]\\
& = K \, \sum_{j=1}^{d}\left[\, 2 \alpha\,\left(B_{j}\left(\frac{\vecn}{K}\right) - D_{j}\left(\frac{\vecn}{K} \right)\right)\, \left( \frac{n_{j}-n^*_{j}}{K}\right) \right. \\
& \left. \;\;\qquad\qquad + \, \left(B_{j}\left(\frac{\vecn}{K}\right) + D_{j}\left(\frac{\vecn}{K}\right)\right) \,
\frac{4 C_1(R)\, \alpha^2 \|\vecn - \vecn^*\|^2}{K^2} \right.\\
& \left. \;\;\qquad\qquad  + \left(B_{j}\left(\frac{\vecn}{K}\right)+ D_{j}\left(\frac{\vecn}{K}\right)\right)
\left( \frac{\alpha}{K} + \frac{2C_1(R)(n_j-n^*_j)}{K^2}+\frac{C_1(R)\, \alpha^2}{K^2}\right)
\,\right]\,.
\end{align*}
Using \eqref{lip} and \eqref{cond:vanish}, there exists $C_2(R)>0$ such that 
\[
0\leq B_{j}\left(\frac{\vecn}{K}\right) + D_{j}\left(\frac{\vecn}{K}\right) \leq C_2(R)\frac{\|\vecn\|}{K}
\]
for all $\vecn\in  \mathcal{B}(\vecz, R K)$. It is easy to verify that the third term in the square bracket
is bounded in absolute value by a constant independent of $K$ provided $K$ is larger than some
$K_0>0$. The second term in the square bracket is bounded by 
\begin{equation}\label{eq:trucmuche}
4dC_2(R)\, \alpha^2 \frac{\left\|\vecn\right\|}{K} \frac{\|{\vecn} - {\vecnf}\|^2}{K}\,.
\end{equation}
We finally deal with the first term in the square bracket.
Writing $\vecX=\vecB-\vecD$ for brevity, we obtain by \eqref{cond:lyap} that
\begin{align*}
\MoveEqLeft 2 \alpha K \sum_{j=1}^{d} \Big(B_{j}\left(\frac{\vecn}{K}\right) -
D_{j}\left(\frac{\vecn}{K}\right)\Big)\, \left( \frac{n_{j}-n^*_{j}}{K}\right) \\
&   =  2 \alpha K \left\langle \vecX\left(\frac{\vecn}{K}\right),\frac{\vecn- \vecn^*}{K}\right\rangle\\
& =   2 \alpha K  \left\langle \vecX\left(\frac{\vecn}{K}\right), \left(\frac{\vecn}{K} - {\vecxf}\right) \right\rangle
+ 2  \alpha K  \left\langle \vecX\left(\frac{\vecn}{K}\right), \left(\vecxf - \frac{\vecn^*}{K}\right) \right\rangle\\ 
& \leq  - 2 \alpha\,\beta K \frac{\left\|\vecn\right\|}{K}  \left\|\frac{\vecn}{K} - \frac{\vecnf}{K}\right\|^2 + \Oun\\ 
& \leq  - 2 \alpha \beta\, \frac{\left\|\vecn\right\|}{K} \frac{\|{\vecn} - {\vecnf}\|^2}{K}
+ \Oun,
\end{align*}
where we used that $\left\|\vecxf - \frac{\vecn^*}{K}\right\| \leq \frac{1}{K}$ and 
$\vecX\left(\frac{\vecn}{K}\right)$ is bounded on $\mathcal{B}(0,R)$, and where $\Oun$ is a quantity uniformly bounded in $K$. To finish the proof, we choose $\alpha$ small enough in such a way that the prefactor $4dC_2(R)\, \alpha^2$ in \eqref{eq:trucmuche} is less than half of $2 \alpha \beta$.
\end{proof}

\begin{corollary}\label{tondeuse}
There exist  $K_{0}>0$ and two constants $\crho{tondeuse}>0$ and $\pcte{tondeuse}>4$ such that, for all $K\geq K_{0}$ and for all $\pcte{tondeuse}\leq \|\vecn\| \leq R K$ satisfying
\[
\|\vecn - \vecnf\|\geq \crho{tondeuse}\, \sqrt{K}
\]
we have
\[
\gen_{K}\lyap(\vecn) \leq - \frac{\alpha\beta}{2} \frac{\left\|\vecn\right\|}{K} \frac{\|{\vecn} - {\vecnf}\|^2}{K}\, \varphi(\vecn).
\]
\end{corollary}

\begin{proof}
We choose $\pcte{tondeuse}$ and $\crho{tondeuse}$ large enough such that for $\vecn$ as in the statement,  $\frac{\alpha\beta}{2} \frac{\left\|\vecn\right\|}{K} \frac{\|{\vecn} - {\vecnf}\|^2}{K}> \cte{formule}$.
\end{proof}
\begin{remark}
The intuitive rate of decrease
\[
\frac{\alpha\beta}{2} \frac{\left\|\vecn\right\|}{K} \frac{ \|\vecn - \vecnf\|^2}{K}
\] 
of the Lyapunov function, given by Corollary \ref{tondeuse}, is uniformly bounded below by the  constant $\cte{formule}$, if $\pcte{tondeuse}\leq \|\vecn\| \leq R K$ and $\|\vecn - \vecnf\|\geq \crho{tondeuse}\, \sqrt{K}$.
However, if $\|{\vecn} \|$ and $ \|{\vecn} - {\vecnf}\|$ are of order $K$, this rate is also of order $K$. We will  later take advantage of this non uniformity of the rate by a suitable decomposition of the set 
$\domaine \cap \{ \|\vecn \|\geq \pcte{tondeuse}$\}. 
\end{remark}

\subsection{Lemma of the four domains}

In this section, we formulate a lemma and a corollary of it which will help us to take advantage of the decomposition of the space $\dintegers$. We could formulate it in a much more abstract setting. 
Since $K$ plays no role here, we drop the $K$ dependence, hence $\vecN(t)$ stands for 
$\vecN^{\sK}(t)$, $\gen$ for $\gen\!_{\sK}$, etc.

\begin{figure}[h]
\begin{center}
\ifpdf
\includegraphics[scale=.6,clip]{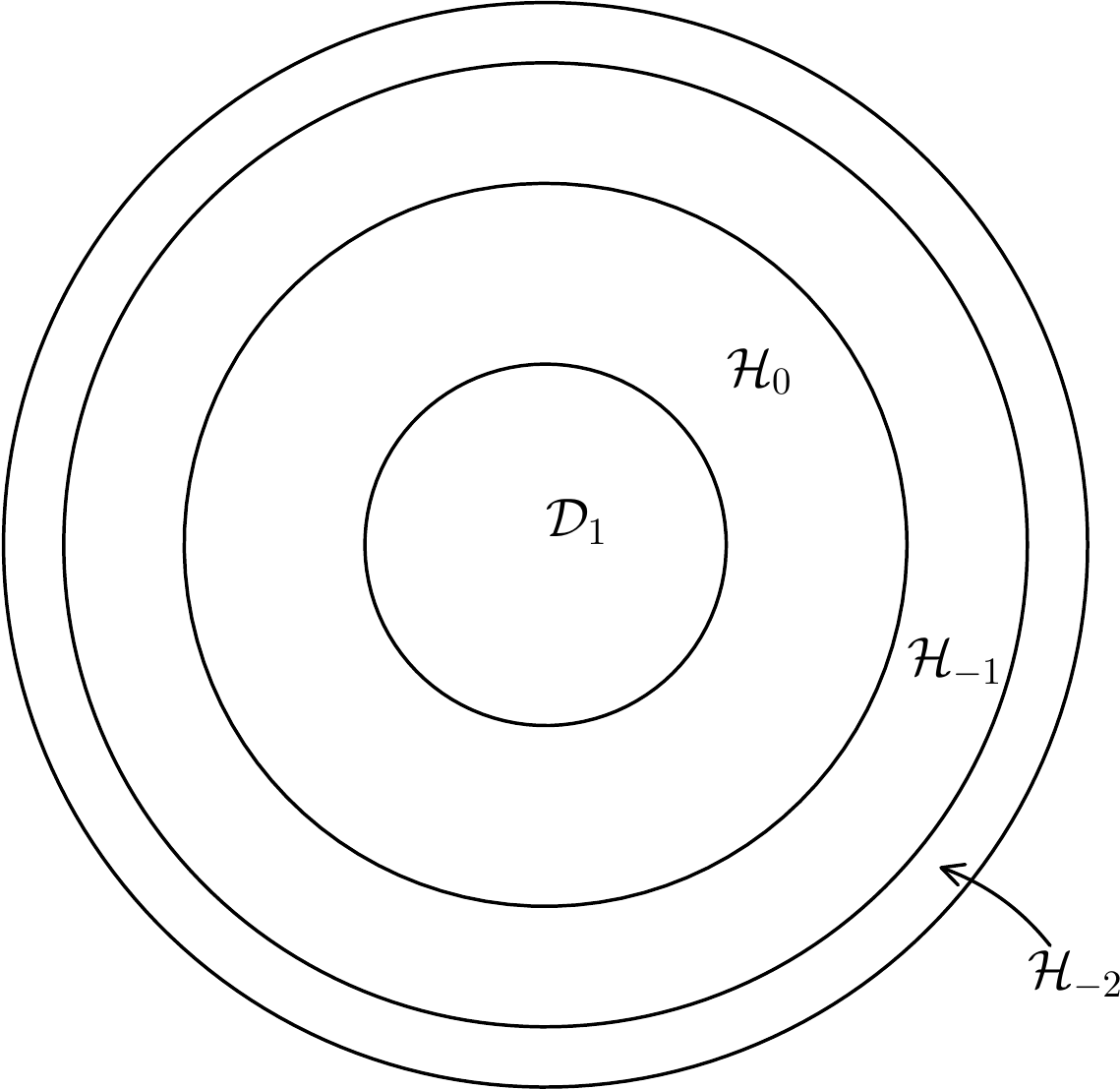}
\else
\includegraphics[scale=.6,clip]{4domaines.eps}
\fi 
\caption{The four domains}\label{fig4dom}
\end{center}
\end{figure}

\begin{lemma}\label{disque}
Let $\mathcal{D}_{-2}, \mathcal{D}_{-1}, \mathcal{D}_{0}, \mathcal{D}_{1}$ be subsets
of $\domaine$ such that 
\[
 \mathcal{D}_{1}\varsubsetneq \mathcal{D}_{0}\varsubsetneq
 \mathcal{D}_{-1} \varsubsetneq \mathcal{D}_{-2} 
\varsubsetneq\domaine,
\]
with $\mathcal{D}_{-2}$ a compact subset.
Next, let 
\[
\cour_{-2}=\mathcal{D}_{-2}\backslash \mathcal{D}_{-1}\,,\quad
\cour_{-1}=\mathcal{D}_{-1}\backslash \mathcal{D}_{0}\,,\quad
\cour_{0}=\mathcal{D}_{0}\backslash \mathcal{D}_{1}\,.
\]
(See Figure \ref{fig4dom}.)
Assume that  for all $\vecn\in \cour_{0}$ we have
\[
\proba_{\vecn}\big(T_{\cour_{-2}}<\infty\big)=1,
\]
and
\[
\cour_{-2}\cap \mathcal{D}_{1}=\emptyset\quad\text{and}\quad 
\{\vecn :  d(\vecn,  \cour_{0} \cup \cour_{-1}) = 1\} \subset
\mathcal{D}_{1} \cup \cour_{-2}.
\]
Assume that there exists a positive function $\psi$ 
defined in  $\domaine$ such that
\[
\Lambda:=-\sup_{\cour_{-2}\,\cup\,\cour_{-1}\,\cup\,\cour_{0}}
\frac{\gen\psi(\vecn)}{\psi(\vecn)}>0.
\]
Let
\[
a_{0}=\sup_{\vecn\in \cour_{0}}\psi(\vecn),\quad
a''_{-2}=\inf_{\vecn\in \cour_{-2}}\psi(\vecn)
\quad\text{and}\quad
a'_{-1}=\inf_{\vecn\in \cour_{-1}\,\cup\,\cour_{0}}\psi(\vecn).
\]
Assume that $a_{0}/a''_{-2}<1$. Then
\[
\inf_{\vecn\in
 \cour_{0}}\proba_{\vecn}\big(T_{\mathcal{D}_{1}}\le
t\,, \,T_{\cour_{-2}}>T_{\mathcal{D}_{1}})\ge   
1-\frac{a_{0}}{a''_{-2}}-\frac{a_{0}}{a'_{-1}}\;\e^{-\Lambda t}.
\]
\end{lemma}

Note that $a_{0}/a'_{-1}\ge 1$. In practice we will use for $\cour_{-2}$ some kind of outer boundary of $\mathcal{D}_{-1}$.
\begin{proof}
Using Dynkin's formula, we have for a path issued from $\vecn\in\cour_{0}$
\begin{align*}
\MoveEqLeft \e^{\Lambda\left(t\wedge T_{\mathcal{D}_{1}}\wedge T_{\cour_{-2}}\right)}
\;\psi\big(\vecN(t\wedge  T_{\mathcal{D}_{1}}\wedge  T_{\cour_{-2}})\big)\\
&=  \int_{0}^{t\wedge  T_{\mathcal{D}_{1}}\wedge  T_{\cour_{-2}}}\e^{\Lambda s}
\big( \Lambda\psi(\vecN(s))+\gen\psi(\vecN(s))\big) \dd s+
\mathcal{M}(t\wedge T_{\mathcal{D}_{1}}\wedge T_{\cour_{-2}})
\end{align*}
where $\left(\mathcal{M}(t \wedge T_{\mathcal{D}_{1}}\wedge T_{\cour_{-2}})\right)_{t\geq 0}$ is a martingale. Using the assumptions and the fact that $\psi$ is bounded by $a_{0}$ on $ \cour_{0}$  we obtain
\begin{equation}\label{calcsto4d}
\esperance_{\vecn}\left[\e^{\Lambda\left(t\wedge T_{\mathcal{D}_{1}}\wedge T_{\cour_{-2}}\right)}
\;\psi\big(\vecN(t\wedge  T_{\mathcal{D}_{1}}\wedge  T_{\cour_{-2}})\big)\right]
\le \psi\big(\vecn\big)\le a_{0} .
\end{equation}
Since $\psi$ is positive we deduce that 
\begin{align*}
a_{0} & \ge \esperance_{\vecn}\left[
\psi\big(\vecN(t\wedge  T_{\mathcal{D}_{1}}\wedge  T_{\cour_{-2}})\big)\right]\\
& \ge \esperance_{\vecn}
\left[
\psi\big(\vecN(t\wedge  T_{\mathcal{D}_{1}}\wedge T_{\cour_{-2}})\big)
\;\un_{\{T_{\mathcal{D}_{1}}\ge T_{\cour_{-2}}\}} \un_{\{T_{\cour_{-2}}\le t\}}\right]\\
& =\esperance_{\vecn}\left[
\psi\big(\vecN(T_{\cour_{-2}})\big)\;\un_{\{T_{\mathcal{D}_{1}}\ge T_{\cour_{-2}}\}}
\un_{\{T_{\cour_{-2}}\le t\}}\right]\\
& \ge a''_{-2}\,\proba_{\vecn}\big(T_{\mathcal{D}_{1}}\ge T_{\cour_{-2}}\,,\, T_{\cour_{-2}}\le t\big).
\end{align*}
Letting $t$ tend to infinity and using our hypothesis (and Lebesgue's
dominated convergence theorem) we get that for all $\vecn\in\cour_{0}$
\[
\proba_{\vecn}\big(T_{\cour_{-2}}\leq T_{\mathcal{D}_{1}}\big)\le \frac{a_{0}}{a''_{-2}}\, .
\]
Using again \eqref{calcsto4d} we also have that for all $\vecn\in \cour_{0}$
\[
\esperance_{\vecn}\left[\e^{\Lambda t}\psi\big(\vecN(t)\big)\;
\un_{\left\{T_{\cour_{-2}}>T_{\mathcal{D}_{1}}> t\right\}}\right]
\le a_{0},
\]
which implies that for all $t\ge 0$
\[
\proba_{\vecn}\left(T_{\cour_{-2}}>T_{\mathcal{D}_{1}}> t\right)
\le \frac{a_{0}}{a'_{-1}}\e^{-\Lambda t}.
\]
We have for all $\vecn\in \cour_{0}$
\begin{align*}
\proba_{\vecn}\big(T_{\mathcal{D}_{1}}\le t\,,\,T_{\cour_{-2}}>T_{\mathcal{D}_{1}})
& =
\proba_{\vecn}\big(T_{\cour_{-2}}>T_{\mathcal{D}_{1}})-
\proba_{\vecn}\big(T_{\mathcal{D}_{1}}>t\,,\,T_{\cour_{-2}}>T_{\mathcal{D}_{1}})\\
& =1-\proba_{\vecn}\big(T_{\cour_{-2}}\le T_{\mathcal{D}_{1}})
-\proba_{\vecn}\big(T_{\cour_{-2}}>T_{\mathcal{D}_{1}}> t).
\end{align*}
The lemma follows from the above estimates.
\end{proof}

\begin{corollary}\label{t4d}
Under the assumptions of Lemma \ref{disque} we have
\[
\inf_{\vecn\in
 \cour_{0}}\proba_{\vecn}\big(T_{\mathcal{D}_{1}}\le
t_{\mathcal{D}_{1}}\,,T_{\cour_{-2}}>T_{\mathcal{D}_{1}})\ge   
1-\eta_{\mathcal{D}_{1}}
\]
with
\[
t_{\mathcal{D}_{1}}=\frac{1}{\Lambda}\;\log\left(\frac{a''_{-2}}{a'_{-1}}\right)
\quad\text{and} \quad
\eta_{\mathcal{D}_{1}}=\frac{2a_{0}}{a''_{-2}}\,.
\]
The estimate also holds with
\[
\eta_{\mathcal{D}_{1}}=\frac{1}{2}+\frac{a_{0}}{2\,a''_{-2}}
\]
and
\[
t_{\mathcal{D}_{1}}=-\frac{1}{\Lambda}\;\log\left(\frac{a'_{-1}}{2\,a_{0}}
\left(1-\frac{a_{0}}{a''_{-2}}\right)\right).
\]
\end{corollary}

\section{Proof of Theorem  \ref{Champ-Vill}}\label{sec:qsd}

\subsection{Plan for the proof: checking conditions \eqref{cond-A1} and \eqref{cond-A2}}

Our proof relies on a general theorem proved in \cite{cv}. We formulate it in our setting.
Let $(\vecN^{\sK}(t),t\geq 0)$ be the birth-and-death process defined above. Suppose there exists a probability measure
$\nu$ on $E$ such that 
\begin{itemize}
\item There exist $t_0,c_1>0$ such that 
\begin{equation}
\label{cond-A1}\tag{A1}  
\proba\!_{\vecn}\big(\vecN^{\sK}(t_{0})\in \cdot \, |\, t_{0}<T_{\vecz}\big) \geq c_{1}\;\nu(\cdot),\;
\forall \vecn\in \domaine.
\end{equation}
\item
There exists $c_{2} >0$ such that 
\begin{equation}
\label{cond-A2}\tag{A2}  
\proba\!_{\nu}( t<T_{\vecz}) \geq c_{2} \,\proba\!_{\vecn}( t<T_{\vecz}),\; 
\forall \vecn\in \domaine,\;\forall  t\geq 0.
\end{equation}
\end{itemize}
Then there exists a unique quasi-stationary distribution $\qsd_{\sK}$ such that
for every initial distribution $\mu$,
\[
\| \proba\!_{\mu}(\vecN^{\sK}(t)\in \cdot \,|\, t< T_{\vecz}) - \qsd_{\sK}(\cdot)\|_{\tv} \leq 2(1-c_{1}c_{_{2}})^{t/t_{0}}.
\]
We shall take $\nu$ as the uniform probability measure supported on a ball centered at $\vecnf$ with radius of order $\sqrt{K}$. We shall also prove that $c_1$ and $c_2$ are independent of $K$, and that $t_0$ is of order $\log K$.

\subsection{Proof of Condition \eqref{cond-A1}}
\label{a1}

Let
\begin{equation}\label{def:Delta}
\Delta = \mathcal{B}\big(\vecn^*, 2 \,\crho{tondeuse}\sqrt{K}\, \big),
\end{equation}
where $\mathcal{B}(\vecn, r)$ denotes the ball centered in  $\vecn$ with radius $r$ and $\crho{tondeuse}$ the constant introduced in Corollary \ref{tondeuse}. Since $\vecn^*$ is of order $K$,  the set  $\Delta$ is included in the interior of $\integers^d$ for $K$ large enough.

\begin{notation}
We shall denote by $\nu$ the uniform probability measure supported on $\Delta$. 
\end{notation}
This discrete measure thus gives each point of $\Delta$ a mass proportional to $K^{-d/2}$.

The proof of Condition \eqref{cond-A1} relies on the following three lemmas whose proofs are given later on. 

The first lemma shows that   the descent (from infinity) into the set $\Delta$ happens with a  time scale of at most $\log K$.

\begin{lemma}\label{entrec}
There exist $\cte{entrec}>0$ and $\ceta{entrec}<1$ such that for all
$K$ large enough 
\[
\inf_{\vecn\in\Delta^{c}}\proba_{\vecn}\big(T_{\Delta}< \cte{entrec} \log K\big)\ge 1-\ceta{entrec}.
\]
\end{lemma}

The second lemma shows that on a time span of order $\log K$, the process starting in $\Delta$ stays near $\Delta$, more precisely in a ball with a radius of order $\sqrt{K}$ centered at $\vecn^*$.
\begin{lemma}\label{restec}
There exists $\cte{restec}>2 \,\crho{tondeuse}$ and $\ceta{restec}<1$ such that for all
$K$ large enough
\[
\inf_{\vecn\in\Delta}\;\inf_{0\le t\le \cte{entrec} \log K+1}\proba_{\vecn}\big(\vecN^{\sK}(t)\in \Delta'\big)
\ge 1-\ceta{restec}
\]
where
\[
\Delta' = \mathcal{B}\big(\vecn^*, \cte{restec}\sqrt{K}\, \big)\supset \Delta.
\]
\end{lemma}

The third lemma says that the probability measure $\nu$ is a significant component of the distribution of the process at time $1$  starting near $\Delta$. This lemma does not seem to be available in the literature. The main difference with existing results (see for instance \cite{coulhon}) is that our generator is not symmetric. 

\begin{lemma}\label{uni}
There exists $\ceta{uni}<1$ such that for all
$K$ large enough  and all $A\subset \Delta$
\[
\inf_{\vecn\in\Delta'}\;
\proba_{\vecn}\big(\vecN^{\sK}(1)\in A\big)\ge (1-\ceta{uni})\,\nu(A),
\]
where $\Delta'$ is defined in Lemma \ref{restec}.
\end{lemma}

\noindent\textbf{Proof of Condition \eqref{cond-A1}}. Applying the three preceding lemmas, we can prove that condition \eqref{cond-A1} holds for $K$ large enough with
\begin{align}
\label{c1}
c_1 &= (1-\ceta{entrec})\, (1-\ceta{restec})\,(1- \ceta{uni})<1,\\
\label{t0}
t_{0} &= t_{0}(K)=1+\cte{entrec} \log K.t_{0}(K)
\end{align}
Indeed, for all $\vecn\in\domaine$ and for all $A\subset\Delta$ we can write
\[
\proba_{\vecn}\big(\vecN^{\sK}(t_{0})\in A\big)=
\esperance_{\vecn}\left[\un_{A}\big(\vecN^{\sK}(t_{0})\big)\right]\ge 
\esperance_{\vecn}\left[\un_{\{T_{\Delta}<\cte{entrec} \log K\}}
\;\un_{A}\big(\vecN^{\sK}(t_{0})\big)\right].
\]
Now by the Markov property we have
\begin{align*}
\MoveEqLeft \proba_{\vecn}\big(\vecN^{\sK}(t_{0})\in A\big) \\
& \ge 
\esperance_{\vecn}\left[\un_{\{T_{\Delta}<\cte{entrec} \log K\}}
\;\esperance_{\vecN^{\sK}(T_{\Delta})}\big[\un_{A}\big(\vecN^{\sK}(t_{0}-T_{\Delta})\big)\big]\right]\\
& \ge \esperance_{\vecn}\left[\un_{\{T_{\Delta}<\cte{entrec} \log K\}}
\;\esperance_{\vecN^{\sK}(T_{\Delta})}\big[
\un_{\Delta'}\big(\vecN^{\sK}(t_{0}-T_{\Delta}-1)\big)\, Y_{A,\Delta}(t_0)\big]\right],
\end{align*}
where
\[
Y_{A,\Delta}(t_0)=\esperance_{\vecN^{\sK}(t_{0}-T_{\Delta}-1)}
\big[\un_{A}\big(\vecN^{\sK}(1)\big)\big].
\]
Using successively Lemma \ref{uni}, Lemma \ref{restec} and Lemma \ref{entrec} we get
\begin{align}
\label{eureka}
\proba_{\vecn}\big(\vecN^{\sK}(t_{0})\in A\big)\ge 
(1-\ceta{entrec})\, (1-\ceta{restec})\,(1- \ceta{uni})\, \nu(A).
\end{align}
Since $\vecz$ is an absorbing point we have
$\proba_{\vecn}\big(\vecN^{\sK}(t_{0})\in A,T_{\vecz}\leq t_{0}\big)=0$,
and using the trivial estimate
$\proba_{\vecn}\big(T_{\vecz}>t_{0}\big)\le 1$ we get
\begin{align*}
\proba_{\vecn}\big(\vecN^{\sK}(t_{0})\in A\,\big|\, T_{\vecz}>t_{0}\big) &\ge 
\frac{(1-\ceta{entrec})\, (1-\ceta{restec})\,(1- \ceta{uni})}{\proba_{\vecn}\big(T_{\vecz}>t_{0}\big)}\;\nu(A)\\
& \ge (1-\ceta{entrec})\, (1-\ceta{restec})\,(1- \ceta{uni})\;\nu(A).
\end{align*}
Thus we have proved that Condition \eqref{cond-A1} holds.

\subsection{Proof of Lemma \ref{entrec}}

The proof of Lemma \ref{entrec} is based on the fine description of the trajectories of the process. For this purpose, we need to introduce a decomposition of  $\domaine$ according to the different time scales at which the process goes down from infinity to
$\Delta$.
 
Let
\[
R_{*}=\frac{1}{2}\;\Big(R+\sup_{y\in P_L\cap\, \rpd} \|y-\vecxf\|\Big)\,,
\]
where $P_L$ is the hyperplane defined in \eqref{def:PL}.

Note that $R_{*}<R$ by hypothesis \eqref{cond:hyperplan}. We define the sets
\begin{align*}
\mathcal{E}_{1} &=\bigg\{\vecn\in\domaine : \sum_{j=1}^{d}n_{j}>L K\bigg\}\\
H_{-5} &=\big\{\vecn\in\domaine : R_{*}\,K\le \|\vecn-\vecnf\|< R K\big\}\\
H_{-4} &=\bigg\{\vecn\in\domaine : \sum_{j=1}^{d}n_{j}> L K\,,\,\|\vecn-\vecnf\|< R_{*} K\bigg\}\\
H_{-3} &=\bigg\{\vecn\in\domaine : \sum_{j=1}^{d}n_{j}\le L K\,,\,\|\vecn-\vecnf\|\ge \|\vecnf\|-\pcte{tondeuse}\bigg\}\\
H_{-2} &=\big\{\vecn\in\domaine : \|\vecnf\|-(\pcte{tondeuse}+4)\le \|\vecn-\vecnf\|< \|\vecnf\|-\pcte{tondeuse}\big\}\\
H_{-1} &=\big\{\vecn\in\domaine : \|\vecnf\|-(\pcte{tondeuse}+8)\le \|\vecn-\vecnf\|< \|\vecnf\|-(\pcte{tondeuse}+4)\big\}\\
H_{0}&=\big\{\vecn\in\domaine :  \|\vecnf\|-(\pcte{tondeuse}+12)\le \|\vecn-\vecnf\|< \|\vecnf\|-(\pcte{tondeuse}+8)\big\}\\
\mathcal{E}_{2}&=\big\{\vecn\in\domaine : \|\vecn\|<\pcte{tondeuse}+17 \big\}.
\end{align*}
These sets are well-defined provided that $K$ is large enough.

\begin{figure}[h]
\begin{center}
\ifpdf
\includegraphics[scale=1.0,clip]{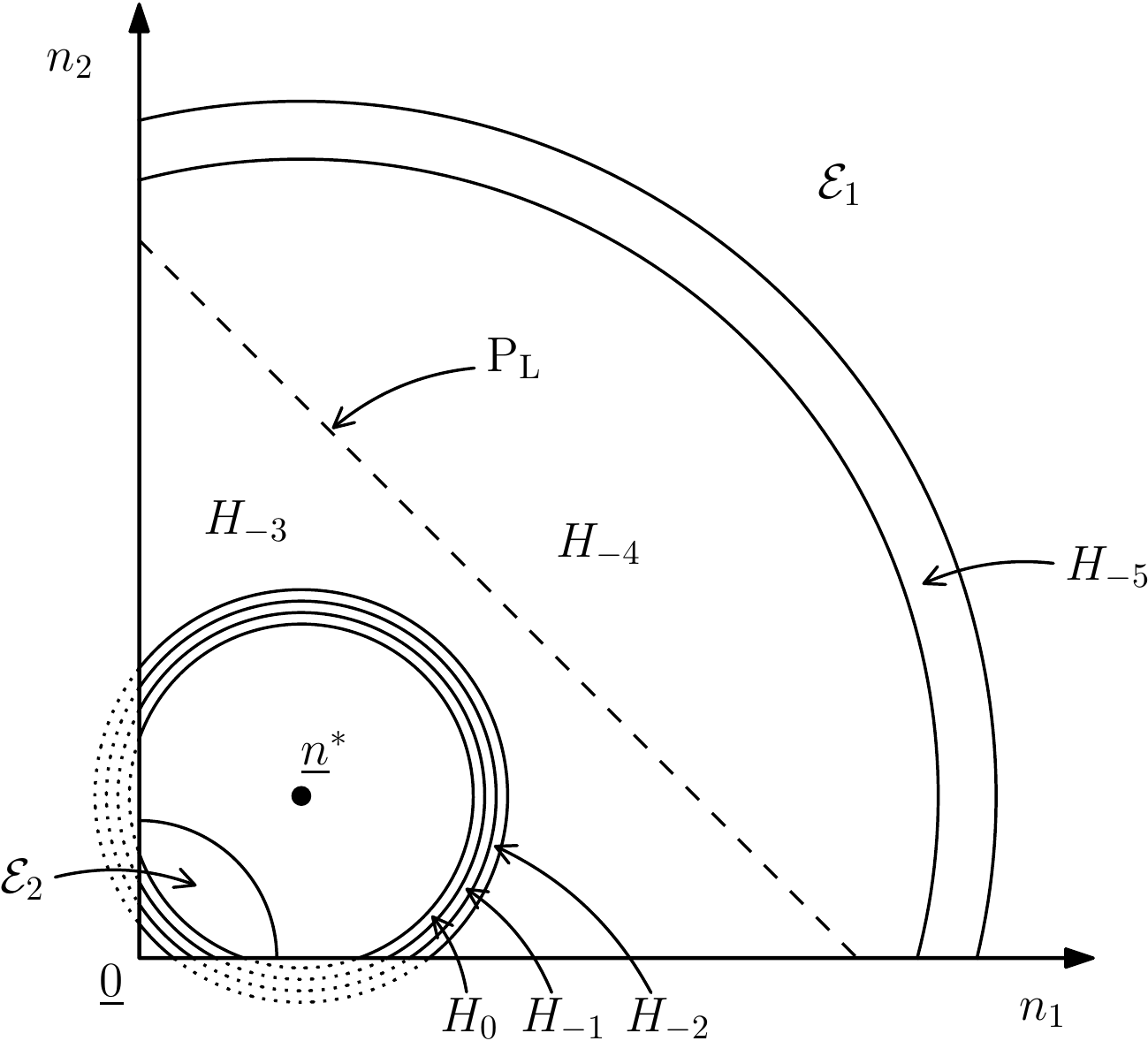}
\else
\includegraphics[scale=1.0,clip]{descente.eps}
\fi
\caption{The various subsets when $d=2$ when $K$ is large enough.}
\end{center}
\end{figure}

The proof  of Lemma \ref{entrec} will result from a series of sublemmas which quantify the probability of coming down from infinity and  crossing the various level sets of the Lyapunov function.

\begin{sublemma}\label{descente1}
There exist two constants $\tcte{descente1}>0$ and $\ceta{descente1}<1$ (independent of $K$) such that for $K$ large enough
\[
\inf_{\vecn\in \mathcal{E}_{1}}\proba_{\vecn}\big(T_{\mathcal{E}_{1}^c}\le
\tcte{descente1}\big)\ge 1-\ceta{descente1}.
\]
\end{sublemma}

\begin{proof}
The process $\big(\sum_{j=1}^d \langle \vecN^{\sK}(t),\vece{j}\rangle, t\geq 0\big)$ can be coupled with a one-dimensional birth-and-death process $(Z(t),t\geq 0)$ with birth rate $\Lambda(m) = K B_{\mathrm{max}}\left(\frac{m}{K}\right)$ and death rate $M(m)= K D_{\mathrm{min}}\left(\frac{m}{K}\right)$. The coupling is such that 
\[
Z(t)\geq \sum_{j=1}^d \langle \vecN^{\sK}(t),\vece{j}\rangle \quad\text{if}\quad Z(0)\geq \sum_{j=1}^d \langle \vecN^{\sK}(0),\vece{j}\rangle.
\]
Let us introduce $p_{\sK} = \lfloor L K\rfloor$ and  denote by $\widehat T_{p_{K}}$ its hitting time.
We are going to prove that 
$A_{\sK}:=\sup_{{p> p_{\sK}}} \esperance_{p}(\widehat T_{p_{K}}) $ is bounded uniformly in $K$.
As shown in \cite[p.384]{TaylorKarlin1998} or in \cite[Chap.3]{Allen2011},  one has
\[
A_{\sK} = \sum_{m=p_{\sK}+1}^\infty
\left(\frac{1}{M(m)} + \sum_{i=m+1}^\infty \frac{\Lambda(m) \cdots \Lambda(i-1)}{M(m) \cdots M(i)}\right).
\]
By assumption \eqref{cond:descente}, for $q\geq p_{\sK}$, $\Lambda(q)/M(q) \leq 1/2$. Then
\[ 
A_{K} \leq \sum_{m=p_{\sK}+1}^\infty \left(\frac{1}{M(m)} + \sum_{i=m+1}^\infty \frac{2^{m-i}}{M(i)}\right)\leq 2 \sum_{m=p_{\sK}+1}^\infty \frac{1}{M(m)}\,,
\]
where we have interchanged the order of the sums to get the second inequality.
By Hypothesis \eqref{cond:integ}, we know that 
\[
\frac{1}{K}\,\sum_{m=p_{K}+1}^\infty \frac{1}{D_{\mathrm{min}}(\frac{m}{K})} 
\xrightarrow[K\to \infty]{} \int_{L}^\infty \frac{\dd s}{D_{\mathrm{min}}(s)}<+\infty.
\]
Then there exists $K_{0}$ such that for all $K\geq K_{0}$, for all $p\geq p_{K}$, we have
\[
\esperance_{p}(\widehat T_{p_{K}}) \leq 3 \int_{L}^\infty \frac{\dd s}{D_{\mathrm{min}}(s)}\,.
\]
The result follows by Markov inequality with 
\[
\tcte{descente1}= 6  \int_{L}^\infty \frac{\dd s}{D_{\mathrm{min}}(s)}
\quad\text{and}\quad \ceta{descente1} = \frac{1}{2}\,.
\]
\end{proof}

\begin{sublemma}\label{descente2}
There exist two constants $\tcte{descente2}>0$ and
$\ceta{descente2}<1$ (independent of $K$) such that for $K$ large enough
\[
\inf_{\vecn\in H_{-3}\,\cup\, H_{-2}\,\cup\, H_{-1}}\proba_{\vecn}\big(T_{H_{0}}\le
\tcte{descente2}\big)\ge 1-\ceta{descente2}.
\]
\end{sublemma}

\begin{proof}
We define 
\begin{align}
\nonumber
D_{-2}&=\big\{\vecn\in\domaine\,:\,\|\vecn-\vecnf\|<R K\big\}\\
\nonumber
D_{-1}&=\big\{\vecn\in\domaine\,:\,\|\vecn-\vecnf\|<R_{*}K\big\}\\
\nonumber
D_{0}&=\left\{\vecn\in\domaine\,:\, \sum_{j=1}^{d}n_{j}\le L K\right\}\\
D_{1}&=\big\{\vecn\in\domaine\,:\,\|\vecn-\vecnf\|< \|\vecnf\|-(\pcte{tondeuse}+8)\big\}.
\label{def-D1}
\end{align}
We now apply Corollary \ref{t4d} with $ \mathcal{D}_{i} = D_{i}$, $i=-2, -1, 0, 1$. For $K$ large enough and using \eqref{cond:hyperplan}, the Lyapunov function $\varphi$ defined in Theorem \ref{formule} and the geometry of the sets, we have
\[
\frac{a''_{-2}}{a'_{-1}}\le \e^{\Oun K},\quad
\frac{a_{0}}{a''_{-2}}<\frac{1}{4}.
\]
Moreover we have
\[
\Lambda = \Oun K
\]
by Theorem \ref{formule}. The result follows since $H_{-3}\,\cup\, H_{-2}\,\cup\, H_{-1} = D_{0}\backslash D_{1}$ and since for $K$ large enough, $D_{1}$ can be reached from $D_{0}\backslash D_{1}$ only
through $H_{0}$. 
\end{proof}

We need a specific estimate near $\vecz$. 
\begin{sublemma}\label{descente3}
There exists $\ceta{descente3}<1$ (independent of $K$) such that for $K$ large enough
\[
\inf_{\vecn\in \mathcal{E}_{2}\backslash D_{1}}\proba_{\vecn}\big(T_{H_{0}}\le
1\big)\ge 1-\ceta{descente3}\,.
\]
\end{sublemma}

\begin{proof}
For all $\vecn\in\mathcal{E}_{2}\backslash D_{1}$, for all $ j\in \{1,\ldots, d\}$, there exists $s\leq 17$, such that $\vecn+ s \vece{j} \in H_{0}$. Since $\vecn \neq \vecz$, there exists $j_{0}$ with $n_{j_{0}} >0$. 

Let
\begin{align*}
\MoveEqLeft \mathcal{V}   =
\big\{ \vecm(t), t_{0}=0, \exists t_{1}<\frac{1}{s}, \ldots, t_{s}<\frac{1}{s} \quad \text{such that}\\
&  \quad  \vecm(t) = \vecn + q\, \vece{j_{0}}, \forall t_{q}\leq  t< t_{q+1}, 0\leq q\leq s-1 \big\}.
\end{align*}
Let us compute the probability for the birth and death process to belong to $\mathcal{V}$. Note that by assumption
\[
KB_{j_{0}}\left(\frac{\vecn}{K}\right) = \sum_{\ell=1}^d n_{\ell}\, \partial_{x_{\ell}} B_{j_{0}}(0) \, + \,{\mathcal O}\left(\frac{1}{K}\right)
\]
and
\[
KD_{j_{0}}\left(\frac{\vecn}{K}\right) = \sum_{\ell=1}^d n_{\ell}\, \partial_{x_{\ell}} D_{j_{0}}(0)\,+ \,{\mathcal O}\left(\frac{1}{K}\right).
\]
Therefore, for $K$ large enough, the birth probability of an individual with type $j_{0}$ is bounded below by 
\begin{align*}
\MoveEqLeft[8]
 \inf_{ \vecn \in \mathcal{E}_{2}} \frac{KB_{j}(\frac{\vecn}{K}) }{K \sum_{\ell=1}^d B_{\ell}(\frac{\vecn}{K}) + K \sum_{\ell=1}^d D_{\ell}(\frac{\vecn}{K})} \\
 & > \frac12\, \inf_{  \vecn \in \mathcal{E}_{2}} \frac{\partial_{x_{j_{0}}} B_{j_{0}}(0)}{\sum_{\ell=1}^d n_{\ell}\, \partial_{x_{\ell}} B_{j_{0}}(0) \, +\,\sum_{\ell=1}^d n_{\ell}\, \partial_{x_{\ell}} D_{j_{0}}(0)}= \zeta ,
\end{align*}
and $\zeta >0$ by \eqref{cond:nais} and since $\max_{1\leq \ell\leq d} n_{\ell}\leq 17$ for
$\vecn \in \mathcal{E}_{2}$.
Note also that  the denominator (which is the jump rate)  is bounded below by $\zeta' =\inf_{j} \partial_{x_{j}} B_{j}(0)>0$ by \eqref{cond:nais}.
Therefore,
\[
\proba_{ \vecn}( N^K \in\mathcal{V}) \geq \zeta^s\Big(1 - \e^{-\zeta'/s}\Big)^s \geq \zeta^{17} \Big(1 - \e^{-\zeta'/17}\Big)^{17}.
\]
The results follows.
\end{proof}

In the following lemma we will partition more finely the disk $D_{1}$ to fit as well as possible the speed of decrease of the distance between the process and $\vecnf$.
\begin{sublemma}\label{descente4}
There exists two constants $\tcte{descente4}>0$ and
$\ceta{descente4}<1$ such that for $K$ large enough
\[
\inf_{\vecn\in D_{1}\backslash \Delta}\proba\big(T_{\Delta}\le
\tcte{descente4}\,\log K\big)\ge 1-\ceta{descente4},
\]
where $\Delta$ is defined in \eqref{def:Delta} and $D_1$ in \eqref{def-D1}.
\end{sublemma}

\begin{proof}
We start by defining a decreasing (finite) sequence of numbers
$(R_{j})$ as follows:
\begin{align*}
R_{-2} &=\|\vecnf\|- \pcte{tondeuse}, \, R_{-1}=\|\vecnf\|-(\pcte{tondeuse}+4), \\
R_{0} &=\|\vecnf\|-(\pcte{tondeuse}+8), \, R_{1}=\|\vecnf\|-(\pcte{tondeuse}+12).
\end{align*}
Define
\[
j_{*}=\inf\left\{j: R_1-2^{j-1}+1\leq \frac{1}{2}\,\inf_{\ell} n^{*}_\ell\right\}.
\]
Note that $j_*=\Oun \log K$.
For $2\leq j\leq j_{*}$ we define
\[
R_j=R_1-2^{j-1}+1.
\]
Note that for $1\leq j\leq j_{*}$, $R_{j}\geq R_{j^{*}}=\Oun K$. 
Define
\[
j_{**}=\sup\{j > j^{*}: R_{j^*} \, 2^{-(j-j^*)}> \crho{tondeuse}\sqrt{K}\}-1.
\]
Note that $j_{**}=\Oun \log K$.
For $j^{*}\leq j\leq j^{**}+1$, let
\[
R_j= R_{j^*}\, 2^{-(j-j^*)}.
\]
Note that $\crho{tondeuse}\sqrt{K} \leq R_{j^{**}-1}\leq 2\crho{tondeuse}\sqrt{K}$ and that for $j\leq j^*$,
$\mathcal{B}(\vecn^*,R_{j}) \subset \mathcal{B}(\vecz,\|n^*\|/2)^c$.
We now define a (finite) decreasing sequence of domains $(D_{j})$, where $-2\le j\le j_{**}+1$, by 
\[
D_{j}=\mathcal{B}(\vecnf,R_{j})\cap \domaine.
\]
We also define a finite sequence of annuli $(H_{j})$, where $-2\le j\le j_{**}$, by 
\[
H_j=D_{j}\backslash D_{j+1}.
\]
Recall that the Lyapunov function $\lyap$ has been defined in Theorem \ref{formule}. We  define the following sequences of positive numbers:\\
$(A_{j})_{-2\le j\le j_{**}}$ by
\[
A_{j}=\sup_{\vecn\in H_{j}}\lyap(\vecn)
\]
$(A'_{j})_{-1\le j\le j_{**}}$ by
\[
A'_{j}=\inf_{\vecn\in H_{j}\cup H_{j-1}}\lyap(\vecn)
\]
$(A''_{j})_{0\le j\le j_{**}}$ by
\[
A''_{j}=\inf_{\vecn\in H_{j-2}}\lyap(\vecn)
\]
$(\Lambda_j)_{0\leq j\leq j_{**}}$ by
\[
\Lambda_j=-\sup_{\vecn\in H_{j-2}\cup H_{j-1}\cup H_{j}}\frac{\gen\!_{\sK} \varphi(\vecn)}{\varphi(\vecn)}.
\]
$(\eta_j)_{0\leq j\leq j_{**}}$ by
\begin{equation}\label{eta}
\eta_j=\frac{2 A_j}{A''_{j}}
\end{equation}
$(t_j)_{0\leq j\leq j_{**}}$ by
\begin{equation}\label{et}
t_j=\frac{1}{\Lambda_j}\log \frac{A''_{j}}{A'_{j}}.
\end{equation}

It is left to the reader to check that there exists a constant $c>1$, independent of $j$ and $K$, such that
\begin{equation}\label{catondplus0}
c^{-1} \exp\left(\frac{\alpha R_j^2}{K}\right) \leq A_j \leq c \exp\left(\frac{\alpha R_j^2}{K}\right)
\end{equation}
\begin{equation}\label{catondplus1}
c^{-1} \exp\left(\frac{\alpha R_{j+1}^2}{K}\right) \leq A'_j \leq c \exp\left(\frac{\alpha R_{j+1}^2}{K}\right)
\end{equation}
\begin{equation}\label{catondplus2}
c^{-1} \exp\left(\frac{\alpha R_{j-1}^2}{K}\right) \leq A''_j \leq c \exp\left(\frac{\alpha R_{j-1}^2}{K}\right)\,.
\end{equation}
If $j\leq j_*$,  we have by Corollary \ref{tondeuse}
\begin{equation}\label{la0}
\Lambda_j \geq c^{-1} \,( \|\vecnf \| -R_{j-2})
\end{equation}
and if $j_*<j\leq j_{**}$ we have by Theorem \ref{formule}
\begin{equation}\label{la1}
\Lambda_j \geq \frac{R_{j+1}^{\,2}}{c K}.
\end{equation}
Let us introduce 
\[
t_{s}=\sum_{j=0}^{j_{**}}t_{j}\;, \text{ with } t_{j}>0.
\]
Using the Markov property and the monotonicity of $\proba_{\vecn}\big(T_{\Delta}\le t\big)$ as a function of $t$, we have, for all $0\le \ell\le j_{**}$ and for $\vecn\in D_{\ell}\backslash D_{\ell+1}$,
\[
\proba_{\vecn}\left(T_{\Delta}\le \sum_{j=\ell}^{j_{**}}t_{j}\right)
\ge \proba_{\vecn}\left(T_{\Delta}\le \sum_{j=\ell}^{j_{**}}t_{j}\;,\;
T_{D_{\ell+1}}\le t_{\ell} \right)
\]
\[
=\esperance_{\vecn}\left[\un_{\{T_{D_{\ell+1}}\le t_{\ell}\}}\;
\proba_{\vecN^{\sK}(T_{D_{\ell+1}})}\left(T_{\Delta}\le \sum_{j=\ell+1}^{j_{**}}t_{j}
+ t_{\ell}-T_{D_{\ell+1}}\right)\right]
\]
\[
\ge \proba_{\vecn}\big(T_{D_{\ell+1}}\le t_{\ell}\big)\;\;
\inf_{\vecn\in D_{\ell+1}\backslash D_{\ell+2}}\proba_{\vecn}
\left(T_{\Delta}\le \sum_{j=\ell+1}^{j_{**}}t_{j}\right).
\]
Using this estimate recursively together with $D_{j_{**}}\subset \Delta$ we obtain 
for all $\vecn\in D_{\ell}\backslash D_{\ell+1}$
\[
\proba_{\vecn}\left(T_{\Delta}\le \sum_{j=\ell}^{j_{**}}t_{j}\right)\ge
\prod_{j=\ell}^{j_{**}}\,\inf_{\vecn\in D_{j}\backslash D_{j+1}}
\proba_{\vecn}\big(T_{D_{j+1}}\le t_{j}\big).
\]
Therefore, from  the monotonicity of $t \mapsto \proba_{\vecn}\big(T_{\Delta}\le t\big)$
we have for all $\vecn\in D_{0}$
\[
\proba_{\vecn}\big(T_{\Delta}\le t_{s}\big)\ge
 \prod_{j=0}^{j_{**}}\,\inf_{\vecn\in D_{j}\backslash D_{j+1}}
\proba_{\vecn}\big(T_{D_{j+1}}\le t_{j}\big).
\]
We now derive a lower bound for each term in the product and an upper
bound for  each $t_{j}$, hence for $t_{s}$.

By elementary computations using the explicit form for $R_{j}$, \eqref{et},\eqref{catondplus0}, \eqref{catondplus1}  \eqref{catondplus2},  \eqref{la0}
and \eqref{la1}, we obtain that for $j=0$ to $j^{**} +1$, $t_{j}$ is of order $1$. Therefore
\[
\sum_{j=0}^{j^{**} +1} t_{j} = {\cal O}(\log K).
\]
One  can also check by considering \eqref{eta} that
\[
\sum_{j=0}^{j^{**} +1} \eta_{j} = \Oun.
\]
The result follows by applying Corollary \ref{t4d}.
\end{proof}

We can now prove Lemma \ref{entrec}. We give the proof for $\vecn\in \mathcal{E}_{1}$, the other cases are
similar and left to the reader. Using Sublemmas \ref{descente1}, \ref{descente2}, \ref{descente3},  \ref{descente4} and the Markov property we have, for all $K$ large enough and all $\vecn\in \mathcal{E}_{1}$,
\[
\proba_{\vecn}\big(T_{\Delta}\le (\tcte{descente1}+\tcte{descente2}
+\tcte{descente4})\,\log K\big)\ge \ceta{descente1}\,\ceta{descente2}\,
\ceta{descente4}.
\]
The result follows.

\subsection{Proof of Lemma \ref{restec}}

\begin{sublemma} \label{b}
Let 
\begin{equation}
\label{defD}
D  =  \mathcal{B} \Big(\vecnf, \frac{1}{2}\inf_{j}\vecnf_{j}\Big)\cap \domaine
\end{equation}
and define $\tilde \rho=\tilde \rho(K) = \frac{1}{K}\frac{\min(n^*_{\ell})}{2}$.

For $K$ large enough and  for all $\vecn\in \Delta$ and  $t\ge 0$,
\[
\proba_{\vecn}\big(t>T_{D^{c}}\big)\le 
(\Oun + \Oun t) \e^{-\alpha\tilde \rho^2 K}.
\]
\end{sublemma}

\begin{proof} Let $\varphi$ defined in \eqref{lyapou} and $t>0$.
We apply  Dynkin's Theorem  to $\varphi(\vecN^{\sK}(T_{D^c}\wedge t))$  (in the spirit of the proof of Lemma \ref{disque}). 
Using Theorem \ref{formule}, we obtain for $K$ large enough 
\[
\e^{\alpha \tilde \rho^2 K} \proba_{\vecn}(T_{D^c}<t) \le  \Oun + \Oun t
\]
and the result follows.
\end{proof}

\begin{proof}[Proof of Lemma \ref{restec}]
We will in fact prove a stronger  result with $t\leq K$ which will imply the result if $K$ is large enough. 

Let us define the ball $\widetilde{\mathcal{B}}= \widetilde{\mathcal{B}}({\vecn^*},\frac{2 \tilde \rho}{3} K)$. 
Let us consider the function 
\[
\psi(\vecn) = \|\vecn - \vecn^*\|^2 \ \un_{\{\vecn\in \widetilde{\mathcal{B}}\}}.
\]
Assuming $\vecn, \vecn+\vece{j}, \vecn-\vece{j} \in \widetilde{\mathcal{B}}$ and using  \eqref{gegene}, \eqref{lip}, \eqref{cond:vanish} and \eqref{cond:lyap},  we obtain
\begin{align*}
& \gen_{K} \psi(\vecn) \\
&=  K \sum_{j=1}^{d}\left[ 2 \Big(B_{j}\left(\frac{\vecn}{K}\right)-D_{j}\left(\frac{\vecn}{K}\right)\Big) (\vecn_{j}- \vecn_{j}^*) 
+ \Big(B_{j}\left(\frac{\vecn}{K}\right) +  D_{j}\left(\frac{\vecn}{K}\right)\Big)\right]\\
& \leq - \sigma \psi(\vecn) + \Oun \,K + \Oun \,K^3 \un_{\widetilde{\mathcal{B}}^c}
\end{align*}
for all $\vecn\in \widetilde{B}$, where $\sigma := \beta \tilde \rho$. From It\^o's formula and for $t>0$, we have
\[
\esperance_{\vecn}\left[\e^{\sigma(t\wedge T_{\tilde{\mathcal{B}}^c})}
\psi\big(\vecN^{\sK}(t\wedge T_{\widetilde{\mathcal{B}}^c})\big)\right]
\le \psi(\vecN^{\sK}(0)) + \Oun K\, \esperance_{\vecn}\left[\frac{\e^{\sigma(t\wedge T_{\widetilde{\mathcal{B}}^c})}-1}{\sigma}\right].
\]
Then
\[
\esperance_{\vecn}\left[\e^{\sigma t} \psi\big(\vecN^{\sK}(t)\big) \un_{\{T_{\widetilde{\mathcal{B}}^c}>t\}}\right] 
\le \psi(\vecn) + \Oun \, \frac{\e^{\sigma t}}{\sigma}\,K .
\]
On another hand, for $\forall t \leq K$,
\[
\esperance_{\vecn}\left[\psi\big(\vecN^{\sK}(t)\big) \un_{\{T_{\widetilde{\mathcal{B}}^c}<t\}}\right] 
\le  \Oun \,K^3 \e^{-\alpha \tilde \rho^2 K}.
\]
Therefore, for all $t \leq K$,
\[
\esperance_{\vecn}\left[\psi\big(\vecN^{\sK}(t)\big)\right] 
\le 4 \crho{tondeuse}^2 K+  \frac{\Oun}{\sigma} K  + \Oun K^3 \e^{-\alpha \tilde \rho^2 K}.
\]
If  $\rho' = \sqrt{4 \crho{tondeuse}^2 K+  \frac{\Oun}{\sigma} + 1}$, we deduce that
\begin{align*}
& \proba_{\vecn}\big(\|\vecN^{\sK}(t)-\vecn^*\|> \rho' \sqrt{K}\,\big)\nonumber\\
&=
\proba_{\vecn}\big(\|\vecN^{\sK}(t)-\vecn^*\|> \rho' \sqrt{K}, T_{\widetilde{\mathcal{B}}^c}<t\big) 
+  \proba_{\vecn}\big(\|\vecN^{\sK}(t)-\vecn^*\|> \rho' \sqrt{K}, T_{\widetilde{\mathcal{B}}^c}>t\big)\nonumber   \\
& \le  K \e^{-\alpha \tilde \rho^2 K} +\frac{1}{(\rho')^2 K}\Big(4 \crho{tondeuse}^2 K+
\frac{ \Oun K}{ \sigma}\Big) \leq 1.
\end{align*}
Then there are positive constants $\cte{delta-prime}$ and $\ceta{delta-prime}$ such that
\begin{equation}
\label{delta-prime}
\inf_{\vecn\in\Delta}  \inf_{0\leq t\leq \cte{delta-prime} K} \proba_{\vecn}\big(\vecN^{\sK}(t) \in \Delta'\big) \geq 1-\ceta{delta-prime}
\end{equation}
for $K$ large enough and Lemma \ref{restec} follows.
\end{proof}

\subsection{Proof of Lemma \ref{uni}}

We first introduce some notations.

\begin{align}
\label{def-petit-lambda}
\lambda(\vecx) &=\sum_{j=1}^{d}\big[B_{j}(\vecx)
+D_{j}(\vecx)\big]\quad\textup{and}\quad 
\lambda_{*}=\lambda\big(\vecxf)\,,\\
\label{def-grand-lambda}
\Lambda\big(\vecn\big) & =\Lambda\big(\vecn,K\big)=
K\lambda\big(\vecn/K\big)
\quad\text{and}\quad
\Lambda_{*}=K\lambda_{*}\,,\\
\label{def-rw}
p\big(\vecn,\vecn+\vece{j}\big) &=
\frac{B_{j}\big(\vecn/K\big)}{\lambda\big(\vecn/K\big)}
\,,\; 
p\big(\vecn,\vecn-\vece{j}\big)=
\frac{D_{j}\big(\vecn/K\big)}{\lambda(\vecn/K\big)}\,.
\end{align}

We will compare the  process  $\vecN^{\sK}$ in $\Delta'$  to the birth-and-death process
$(\widetilde  \vecN, t\geq 0)$ defined by
\begin{align*}
& \proba\big(\,\widetilde\vecN(t+\dd t) = \vecn + \vece{j} \,|\,\tilde  \vecN(t)=\vecn\big) = K\,B_{j}(\vecxf)\, \dd t \,,\\
& \proba\big(\,\widetilde  \vecN(t+\dd t) = \vecn - \vece{j} \,|\,\tilde  \vecN(t)=\vecn\big) = K\,D_{j}(\vecxf)\, \dd t.
\end{align*}
The embedded chain will be  the symmetric random walk $(\plat_{\ell})_{\ell\in \integers}$ with state space  $\zentiers^d$ and transition matrix $p^{*}$ defined by 
\begin{equation}\label{rj}
r_j=p^{*}\big(\vecn,\vecn+\vece{j}\big)=\frac{B_{j}(\vecxf)}{\lambda(\vecxf)}
=p^{*}\big(\vecn,\vecn-\vece{j}\big)=\frac{D_{j}(\vecxf)}{\lambda(\vecxf)}.
\end{equation}
To prove Lemma \ref{uni}, we need to obtain a lower bound for  $\proba_{\vecn}\big(N^{\sK}(1)=\vecm\big)$ with $\vecn\in \Delta'$ and $\vecm \in \Delta$.
We have 
\begin{align*}
\MoveEqLeft \proba_{\vecn}\big(N^{\sK}(1)=\vecm\big) =\sum_{q}
\sum_{\substack{\vg:\\ \vg(0)=\vecn,\;\vg(q)=\vecm}}
\int\cdots\int_{t_{0}+\cdots +t_{q-1}<1} \\
&
\prod_{\ell=0}^{q-1}p\big(\vg(\ell),\vg(\ell+1)\big)
\prod_{\ell=0}^{q-1}\Lambda\big(\vg(\ell)\big)
\prod_{\ell=0}^{q-1}\e^{-t_{\ell}\Lambda(\vg(\ell))}
\e^{-\Lambda\left(\vecm\right) \big(1-\sum_{\ell=0}^{q-1}t_{\ell}\big)}
\prod_{\ell=0}^{q-1}\dd t_{\ell}.
\end{align*}
We  restrict our attention to the paths whose number of jumps between $0$ and $1$ belongs to
$[\, \Lambda_{*}-\sqrt K,\Lambda_{*}+\sqrt K\, ]$ (and is then of order $K$) and whose values belong to
$\mathcal{B}(\vecnf, \sqrt{K}\log K)$.
Moreover, we make a change of law and write a kind of Girsanov formula with respect to the law of $\tilde  \vecN$. We obtain
\begin{align*}
\MoveEqLeft  \proba_{\vecn}\big(N^{\sK}(1)=\vecm\big)\\
& \ge \sum_{q\in[\Lambda_{*}-\sqrt K,\Lambda_{*}+\sqrt K]}\quad
\sum_{\substack{\vg\,:\, \vg(0)=\vecn,\;\vg(q)=\vecm\\ \sup_{1\le\ell\le q-1}|\vg(\ell)-\vecnf|
\le \sqrt K\log K}} \int\cdots\int_{t_{0}+\cdots+t_{q-1}<1}\\
&
\prod_{\ell=0}^{q-1}p(\vg(\ell),\vg(\ell+1))
\prod_{\ell=0}^{q-1}\Lambda(\vg(\ell))
\prod_{\ell=0}^{q-1}\e^{-t_{\ell}\,\Lambda(\vg(\ell))}
\e^{-\Lambda(\vecm)\,\big(1-\sum_{\ell=0}^{q-1}t_{\ell}\big)}
\prod_{\ell=0}^{q-1}\dd t_{\ell}\\
&
= \sum_{q\in[\Lambda_{*}-\sqrt K,\Lambda_{*}+\sqrt K]}
\Lambda_{*}^{q}\e^{-\Lambda_{*}}\!\!\!\!\!\!
\sum_{\substack{\vg\,:\,\vg(0)=\vecn,\;\vg(q)=\vecm\\ \sup_{1\le\ell\le  q-1}|\vg(\ell)-\vecnf|\le \sqrt K\log K}}
\!\!\!\!\int\cdots\int_{t_{0}+\cdots+t_{q-1}<1}\\
&
\qquad \prod_{\ell=0}^{q-1}p^{*}(\vg(\ell),\vg(\ell+1))\;
\Pi_{q}\big(\vg,t_{0},\ldots,t_{q-1}\big)\;
\prod_{\ell=0}^{q-1}\dd t_{\ell}
\end{align*}
with
\[
\Pi_{q}=\Pi_{q}^1\, \Pi_{q}^2 \, \Pi_{q}^3
\]
where
\[
\begin{aligned}
\Pi_{q}^1\big(\vg,t_{0},\ldots,t_{q-1}\big) &=
\prod_{\ell=0}^{q-1} \frac{p\big(\vg(\ell),\vg(\ell+1)\big)}{p^{*}\big(\vg(\ell),\vg(\ell+1)\big)}\\
\Pi_{q}^2\big(\vg,t_{0},\ldots,t_{q-1}\big) &=
\;\prod_{\ell=0}^{q-1}\frac{\Lambda\big(\vg(\ell)\big)}{\Lambda_{*}}
\;\prod_{\ell=0}^{q-1}\e^{-t_{\ell}\,\big[\Lambda \big(\vg(\ell)\big)-\Lambda_{*}\big]}
\end{aligned}
\]
and
\[
\Pi_{q}^3\big(\vg,t_{0},\ldots,t_{q-1}\big)=
\e^{-\big[\Lambda\big(\vecm\big)-\Lambda_{*}\big]\,\big(1-\sum_{\ell=0}^{q-1}t_{\ell}\big)}.
\]
The structure of the previous expression is as follows:
\begin{equation}\label{eureka}
\sum_{q\in[\, \Lambda_{*}-\sqrt K,\Lambda_{*}+\sqrt K\, ]}\quad \esperance^*_{\vecn\to\vecm,q}(\Pi_{q}(\tilde  \vecN) \un_{C_{q}}),
\end{equation}
where $C_{q}$ describes the restriction of the $q$ states of the process to  $\mathcal{B}(\vecnf, \sqrt{K}\log K)$ and $\Pi_{q}=  \Pi_{q}^1\,\Pi_{q}^2\,\Pi_{q}^3$ and $\esperance^*_{\vecn\to\vecm,q}$ denotes the expectation related to the law of the process $\tilde \vecN$ going from $\vecn$ to $\vecm$ in $q$ jumps. 
 
Equation  \eqref{eureka} writes
\[
\sum_{q\in[\,\Lambda_{*}-\sqrt K,\Lambda_{*}+\sqrt K\,]}\quad \esperance^*_{\vecn\to\vecm,q}(\un_{C_{q}})\, \frac{\esperance^*_{\vecn\to\vecm,q}(\Pi_{q}(\tilde  \vecN) \,\un_{C_{q}})}{\esperance^*_{\vecn\to\vecm,q}(\un_{C_{q}})}.
\]
To get a lower bound of this expression, we use Jensen's inequality and obtain
\begin{align*} 
\MoveEqLeft \sum_{q\in[\, \Lambda_{*}-\sqrt K,\Lambda_{*}+\sqrt K\, ]}\quad \esperance^*_{\vecn\to\vecm,q}(\Pi_{q}(\tilde  \vecN) \,\un_{C_{q}})\ge\\
& \sum_{q\in[\,\Lambda_{*}-\sqrt K,\Lambda_{*}+\sqrt K\,]}\quad \esperance^*_{\vecn\to\vecm,q}(\un_{C_{q}})\, \exp\left(\frac{\esperance^*_{\vecn\to\vecm,q}(\log \Pi_{q}(\tilde  \vecN) \,\un_{C_{q}})}{\esperance^*_{\vecn\to\vecm,q}(\un_{C_{q}})}\right).
\end{align*}
 
Replacing each term with its complete expression, we obtain
\begin{align}
\label{Jen}
\MoveEqLeft[8] \proba_{\vecn}\big(N^{\sK}(1)=\vecm\big)\ge
\sum_{q\in[\Lambda_{*}-\sqrt K,\,\Lambda_{*}+\sqrt K]}
 \Bigg(  \e^{-\Lambda_{*}}\;
\frac{\Lambda_{*}^{q}}{q!}\;\e^{S_{q}(\vecn,\vecm)}\; \times \nonumber\\
& \quad 
\sum_{\substack{\vg\, :\, \vg(0)=\vecn,\;\vg(q)=\vecm\\ \sup_{1\le\ell\le q-1}|\vg(\ell)-\vecnf|\le \sqrt K\,\log K}}
\prod_{\ell=0}^{q-1}p^{*}\big(\vg(\ell),\vg(\ell+1)\big)\Bigg)
\end{align}
with
\[
S_{q}(\vecn,\vecm)=\frac{1}{Z_{q}(\vecn,\vecm)}\;\e^{-\Lambda_{*}}\;{\Lambda_{*}}^{q}
\sum_{\substack{\vg\, :\, \vg(0)=\vecn,\;\vg(q)=\vecm\\ 
\sup_{1\le\ell\le q-1}|\vg(\ell)-\vecnf|\le \sqrt K \log K}}
\int\cdots\int_{t_{0}+\cdots+t_{q-1}<1}\;
\]
\[
\prod_{\ell=0}^{q-1}p^{*}\big(\vg(\ell),\vg(\ell+1)\big)
\;\log\Pi_{q}\big(\vg,t_{0},\ldots,t_{q-1}\big)\;
\prod_{\ell=0}^{q-1}dt_{\ell}
\]
and
\[
Z_{q}(\vecn,\vecm)=\e^{-\Lambda_{*}}\;
\frac{\Lambda_{*}^{q}}{q!}\;
\sum_{\substack{\vg,\;\vg(0)=\vecn,\;\vg(q)=\vecm\\
\sup_{1\le\ell\le q-1}|\vg(\ell)-\vecnf|\le \sqrt K\log K}}
\prod_{\ell=0}^{q-1}p^{*}\big(\vg(\ell),\vg(\ell+1)\big). 
\]
Our aim is now to give a lower bound for the right-hand side term in \eqref{Jen}. It  will be deduced from the three next lemmas which show that
\[
\esperance^*_{\vecn\to\vecm,q}(\un_{C_{q}})= Z_{q}(\vecn,\vecm) \sim 1/K^{(d+1)/2},\; \text{as}\;K\to \infty
\]
and
\[
\frac{\esperance^*_{\vecn\to\vecm,q}(\log \Pi_{q}(\tilde  \vecN) \,\un_{C_{q}})}{\esperance^*_{\vecn\to\vecm,q}(\un_{C_{q}})}\, = S_{q}(\vecn,\vecm)
\]
is of order one uniformly in $q$.  
 
Let us first estimate $Z_{q}(\vecn,\vecm)$.
\begin{lemma}\label{Zq}
We have the following estimates.
\begin{itemize}
\item[\textup{(i)}] There exists a constant $\cte{Zq}>1$ independent of $K$ such that for
$K$ large enough, for all $q\in\big[\Lambda_{*}-\sqrt K,\Lambda_{*}+\sqrt K \big]$
and for all $\vecm\in\Delta$, $\vecn \in \Delta'$
\[
\cte{Zq}^{-1}\;K^{-d/2-1/2}\le Z_{q}(\vecn,\vecm)\le\cte{Zq}\;K^{-d/2-1/2}.
\]
\item[\textup{(ii)}]
There exists a constant $\cte{Zq}'>0$ independent of $K$ such that for
$K$ large enough
\begin{align}
\label{nouvelan}
\MoveEqLeft[4]
\inf_{\substack{\vecn\in\Delta' \\ \vecm\in\Delta}}
\quad
\sum_{q\in[\Lambda_{*}-\sqrt K,\,\Lambda_{*}+\sqrt K]} \;\e^{-\Lambda_{*}}\;
\frac{\Lambda_{*}^{q}}{q!}\quad\times\nonumber \\
&\!\!\!\!\!
\sum_{\substack{\vg,\;\vg(0)=\vecn,\;\vg(q)=\vecm\\ \sup_{1\le\ell\le q-1}|\vg(\ell)-\vecnf|\le \sqrt K\,\log K}}
\prod_{\ell=0}^{q-1}p^{*}\big(\vg(\ell),\vg(\ell+1)\big)
\ge \cte{Zq}'\;K^{-d/2}\,.
\end{align}
\end{itemize}
\end{lemma}

\begin{proof}
(i) Note first, using Stirling's formula, that 
 for $K$ large enough
\[
\frac{1}{2\sqrt{\Lambda_{*}}}\, \e^{-\gamma}
\leq
\e^{-\Lambda_{*}}\frac{\Lambda_{*}^{q}}{q!}
\leq 
\frac{2}{\sqrt{\Lambda_{*}}}\, \e^{-\gamma}
\]
where $\gamma$ is Euler's constant.  Then $\e^{-\Lambda_{*}} \frac{\Lambda_{*}^{q}}{q!}$ is of order $1/\sqrt{K}$. Now, we note that  for all $K$ large enough and all $q\in[\Lambda_{*}-\sqrt K,\Lambda_{*}+\sqrt K]$, we have
\begin{align*}
\MoveEqLeft[4]
\sup_{\substack{\vecn\in\Delta'\\\vecm\in\Delta}}
\sum_{\substack{\vg\, :\,\vg(0)=\vecn,\;\vg(q)=\vecm\\ \sup_{1\le\ell\le q-1}|\vg(\ell)-\vecnf|>\sqrt K\,\log K}}
\prod_{\ell=0}^{q-1}p^{*}\big(\vg(\ell),\vg(\ell+1)\big)\\
& \leq
\sup_{\vecn\in\Delta'}\; \sum_{\ell=1}^{q-1} \proba_{\vecn} \left( |\plat_\ell-\vecnf|>\sqrt{K}\log K\right)\,.
\end{align*}
Applying Hoeffding's inequality to the random walk $(\plat_{\ell})_{\ell\in \integers}$ we get  
\begin{equation}
\label{coupe}
\sup_{\substack{\vecn\in\Delta'\\ \vecm\in\Delta}}
\sum_{\substack{\vg\,:\, \vg(0)=\vecn,\;\vg(q)=\vecm\\ \sup_{1\le\ell\le q-1}|\vg(\ell)-\vecnf|>\sqrt K\,\log K}}
\prod_{\ell=0}^{q-1}p^{*}\big(\vg(\ell),\vg(\ell+1)\big)
\le \e^{-\Oun (\log K)^2}\,.
\end{equation}
We deduce that
\[
\left|Z_{q}(\vecn,\vecm)-\e^{-\Lambda_{*}}\;
\frac{\Lambda_{*}^{q}}{q!}\;
\sum_{\vg,\;\vg(0)=\vecn,\;\vg(q)=\vecm}
\prod_{\ell=0}^{q-1}p^{*}\big(\vg(\ell),\vg(\ell+1)\big)\right|\leq \e^{-\Oun (\log K)^2}.
\]
To finish the proof, we apply the local  limit theorem \cite[Chapter 3]{Durrett} to the random walk $(\plat_{\ell})_{\ell\in \integers}$.
Statement  (ii)  immediately follows at once from (i).
\end{proof}

\begin{lemma}\label{boulot}
There exists a constant $\cte{boulot}>0$ independent of $K$ such that for
$K$ large enough
\[
\sup_{q\in[\Lambda_{*}-\sqrt K,\Lambda_{*}+\sqrt K]}\big|S_{q}(\vecn,\vecm)\big|\le \cte{boulot}.
\]
\end{lemma}
\begin{proof}
Observe that
\[
\big|S_{q}(\vecn,\vecm)\big|\le K^{d/2} \;q! \;\Oun\quad\times
\]
\[
\left|\;\sum_{\substack{\vg,\;\vg(0)=\vecn,\;\vg(q)=\vecm\\ \sup_{1\le\ell\le q-1}|\vg(\ell)-\vecnf|\le \sqrt K\,\log K}}
\int\cdots\int_{t_{0}+\cdots+t_{q-1}<1}\;\right.
\]
\[
\qquad\qquad \left.\prod_{\ell=0}^{q-1}p^{*}\big(\vg(\ell),\vg(\ell+1)\big)
\;\log\Pi_{q}\big(\vg,t_{0},\ldots,t_{q-1}\big)\;
\prod_{\ell=0}^{q-1}dt_{\ell}
\vphantom{\sum_{\substack{\vg,\;\vg(0)=\vecn,\;\vg(q)=\vecm\\ \sup_{1\le\ell\le q-1}|\vg(\ell)-\vecnf|\le \sqrt K\,\log K}}\int\cdots\int_{t_{0}+\cdots+t_{q-1}<1}}\;\right|.
\]
Since $\Pi_{q} = \Pi_{q}^1\,\Pi_{q}^2\,\Pi_{q}^3$, 
\[
\log\Pi_{q}\big(\vg,t_{0},\ldots,t_{q-1}\big) = \sum_{i=1}^3 \log\Pi_{q}^i \big(\vg,t_{0},\ldots,t_{q-1}\big)
\]
and we have to estimate separately  the three terms.
The result  follows from several  technical lemmas which are postponed to Section \ref{sec:pleindelemmes}.
\end{proof}

It follows from  \eqref{Jen} and  \eqref{nouvelan} and Lemma \ref{boulot}  that there exists $\ceta{uni}<1$ such that 
\[
\proba_{\vecn}\big(N^{\sK}(1)=\vecm\big) \geq (1-\ceta{uni})\,\nu(\{m\}),
\]
where $\nu$ is the  measure  defined in Subsection \ref{a1}.

\subsection{Proof of Condition \eqref{cond-A2}}

Our aim is to show the existence of a constant $c_{2}$ such that for all $t\geq 0$ and $\vecn \in \Delta$, $\vecm\in \Delta^c$
\begin{equation}\label{eq-harnack}
\proba_{\vecn}(T_{\vecz}>t)\geq c_{2}\, \proba_{\vecm}(T_{\vecz}>t)
\end{equation}
where $\Delta$ is defined in \eqref{def:Delta}. For all $t\ge 0$ let 
\[
g(t)=\sup_{\vecn\in \Delta} \proba_{\vecn}\big(T_{\vecz}>t\big)
\quad\text{and}\quad
f(t)=\sup_{\vecn\in \Delta^{c}} \proba_{\vecn}\big(T_{\vecz}>t\big).
\]
The proof of Condition \eqref{cond-A2} will be the consequence of the four following lemmas which we prove hereafter.

\begin{lemma}\label{regg}
There exist $\eta>1$ and $\delta>0$ such that
\[
\eta\, \delta <1 \ \ (\hbox{whence}\, \delta < 1),
\] 
and, for all $K$ large enough, there exists $t_*=t_*(K)$ such that
\[
\forall  t\ge t_{*}, \;g(t-t_{*})\le \eta g(t)
\]
and 
\[
\sup_{\vecn\in \Delta^{c}}\proba_{\vecn}\big(T_{\Delta}\wedge T_{\vecz}>t_{*}\big)<\delta.
\]
\end{lemma}

\begin{proof}
The proof consists in several steps.

We first show that there exists a constant $\ceta{tcto} \in (0,1)$ such that for all $K$
large enough and  $\ttc=\cte{entrec} \log K$,
\begin{equation}
\label{tcto}
\sup_{\vecn\in\domaine}\proba_{\vecn}\big(T_{\Delta}\wedge T_{\vecz}>\ttc\big) \le 1-\ceta{tcto}.
\end{equation}
We have
\begin{align*}
\proba_{\vecn}\big(T_{\Delta}>\ttc\,,\, T_{\vecz}>\ttc\big)
&=\proba_{\vecn}\big(T_{\vecz}>\ttc\big)-\proba_{\vecn}\big(T_{\Delta}\le \ttc\,,\, T_{\vecz}>\ttc\big)\\
& \le 1-\proba_{\vecn}\big(T_{\Delta}\le \ttc\,,\, T_{\vecz}>\ttc\big)\,.
\end{align*}
We also have, using the Markov property, the monotonicity of
$\proba_{\vecn}\big(T_{\vecz}>t\big)$ and Sublemma \ref{b}
\begin{align*}
\proba_{\vecn}\big(T_{\Delta}\le \ttc\,,\, T_{\vecz}>\ttc\big)
&= \esperance_{\vecn}\left[\un_{\{T_{\Delta}\le \ttc\}} \,
      \proba_{\vecN^{\sK}(T_{\Delta})}\big(T_{\vecz}>\ttc-T_{\Delta}\big)\right]\\
& \geq \esperance_{\vecn}\left[\un_{\{T_{\Delta}\le \ttc\}} \,
          \proba_{\vecN^{\sK}(T_{\Delta})}\big(T_{\vecz}>\ttc\big)\right]\\
          & \geq \proba_{\vecn}\big(T_{\Delta}\le \ttc\big)\;\,\inf_{\vecm\in \Delta}\proba_{\vecm}\big(T_{D^{c}}> \ttc \big)\\
& \ge \proba_{\vecn}\big(T_{\Delta}\le \ttc\big)\;
  \big(1-     (\Oun + \Oun \ttc) \e^{-\alpha\,\tilde \rho^2 K}\big)\;,
\end{align*}
and the result follows for $K$ large enough using Lemma \ref{entrec}.

Let us now prove recursively that  for $K$ large enough and for all integer $q$,
\begin{equation}\label{ntcto}
\sup_{\vecn\in\domaine}\proba_{\vecn}\big(T_{\Delta}\wedge T_{\vecz}>q\,\ttc\big) \le (1-\ceta{tcto})^{q}.
\end{equation}
 The inequality is true for $q=1$. For $q>1$, we
can write using \eqref{tcto}
\begin{align*}
\proba_{\vecn}\big(T_{\Delta}\wedge T_{\vecz}>q\,\ttc\big)
& =
\esperance_{\vecn}\left[\un_{\{T_{\Delta}\wedge T_{\vecz}>(q-1)\,\ttc\}}\; 
\proba_{\vecN^{\sK}((q-1)\ttc)}\big(T_{\Delta}\wedge T_{\vecz}>\ttc\big)\right]\\
& \le \esperance_{\vecn}\left[\un_{\{T_{\Delta}\wedge T_{\vecz}>(q-1)\,\ttc\}}\right] \;
(1-\ceta{tcto})\\
& =(1-\ceta{tcto}) \, \proba_{\vecn}\big(T_{\Delta}\wedge T_{\vecz}>(q-1)\,\ttc\big)
\end{align*}
and the result follows.

By a similar proof as for \eqref{tcto},  (use first Lemma \ref{entrec} and after the first entrance in $\Delta$ 
use the Markov property and Sublemma \ref{b}), we also obtain that 
there exists a number $\cte{Tg}>0$ such that for all $K$ large enough
\begin{equation} \label{Tg}
\inf_{\vecn\in\domaine}\proba_{\vecn}\big(T_{\vecz}>K\big)\ge \cte{Tg}.
\end{equation}
It implies that for all $s\ge0$, for all $\vecn\in \Delta$ and for $K$ large enough,
\begin{align*}
\proba_{\vecn}\big(T_{\vecz}>s+K)
&= \esperance_{\vecn}\left[\un_{\{T_{\vecz}>s\}}\,
\proba_{\vecN^{\sK}(s)}\big(T_{\vecz}>K\big)\right]\\
&  \ge  \proba_{\vecn}\big(T_{\vecz}>s\big)\,\inf_{\vecn\in\domaine}\proba_{\vecn}\big(T_{\vecz}>K\big)\\
& \ge \cte{Tg}\, \proba_{\vecn}\big(T_{\vecz}>s\big).
\end{align*}
Then, for all $s\geq 0$ and all $\vecn \in \Delta$, 
we have $g(s+K)\geq \cte{Tg}\, \proba_{\vecn}\big(T_{\vecz}>s\big)$
and $g(s)\le \cte{reregg}\;g(s+K)$ with $\cte{reregg}=\cte{Tg}^{-1}>0$.
We have thus proved that for all $K$ large enough and all $t\ge K$
\begin{equation}\label{reregg}
g(t-K)\le \cte{reregg}\;g(t).
\end{equation}
Note that $\cte{reregg}$ is necessarily strictly greater than $1$. \\
  Let us now take $t_*=K$ and let $q_*$ be the smallest positive integer such that
$(1-\ceta{tcto})^{q}(\cte{reregg})<\frac12$. We take $\eta=\cte{reregg}$
and $\delta=(1-\ceta{tcto})^{q_*}$. We of course have $\eta\delta<1$.
Inequality \eqref{reregg} implies $g(t-t_*)\leq \eta g(t)$ for all $t\geq t_*$.
Moreover, since for $K$ large enough $q_* t_\Delta<K$ and by \eqref{ntcto}, we have
\[
\sup_{\vecn\in \Delta^{c}}\proba_{\vecn}\big(T_{\Delta}\wedge T_{\vecz}>t_{*}\big)
\leq \sup_{\vecn\in \Delta^{c}}\proba_{\vecn}\big(T_{\Delta}\wedge T_{\vecz}>q_* t_\Delta\big)
<(1-\ceta{tcto})^{q_*}=\delta.
\]
The lemma is proved.
\end{proof}

\begin{lemma}\label{fg}
With $t_{*}$, $\eta$  and $\delta>0$  defined in Lemma \ref{regg}, we get 
for all integer $n$
\[
f(n\,t_{*})\le \left(1+\frac{\eta}{1-\delta\,\eta}\right)\;g(n\,t_{*}).
\]
\end{lemma}
\begin{proof}
For all $\vecm\in \Delta^{c}$ and $t\ge t_{*}$ we have using the Markov property
\begin{align*}
& \proba\!_{\vecm}\big(T_{\vecz}>t\big)\\
& =
\proba\!_{\vecm}\big(T_{\vecz}>t,T_{\Delta}\le  t_{*}\big)+
\proba\!_{\vecm}\big(T_{\vecz}>t,T_{\Delta}> t_{*}\big)\\
&= \esperance_{\vecm}\left[\un_{\{T_{\Delta}\le t_{*}\}}\proba\!_{\vecN^{\sK}(T_{\Delta})}\big(T_{\vecz}>t-T_{\Delta}\big)\right]+\esperance_{\vecm}\Big[\un_{\{T_{\Delta}\wedge T_{\vecz}>t_{*}\}}\proba\!_{\vecN^{\sK}(t_{*})}
\big(T_{\vecz}>t-t_{*}\big)\Big]\\
& \le g(t-t_{*})+\delta f(t-t_{*})\le \eta \,g(t)+\delta f(t-t_{*})\, , 
\end{align*}
where we have used Lemma \ref{regg}.
This implies for all $n\ge 0$ 
\[
f(n\,t_{*}) \le \eta \,g(n\,t_{*}) +\delta\, f((n-1)\,t_{*}).
\]
It is easy to verify recursively that this implies 
\[
f(n\,t_{*})\le\, \delta^{n}+ \frac{\eta}{1-\delta\,\eta}\;g(n\,t_{*}).
\]
The result follows by observing that from Lemma \ref{regg} we have 
$g(n\,t_{*})\ge \eta^{-n}$ for all integers $n$.
\end{proof}

\begin{lemma}
With notations of Lemma \ref{regg} and Lemma \ref{fg}, for all $t>0$
\[
f(t)\le \eta\,\left(1+\frac{\eta}{1-\delta\,\eta}\right)\;g(t).
\]
\end{lemma}
\begin{proof}
We first consider the case $t>t_{*}$. Let $n=[t/t_{*}]$. We have by Lemma \ref{fg}, the monotonicity of
$f$, Lemma \ref{regg} and the monotonicity of $g$
\[
f(t)\le f(n\,t_{*})\le
\left(1+\frac{\eta}{1-\delta\,\eta}\right)\;g(n\,t_{*})
\le \eta\,\left(1+\frac{\eta}{1-\delta\,\eta}\right)\;g((n+1)\,t_{*})
\]
\[
\le \eta\,\left(1+\frac{\eta}{1-\delta\,\eta}\right)\;g(t).
\]
For $0\le t\le t_{*}$ we have by Lemma \ref{regg} and using the monotonicity
of $g(t)$
\[
f(t)\le 1\le \eta\,g(t_{*})\le \eta\;g(t).
\]
\end{proof}

\begin{lemma}\label{Harnack}
There exists a constant $0<\cte{Harnack}<1$, such that for all $K$ large
enough and all $t>0$ we have
\[
\inf_{\vecn\in \Delta}\proba_{\vecn}\big(T_{\vecz}>t\big)\ge
\cte{Harnack}\sup_{\vecn\in \Delta}\proba_{\vecn}\big(T_{\vecz}>t\big).
\]
\end{lemma}

\begin{proof}
Let $\alcte{partout}=1+\cte{entrec}$. With use of Lemma \ref{entrec},  \eqref{delta-prime}, Lemma \ref{uni}, Sublemma \ref{b} and twice the Markov 
property, we obtain 
 that there exists  $\cte{partout}>0$ such that for all $K$ large enough 
\begin{equation}\label{partout}
\sup_{\vecm\in \Delta}\,\sup_{\vecn\in D} \proba_{\vecn}\big( \vecN^{\sK}(\alcte{partout}\log K)=\vecm\big)\ge \frac{\cte{partout}}{K^{d/2}}.
\end{equation}
Indeed, for $\vecn\in D$ and  $\alpha'<\alcte{partout}$, we have
\begin{align*}
\MoveEqLeft \proba_{\vecn}\big( \vecN^{\sK}(\alcte{partout}\log K)=\vecm\big)\ge \;\proba_{\vecn}\big(T_{\Delta<\alpha'\log K}\,;\, \vecN^{\sK}(\alcte{partout}\log K)=\vecm\big)\\
&\ge    \esperance_{\vecn}\left[\un_{\{T_{\Delta}<\alpha'\log K\}}\,
\proba_{\vecN^{\sK}(s)}\big( \vecN^{\sK}(\alcte{partout}\log K - T_{\Delta})=\vecm \big)\right]\\
&\ge \Oun\,\inf_{\vecm\in \Delta} \proba_{\vecm}\big( \vecN^{\sK}(\alcte{partout}\log K - T_{\Delta} -1)\in \Delta\,;\,\vecN^{\sK}(\alcte{partout}\log K - T_{\Delta})=\vecm\big)
\end{align*}
and \eqref{partout} follows. 

We have for all $\vecn\in \Delta$ and all $\vecm\in \Delta$
\begin{align*}
& \proba_{\vecn}\big(T_{\vecm}>\alcte{partout}\, \cte{partout} K^{d/2}\log K)\\
& \le 
\proba_{\vecn}\big(T_{\vecm}>\alcte{partout}\, \cte{partout}K^{d/2}\;\log
K\,,\, T_{D^{c}}>\alcte{partout}\,\cte{partout} K^{d/2}\log K\big)\\
&  \quad + \proba_{\vecn}\big(T_{D^{c}}<\alcte{partout}\, \cte{partout} K^{d/2}\log K\big)\\
& \le \proba_{\vecn}\left(\bigcap_{q=1}^{\cte{partout}\,K^{d/2}}
\bigg\{\vecN^{\sK}(q\alcte{partout}\log K)\in D\,,\vecN^{\sK}(q\alcte{partout}\log K)\neq\vecm \bigg\}\right)\\
& \quad +\proba_{\vecn}\big(T_{D^{c}}<\alcte{partout}\, \cte{partout} K^{d/2}\log K\big)
\end{align*}
and using the Markov property and  \eqref{partout}, we obtain

\[
\le \left(1-\frac{\cte{partout}}{K^{d/2}}\right)^{\cte{partout}\,K^{d/2}}+\,
\proba\!_{\vecn}\big(T_{D^{c}}<\alcte{partout}\, \cte{partout} K^{d/2} \log K\big).
\]

Using Sublemma \ref{b}, we deduce that there exists $\ceta{vapartout}>0$
such that 
\begin{equation}\label{vapartout}
\sup_{\vecm\in \Delta}\, \sup_{\vecn\in \Delta}
\proba_{\vecn}\big(T_{\vecm}>\alcte{partout}\, \cte{partout} K^{d/2}\log K\big)
\le 1- \ceta{vapartout}.
\end{equation}

For all $t>0$, let us now  define
\[
\vecnt=\mathrm{argmax}_{\vecm\in \Delta}\proba_{\vecm}\big(T_{\vecz}>t\big).
\]
For $t>\alcte{partout}\,\cte{partout}\, K^{d/2}\log K$, we have for all $\vecn\in \Delta$
\begin{align*}
\proba_{\vecn}\big(T_{\vecz}>t\big)
& \ge \esperance_{\vecn}\left[\un_{\{T_{\vecnt}<\alcte{partout}\,\cte{partout}\, K^{d/2}\,\log K\}}\,
\proba_{\vecN^{\sK}({T_{\vecnt}})}\big(T_{\vecz}>t-T_{\vecnt}\big)\right]\\
& \ge \esperance_{\vecn}\left[\un_{\{T_{\vecnt}<\alcte{partout}\,\cte{partout}\, K^{d/2}\,\log K\}}\,
\proba_{\vecN^{\sK}({T_{\vecnt}})}\big(T_{\vecz}>t\big)\right] \\
& \ge \proba_{\vecn}\left[T_{\vecnt}<\alcte{partout}\,\cte{partout}\, K^{d/2}\,\log K\right]\,
\sup_{\vecm\in \Delta}\proba_{\vecm}\big(T_{\vecz}>t\big)
\end{align*}
and the result follows from \eqref{vapartout}.\\
For $t\le \alcte{partout}\,\cte{partout}\, K^{d/2}\,\log K$, we use that 
\[
\sup_{\vecm\in \Delta}\proba_{\vecm}\big(T_{\vecz}>t\big)\le 1
\]
and Sublemma \ref{b}. This concludes the proof of Lemma \eqref{Harnack}.
\end{proof}

Condition \eqref{cond-A2} follows immediately using successively
the four preceding lemmas. The constant $c_{2}$ in \eqref{eq-harnack} is given by
\begin{equation}
\label{c2}
c_{2} = \frac{\cte{Harnack}}{\eta\,\left(1+\frac{\eta}{1-\delta\eta}\right)}<1.
\end{equation}

\subsection{Proof of Theorem \ref{Champ-Vill}}
The proof of Theorem \ref{Champ-Vill} follows from Conditions \eqref{cond-A1} and \eqref{cond-A2} using the result in \cite{cv}. The constant $c$ is equal to $c_{1} c_{2}<1$, where $c_{1}$ and $c_{2}$ have been defined in \eqref{c1} and \eqref{c2}. The number $t_{0}(K)$ defined in  \eqref{t0} is of order $\log K$.

\section{Proof of Theorem  \ref{pertemasse}}\label{sec:MTE}

\subsection{Proof of the upper bound}

The proof will be the consequence of the following  lemma. 
\begin{lemma}\label{masseinf}
Recall that $D$ has been defined in Sublemma \ref{b}. There exist $K_{0}>0$, $\bcte{masseinf}>0$ and $0<\cte{masseinf}<1$ such that for all $K\geq K_{0}$
\[
\inf_{\vecn\in D}\proba_{\vecn}\left(T_{D^{c}}>\e^{\bcte{masseinf} K}\right)\geq \cte{masseinf}.
\]
\end{lemma}

\begin{proof}
As in the proof of Sublemma \ref{b}, we use  Dynkin's Theorem applied to  $\lyap\big(\vecN^{\sK}(t\wedge  T_{D^{c}})\big)$ to obtain 
\begin{equation}
\label{dyn}
\esperance_{\vecn}\left[\lyap\big(\vecN^{\sK}(t\wedge  T_{D^{c}})\big)\right] =\lyap(\vecn)+ \esperance_{\vecn}\left[\int_{0}^{t\wedge T_{D^c}} \gen_{K}\varphi(\vecN^{\sK}(s)) \, \dd s\right].
\end{equation}
We distinguish the cases $\vecn\in \Delta$ and $\vecn\notin \Delta$.

Let us introduce the set $\widetilde \Delta = \{\vecn : \|\vecn - \vecnf\|\leq \crho{tondeuse}\, \sqrt{K}\}$, where the constant $ \crho{tondeuse}$ has been defined in Corollary \ref{tondeuse}. 

For  an initial state $\vecn\in \Delta$, we  remark that $T_{{\widetilde \Delta}^c} < T_{{D}^c}$.
For any $t>0$, Theorem \ref{formule} yields  
\[
\esperance_{\vecn}\left[\int_{0}^{t\wedge T_{{\widetilde \Delta}^c}} \gen_{K}\varphi(\vecN^{\sK})(s) \, \dd s\right] \leq \Oun\, t.
\]
We can write
\[
\int_{0}^{t\wedge T_{D^c}} \gen_{K}\varphi(\vecN^{\sK}(s)) \, \dd s =  \int_{0}^{t\wedge T_{\widetilde \Delta}^c} \gen_{K}\varphi(\vecN^{\sK}(s)) \, \dd s+ \int_{t\wedge T_{\widetilde \Delta}^c}^{t\wedge T_{D^c}} \gen_{K}\varphi(\vecN^{\sK}(s)) \, \dd s\,.
\]
Using Theorem \ref{formule}, we remark that   the first term of the rhs  is bounded by $\Oun\, t$. Corollary \ref{tondeuse} implies that the second term is non positive. 
In the other hand, there exists $b>0$ such that 
\[
\inf_{\vecn\in\partial D^{c}}\varphi(\vecn)\ge \e^{b K}
\]
By \eqref{dyn} we finally obtain
\[
\e^{b K} \proba_{\vecn}\left(T_{D^{c}}<t\right) \le \Oun \,t+\varphi(\vecn).
\]
Since
\[
\sup_{\vecn\in\Delta}\varphi(\vecn)=\Oun
\]
we conclude that for $K$ large enough 
$\sup_{\vecn\in\Delta}\proba_{\vecn}\left(T_{D^{c}}<\e^{b K/2}\right)<\frac{1}{2}$.
Therefore
\begin{equation}
\label{min}
\inf_{\vecn\in\Delta}\proba_{\vecn}\left(T_{D^{c}}\ge \e^{bK/2}\right)>\frac{1}{2}.
\end{equation}
For $\vecn\in D\backslash \Delta$ we have for $K$ large enough
(in particular $\e^{b\,K/2}>\cte{entrec}\;\log K$)
by Lemma  \ref{entrec} and the Markov property and monotonicity of
$t\mapsto \proba_{\vecn}\left(T_{D^{c}}\ge t\right)$  and \eqref{min}
\begin{align*}
\proba_{\vecn}\left(T_{D^{c}}\ge \e^{b\,K/2}\right)
& \ge
\esperance_{\vecn}\left[\un_{\{T_{\Delta}< \cte{entrec} \log K\}}
\proba\!_{\vecN^{\sK}(T_{\Delta})}\left(T_{D^{c}}\ge \e^{b\,K/2}-\, T_{\Delta}\right)\right]\\
&  \ge 
\esperance_{\vecn}\left[\un_{\{T_{\Delta}<\cte{entrec} \log K\}}
\; \proba\!_{\vecN^{\sK}(T_{\Delta})}\left(T_{D^{c}}\ge \e^{b\,K/2}-\, \cte{entrec} \log K\right)\right] \\
& \ge 
\esperance\!_{\vecn}\left[\un_{\{T_{\Delta}<\cte{entrec} \log K\}}
\; \proba_{\vecN^{\sK}(T_{\Delta})}\left(T_{D^{c}}\ge \e^{b\,K/2}\right)\right]\\
& \ge
\frac{1}{2} \;\proba\!_{\vecn}\big(T_{\Delta}<\cte{entrec} \log K\big)\,.
\end{align*}
The result follows from Lemma \ref{entrec} with $\cte{masseinf}=\frac{1-\ceta{entrec}}{2}$.
\end{proof}

We can now prove the upper bound of Theorem  \ref{pertemasse}
\[
\lambda_{0}(K) \leq \e^{-d_{2} K}
\]
for $d_{2}>0$. For $\vecn\in D$ and for all integer $q>1$ we have by the Markov property
\begin{align*}
\proba_{\vecn}\left(T_{D^{c}}\ge q\,\e^{\bcte{masseinf} K}\right)
& =
\esperance_{\vecn}\left[\un_{\left\{T_{D^{c}}>\e^{\bcte{masseinf} K}\right\}}\,
\proba_{\vecN\left(\e^{\bcte{masseinf} \sK}\right)}\left(T_{D^{c}}\ge (q-1)\e^{\bcte{masseinf} K}\right)\right] \\
& \ge   
\proba_{\vecn}\left({T_{D^{c}}>\e^{\bcte{masseinf} K}}\right)
\;\inf_{\vecm\in D} \proba_{\vecm}\left(T_{D^{c}}\ge (q-1) \e^{\bcte{masseinf} K}\right).
\end{align*}
Using Lemma \ref{masseinf} we get for all $q\ge 1$
\[
\inf_{\vecn\in D}
\proba_{\vecn}\left(T_{D^{c}}\ge q\,\e^{\bcte{masseinf} K}\right)\ge \cte{masseinf}^{\,q}.
\]
Therefore
\[
\inf_{\vecn\in D}
\proba_{\vecn}\left(T_{\vecz}\ge q\, \e^{\bcte{masseinf} K}\right)\ge \cte{masseinf}^{\,q}.
\]
By Property \eqref{cond-A1} proved in Subsection \ref{a1}, we know that
\[
\proba_{\vecn}\left(\vecN^{\sK}(t_{0}) \in .\right)\ge c_{1}\, \nu(.) \, \proba_{\vecn}\left(t_{0} < T_{0}\right).
\]
Integrating by $\qsd_{\sK}$ and using that $\proba_{ \qsd_{\sK}}\left(t_{0} < T_{0}\right) = e^{-\lambda_{0}(K) t_{0}}>0$, we obtain
\[
\qsd_{\sK}(.)\ge c_{1}\, \nu(.) \,.
\]
Then for all point  $\vecn_{0}\in \Delta$, 
$\qsd_{\sK}(\vecn_{0})\ge c_{1}\, \nu(\vecn_{0})>0$ and integrating by $\qsd_{\sK}$ we get for all $q\ge 1$
\[
\proba_{ \qsd_{\sK}}\Big(T_{\vecz}\ge q\;\e^{\bcte{masseinf}\;K}\Big)\ge
 \qsd_{\sK}(\vecn_{0})\; \cte{masseinf}^{\,q}.
\]
From \eqref{lambo} and this bound we deduce that 
\begin{align*}
\lim_{t\to\infty}-\frac{1}{t}\log \proba_{ \qsd_{\sK}}\left(T_{\vecz}\ge t\right)
& =\lim_{q\to\infty} -\frac{1}{q\,\e^{\bcte{masseinf}K}} 
\log \proba_{ \qsd_{\sK}}\left(T_{\vecz}\ge q\,\e^{\bcte{masseinf}K}\right)\\
& \leq -\log \cte{masseinf}\, \e^{-\bcte{masseinf}\;K}.
\end{align*}
We have used the fact that, since the limit exists, we can compute it along all diverging sequence.
Therefore we have proved the upper bound in Theorem \ref{pertemasse} with $d_2=\bcte{masseinf}/2$ provided that $K$ is large enough.

\subsection{Proof of the lower bound}

The proof will result from the following two lemmas. 

\begin{lemma}\label{lbi1}
There exists   $\bcte{lbi1}>0$ and $\tcte{lbi1}>0$ independent of $K$
such that for $K$ large enough
\[
\sup_{\vecn\in \mathcal{E}_{1}^{c}}\proba_{\vecn}\left(T_{\vecz}\ge \tcte{lbi1}\right)
\le 1-\e^{-\bcte{lbi1}\,K}.
\]
\end{lemma}
Recall that $\mathcal{E}_{1}^{c} =\big\{\vecn\in\domaine : \sum_{j=1}^{d}n_{j}\le L K\big\}$.
\begin{proof}
Starting from $\vecn\in \mathcal{E}_{1}^{c}$, we consider a path from $\vecn$ to
$\vecz$ obtained by decreasing successively the maximum of the
components.  We denote this path by $(\vecm(p))_{0\le p\le Q(\vecn)}$
with $\vecm(0)=\vecn$ and $\vecm(Q(\vecn))=\vecz$. We observe that
from the construction of $(\vecm(p))_{0\le p\le Q(\vecn)}$, the
sequence of integers $\max(\vecm_{\ell}(p))$ is nonincreasing (with jumps of size 1) and can
have plateaus of length at most $d$. Note also that $Q(\vecn) \leq d L K$. 

Using Hypotheses \eqref{lip}, \eqref{cond:vanish} and \eqref{cond:mort}  there exists $\xi'>0$ such that
for all $\vecn\in \mathcal{E}_{1}^{c}$ and for all $j$ such that $n_{j} = \max_{\ell=1,\ldots,d} n_{\ell}$,  
\[
K D_{j}\left(\frac{\vecn}{K}\right) \ge \xi  \,\max(n_{\ell})\quad\text{and}\quad
K\left(B_{j}\left(\frac{\vecn}{K}\right)+D_{j}\left(\frac{\vecn}{K}\right)\right) \leq \xi'
\,\max(n_{\ell}). 
\]
Therefore,
\[
\inf_{\vecn\in \mathcal{E}_{1}^{c}}\;\inf_{j\in \mathrm{argmax }\, \{n_{\ell}\}}
\frac{D_{j}\left(\frac{\vecn}{K}\right)}{\sum_{\ell=1}^d\left(B_{\ell}\left(\frac{\vecn}{K}\right)
+D_{\ell}\left(\frac{\vecn}{K}\right)\right)} \geq \frac{\xi}{d \xi'}>0. 
\]
This implies  that the probability of the path of the embedded chain
is larger than $(\xi/d \xi')^{dLK}$.  

If $\Theta$ denotes the first jump time of the process, it follows from the above
inequalities that for all we have for all $\vecn\in \mathcal{E}_{1}^{c}$ and for
all $\vecn$ 
\begin{align*}
\proba_{\vecn}\Big(\Theta < \frac{1}{\xi K}\Big)
& =
1 - \exp\left[- \frac{1}{\xi}\sum_{\ell=1}^d \left(B_{\ell}\left(\frac{\vecn}{K}\right)
+D_{\ell}\left(\frac{\vecn}{K}\right)\right)\right] \\
& \geq 
1 - \exp\Big(-\frac{\max_{1\leq \ell\leq d} n_{\ell}}{K}\Big). 
\end{align*}
For all $\vecn\in \mathcal{E}_{1}^{c}$, we  define a (measurable) set of trajectories
$\mathcal{T}_{\vecn}$ of the stochastic  process by
\[
\mathcal{T}_{\vecn}=\bigg\{\vecn(\cdot)\,\bigg|\,\;\exists \;
s_{0}=0<s_{1}<s_{2}<\cdots<s_{Q(\vecn)}\,,\,
\forall\,0\le
p\le Q(\vecn)-1, 
\]
\[
s_{p+1}-s_{p}<\frac{1}{\xi\,K}\,,\;\mathrm{and}\;
\vecn(t)=\vecm(p)\;\mathrm{for}\;t\in[ s_{p},s_{p+1}[\,,\;
\vecn(s_{Q(\vecn)})=\vecz\bigg\}\,.
\]
Let
\[
\tcte{lbi1}=\frac{d L}{\xi}\,.
\]
Since $Q(\vecn) \leq d L K$, we have for all $\vecn\in \mathcal{E}_{1}^{c}$
\begin{align*}
\MoveEqLeft \proba_{\vecn}\left(T_{\vecz}< \tcte{lbi1}\right) \\
&\ge \proba_{\vecn}\big(\mathcal{T}_{\vecn}\big) 
\ge \left(\frac{\xi}{d\xi'}\right)^{dLK}
\prod_{p=0}^{Q(\vecn)-1}\left(1-\e^{-\max_{1\leq \ell\leq d} m_{\ell}(p)/K}\right)\\
&\ge \left(\frac{\xi}{d\xi'}\right)^{dLK}\;
\left[\prod_{q=1}^{[L K]}\left(1-\e^{-q/K}\right)\right]^{d}\,.
\end{align*}
Let $\theta$ be the unique solution of $1-\e^{-\theta}=\theta/2$.
Note that $\theta>0$ and 
\[
1-\e^{-x}\ge \frac{x}{2}, \,\forall \;x\in[0,\theta],
\; 1-\e^{-x}\ge 1-\e^{-\theta},\,\forall \;x>\theta.
\]
For $L>\theta$ we have
\begin{align*}
\prod_{q=1}^{\lfloor L K\rfloor}\left(1-\e^{-q/K}\right)
&=
\prod_{q=1}^{\lfloor \theta K\rfloor}\left(1-\e^{-q/K}\right)
\prod_{q=1+\lfloor\theta K\rfloor}^{\lfloor L K\rfloor}\left(1-\e^{-q/K}\right)\\
& \ge
2^{-\lfloor\theta K\rfloor}\,\frac{\lfloor\theta K\rfloor!}{K^{\lfloor\theta K\rfloor}}\;
\left(1-\e^{-\theta}\right)^{\lfloor L K\rfloor-\lfloor\theta K\rfloor}\\
& \ge 
2^{-\lfloor\theta K\rfloor}\;\e^{-\lfloor \theta K \rfloor}\;\left(1-\e^{-\theta}\right)^{\lfloor L K\rfloor-\lfloor \theta K\rfloor}
\end{align*}
using Stirling's formula and $K$ large enough to obtain the last estimate. For $L\le\theta$ we have by a similar argument
\[
\prod_{q=1}^{\lfloor L K\rfloor}\left(1-\e^{-q/K}\right)\ge 2^{-\lfloor L K\rfloor}\;\e^{-\lfloor L K\rfloor}.
\]
We finally get
\[
\proba_{\vecn}\left(T_{\vecz}< \tcte{lbi1}\right)
\ge \left(\frac{\xi}{d\, \xi'}\right)^{d L K}
(2\,e)^{-d\,K\;(L\,\wedge\,\theta)}
\left(1-\e^{-\theta}\right)^{d\,(1+K\,((L-\theta)\vee \,0))},
\]
and the result follows since $\frac{\xi}{d\,\xi'}<1$.
\end{proof}

\begin{lemma}\label{lbi2}
There exists  $\tcte{lbi2}>0$ independent of $K$
such that for $K$ large enough
\[
\sup_{\vecn\in\domaine}\proba_{\vecn}\left(T_{\vecz}\ge \tcte{lbi2}\right)
\le 1-(1-\ceta{descente1})\;\e^{-\bcte{lbi2}\,K}.
\]
\end{lemma}

\begin{proof}
Let 
\[
\tcte{lbi2}=\tcte{lbi1}+\tcte{descente1}.
\]
By Lemma \ref{lbi1} we have 
\[
\sup_{\vecn\in\mathcal{E}_{1}^{c}}\proba_{\vecn}\left(T_{\vecz}\ge \tcte{lbi2}\right)
\le 1-\e^{-\bcte{lbi1}\,K}.
\]
For all $\vecn\in \mathcal{E}_{1}$ using the Markov property we obtain the following estimate
\begin{align*}
\proba_{\vecn}\big(T_{\vecz}>\tcte{lbi2}\big)
& = \proba_{\vecn}
\big(T_{\vecz}>\tcte{lbi2}\,,\,T_{\mathcal{E}_{1}^{c}}>\tcte{descente1}\big)+
\proba_{\vecn}\big(T_{\vecz}>\tcte{lbi2}\,,\,T_{\mathcal{E}_{1}^{c}}\le
\tcte{descente1}\big)\\
& \le
\proba_{\vecn}\big(T_{\mathcal{E}_{1}^{c}}> \tcte{descente1}\big) + \proba_{\vecn}\big(T_{\vecz}>\tcte{lbi2}\,,\,
T_{\mathcal{E}_{1}^{c}}\le\tcte{descente1}\big)  \\
& =
\proba_{\vecn}\big(T_{\mathcal{E}_{1}^{c}}> \tcte{descente1}\big)+
\esperance_{\vecn}\Big[\un_{\{T_{\mathcal{E}_{1}^{c}}\le \tcte{descente1}\}}\;
\proba_{\vecN^{\sK}(T_{\mathcal{E}_{1}^{c}})}\big(T_{\vecz}>\tcte{lbi2}-T_{\mathcal{E}_{1}^{c}}\big)\Big]\\
& \le 
\proba_{\vecn}\big(T_{\mathcal{E}_{1}^{c}}>
\tcte{descente1}\big)+
\esperance_{\vecn} \Big[\un_{\{T_{\mathcal{E}_{1}^{c}}\le \tcte{descente1}\}}\;
\proba_{\vecN^{\sK}(T_{\mathcal{E}_{1}^{c}})}\big(T_{\vecz}>\tcte{lbi1}\big)\Big]\\
& \le 
\proba_{\vecn}\big(T_{\mathcal{E}_{1}^{c}}>
\tcte{descente1}\big)+\proba_{\vecn}\big(T_{\mathcal{E}_{1}^{c}}\le
\tcte{descente1} \big)\;\left(1-\e^{-\bcte{lbi1}\,K}\right)
\end{align*}
where we made use of Lemma \ref{lbi1}. Using Lemma \ref{descente1} we get
\[
\proba_{\vecn}\big(T_{\vecz}>\tcte{lbi2}\big)
\le 1-\proba_{\vecn}\big(T_{\mathcal{E}_{1}^{c}}\le
\tcte{descente1} \big)\;\e^{-\bcte{lbi1}K}
\le 1-(1-\ceta{descente1})\;\e^{-\bcte{lbi1} K}\,.
\]
\end{proof}

We can now prove the lower bound.

\begin{proof}[Proof of the lower bound in Theorem \ref{pertemasse}]
Using Lemma \ref{lbi2} and the Markov property we get for all
$q\ge 0$
\[
\sup_{\vecn\in\domaine}\proba_{\vecn}\left(T_{\vecz}\ge q\, \tcte{lbi2}\right)\le \left(
1-(1-\ceta{descente1})\;\e^{-\bcte{lbi1}K}\right)^{q}\,.
\]
This implies
\[
\proba_{\qsd_{\sK}}\big(T_{\vecz}\ge q\, \tcte{lbi2}\big)\le  \left(
1-(1-\ceta{descente1})\;\e^{-\bcte{lbi1}K}\right)^{q}\,.
\]
Therefore
\begin{align*}
\lim_{t\to\infty}\frac{1}{t}\;\log \proba_{\qsd_{\sK}}\left(T_{\vecz} \ge
t\right)
& =\lim_{q\to\infty}\frac{1}{q}\;\log \proba_{\qsd_{\sK}}\left(T_{\vecz}\ge
q\, \tcte{lbi2}\right)\\
& \le \log\left(1-(1-\ceta{descente1})\;\e^{-\bcte{lbi1}K}\right)\,.
\end{align*}
and the lower bound follows by taking for instance
\[
d_{1}=1+\bcte{lbi1}.
\]
\end{proof}

\section{Proofs of Theorems \ref{mainincon} and  \ref{ancien}}\label{sec:mixtureandspectrum}

In the sequel we will assume that $K$ is large enough.
We first observe that since $\domaine$
is discrete and countable, the Banach space of bounded
complex measures on $\domaine$ equipped with the norm of total variation is  identical  to
$\lun$. Its dual is therefore $\linf$.

We establish a consequence of  Theorem \ref{Champ-Vill}.  To simplify the notation we write $\sg=\sg^K$ and $\sg^{\dagger}$  denotes the adjoint semigroup.\begin{corollary}\label{maincor}
There exists a constant $\cte{maincor}>0$ such that for all $K>K_{0}$, and for any $t>0$ 
\[
\big\|\rangun_{t}-\sg_{t}^{\dagger}\big\|_{\lun}\le \cte{maincor}\e^{-\omega t}
\]
where $\rangun_{t}$ is the rank one operator given by
\[
\langle\rangun_{t}\mu\,,\,f\rangle= \qsd_{\sK}(f)
\proba_{\mu}\big(t<
T_{\vecz}\big)
\]
and 
\[
\omega=\omega(K)=\frac{-\log(1-c)}{t_{0}(K)}\ge \frac{d_{3}}{\log K}
\]
where $d_{3}>0$ is independent of $K$.
\end{corollary}
\begin{proof} If $\mu$ is a probability measure, we get from Theorem
\ref{Champ-Vill} multiplying the estimate by $\proba_{\mu}\big(t<T_{\vecz}\big)$
\[
\big\|\rangun_{t}\mu-\sg_{t}^{\dagger}\mu\big\|_{\lun}\le \frac{2}{1-c}\; 
\proba_{\mu}\big(t<T_{\vecz}\big)\; \e^{-\omega t}
\le \frac{2}{1-c}\; \e^{-\omega t}\,.
\]
By standard arguments this implies that for  any sequence 
$f\in\lun$ we have
\[
\big\|\rangun_{t}f-\sg_{t}^{\dagger}f\big\|_{\lun}\le  \cte{maincor}\; \e^{-\omega t}\;\|f\|_{\lun},
\]
with $\cte{maincor}=4/(1-c)$. 
\end{proof}

We now derive some consequences of this estimate.

\begin{lemma}\label{evpt}
For any $s>0$, the operator $\sg_{s}^{\dagger}$ has only one eigenvalue
of modulus larger than  $\exp(-\omega s)$ which is equal to $\exp(-\lambda_{0} s)$. This eigenvalue is simple, the
corresponding eigenvector is $\qsd_{\sK}$.
\end{lemma}
\begin{proof}
Assume $f$ is an eigenvector of $\sg_{s}^{\dagger}$ with eigenvalue $z$
such that $|z|>\exp(-\omega s)$. From 
\[
{\sg_{s}^{\dagger}}^{n}f=z^{n}\;f,
\] 
we get using Corollary \ref{maincor} and the semi-group property
\begin{align*}
\big\|z^{n}\;f-\rangun_{n s}f\big\|_{\lun}
& =\big\|{\sg_{s}^{\dagger}}^{n}f-\rangun_{n s}f\big\|_{\lun} \\
& =\big\|\sg_{n s}^{\dagger}f-\rangun_{n s}f\big\|_{\lun} \\
& \le \cte{maincor}
\e^{-n\,\omega s}\|f\|_{\lun}\,.
\end{align*}
In other words
\[
\big\|f-z^{-n}\,\rangun_{n s}f\big\|_{\lun}
\le \cte{maincor}\, |z|^{-n}
\e^{-n \omega s} \|f\|_{\lun}\,.
\]
If $|z| > \exp(-\omega s)$ the right-hand side tends to zero when
$n$ tends to infinity and $f$ must be proportional to $\qsd_{\sK}$
(since $\rangun_{n s}f$ is proportional to  $\qsd_{\sK}$)
which is an eigenvector of $\sg_{s}^{\dagger}$ with eigenvalue
$\exp(-\lambda_{0}s)$. 

We now prove (by contradiction) that the equation
\[
\sg_{s}^{\dagger}f-\e^{-\lambda_{0} s}f=\qsd_{\sK}
\]
has no solution. Assume there exists such an $f\in \lun$ (which is
necessarily non zero).  We get
\[
{\sg_{s}^{\dagger}}^{n}f=n\,\e^{-(n-1)\lambda_{0}s}\, \qsd_{\sK}+\e^{-n\lambda_{0} s}\,f. 
\]
Therefore using again Corollary \ref{maincor} and the semi-group property we obtain
\[
\left\|n\,e^{-(n-1)\lambda_{0}s}\,\qsd_{\sK}+e^{-n\lambda_{0}s}f-\rangun_{n s}f\right\|_{\lun}
\le \cte{maincor}\;\e^{-n\,\omega s}\;\|f\|_{\lun}
\]
which implies
\[
\left\|f+n\,\e^{\lambda_{0}s}\qsd_{\sK}-
\e^{n\lambda_{0}s}\rangun_{n s}f\right\|_{\lun}
\le \cte{maincor}\;
\e^{-n\,(\omega-\lambda_{0}) s}\,.
\] 
Since $\omega>\lambda_{0}$ the right hand side tends to zero when $n$
tends to infinity and we deduce that $f$ must be proportional to
$\qsd_{\sK}$, a contradiction.
\end{proof}

The following result completes the description of the spectrum of
$\sg_{s}^{\dagger}$ outside the disk in the complex plane 
of radius $\exp(-\omega  s)$.

\begin{proposition}\label{specpt}
For any $s>0$ the operator $\sg_{s}^{\dagger}$ \textup{(}as an operator in $\lun$\textup{)} 
has spectral radius $\exp(-\lambda_{0}\,s)$
and essential spectral radius at most $\exp(-\omega s)$. 
Outside the disk $|z|\le \exp(-\omega\,s)$ the spectrum consists of only one  simple
eigenvalue $\exp(-\lambda_{0}s)$ with eigenvector $\qsd_{\sK}$.
\end{proposition}
\begin{proof}
From Corollary \ref{maincor} and the semi-group property, 
we have for any integer $n$
\[
\big\|\rangun_{n s}-{\sg_{s}^{\dagger}}^{n}\big\|_{\lun}
\le \cte{maincor} \e^{-n \omega s}.
\]
Therefore  since all operators $\rangun_{ns}$ have rank one and
therefore are compact, it follows from Corollaries  I.4.9 and I.4.11 in \cite{EE} (page 44) that the essential spectral radius of
$\sg_{s}^{\dagger}$ is at most $\exp(-\omega s)$. The rest of the proposition follows from Lemma  \ref{evpt} since outside of the essential spectrum, the spectrum can only consists of isolated eigenvalues with finite
algebraic and geometric multiplicities.
\end{proof}

\begin{proposition}\label{specdual}
For any $s>0$,
the operator $\sg_{s}$ acting in $\linf$ 
has spectral radius $\exp(-\lambda_{0}\,s)$
and essential spectral radius at most $\exp(-\omega s)$.
Outside the disk $|z|\le \exp(-\omega \,s)$ the spectrum consists of only one
eigenvalue $\exp(-\lambda_{0}s)$ with a simple strictly positive 
eigenvector $\vp_{\sK}$ satisfying $\qsd_{\sK}(\vp_{\sK})=1$ and independent of $s$.
\end{proposition}
\begin{proof}
The result follows from Theorem IX.1.1 in \cite{EE} and 
Proposition \ref{specpt} except for the
properties of the eigenvector. We have (where $\fun$ is the constant
function one on $\domaine$) 
\[
\lim_{n\to \infty} \e^{n\,\lambda_{0} s}\sg_{n\,s}\fun= v_{s},
\]
From $\langle \qsd_{\sK}\,,\, \fun\rangle =1$, we conclude that 
$\langle \qsd_{\sK}\,,\, v_{s}\rangle =1$, and hence $v_{s}$
is an eigenvector of $\sg_{s}$ with eigenvalue 
$\exp(-\lambda_{0}s)$.  Since the operator $\sg_{s}$ maps
positive functions to positive functions we conclude that $v_{s}$
is positive. 

Let $t'>0$ not being an integer multiple of $s$. By the semi-group
property we have 
\[
\sg_{s}\big(\sg_{t'}\,v_{s}\big)= \e^{-\lambda_{0}\,s} \sg_{t'}\,v_{s}.
\]
Since $\sg_{t'}\,v_{s}$ is positive and the eigenvalue
$\exp(-\lambda_{0}s)$ of $\sg_{s}$ is simple, this function must be
proportional to $v_{s}$. From
\[
\langle \qsd_{\sK}\,,\,\sg_{t'} v_{s}\rangle= \e^{-\lambda_{0} t'}
\]
we conclude that
\[
\sg_{t'}\,v_{s}= \e^{-\lambda_{0}\,t'} v_{s}\,.
\]
The independence of $v_{s}$ on $s$ follows and we denote this vector by $\vp_{\sK}$.
\end{proof}

\begin{proposition}\label{normevp}
There exists a positive constant $\cte{normevp}$ such that
\[
\|\vp_{\sK}\|_{\linf}\le 1+\e^{-\cte{normevp}K}\,.
\]
Moreover
\[
\inf_{\vecn\in\domaine}\vp_{\sK}(\vecn)\ge c
\]
where the constant $c$ is defined in Theorem  \ref{Champ-Vill}.
\end{proposition}
\begin{proof}
Form Corollary \ref{maincor} it follows that  for any $t>0$
\begin{equation}\label{maincorbis}
\big\|\rangun_{t}^{\dagger}-\sg_{t}\big\|_{\linf}\le 
\cte{maincor} \e^{-\omega\,t}.
\end{equation}
From the definition of $\rangun_{t}$ and $\qsd_{\sK}(\vp_{\sK})=1$ we have
\[
\rangun_{t}^{\dagger}\vp_{\sK}=\proba_{(\,\cdot\,)}\big(t<T_{\vecz}\big).
\]
Hence
\begin{align}
\nonumber
\big\|\proba_{(\,\cdot\,)}\big(t<T_{\vecz}\big)
-\e^{-\lambda_{0}\,t}\vp_{\sK}\big\|_{\linf} &=
\big\|\rangun_{t}^{\dagger}\vp_{\sK}-\sg_{t}\vp_{\sK}\big\|_{\linf} \\
\label{restep}
& \le \cte{maincor} \e^{-\omega\,t}\; \|\vp_{\sK}\|_{\linf}.
\end{align}
Therefore if $t$ is large enough so that $\exp(-\lambda_{0}\,t)>
\cte{maincor} \exp(-\omega\,t)$ we obtain
\[
\|\vp_{\sK}\|_{\linf}\le
\frac{\big\|\proba_{(\,\cdot\,)}\big(t<T_{\vecz}\big)\big\|_{\linf}} 
{\e^{-\lambda_{0}\,t} -\cte{maincor}
  \e^{-\omega\,t}} \le \frac{1}  
{\e^{-\lambda_{0}\,t}-\,\cte{maincor} \e^{-\omega\,t}} .
\]
The first result follows by taking $t=K\,\log K$ since
$\lambda_{0}=\exp(-\Oun\,K)$ (see Theorem \ref{pertemasse}) and $\omega=\Oun/\log K$.

From the positivity of $\vp_{\sK}$ and by \eqref{eureka}
 we get (with $t_{0}=t_{0}(K)$)
\begin{equation}\label{lbc1}
\e^{-\lambda_{0}\,t_{0}}\vp_{\sK}\geq \sg_{t_{0}}(\vp_{\sK})\ge \sg_{t_{0}}(\un_{\Delta}\vp_{\sK}) \ge
c_{1}\;\nu(\vp_{\sK}).
\end{equation} 
For any $t>0$ by integrating both sides of \ref{cond-A2} against the positive
measure $\qsd_{\sK}$ we get
\[
c_{2}\; \e^{-\lambda_{0}\,t}=c_{2}\;\proba_{\qsd_{\sK}}\big(t<T_{\vecz}\big)
\le \proba_{\nu}\big(t<T_{\vecz}\big)
=\sum_{\vecm}\nu(\vecm)\;\proba_{\vecm}\big(t<T_{\vecz}\big).
\]
From the estimate \eqref{restep} and the first result we get
\[
c_{2}\;\e^{-\lambda_{0}\,t}\le
\sum_{\vecm}\nu(\vecm)\; \e^{-\lambda_{0}\,t}\vp_{\sK}(\vecm)
+\Oun \; \e^{-\omega\,t}.
\]
Multiplying by $\exp(\lambda_{0}\,t)$ and letting $t$ tend to
infinity, we get (since $\lambda_{0}<\omega$)
\[
c_{2}\le \nu(\vp_{\sK}).
\]
The second result follows by combining this estimate with the lower
bound \eqref{lbc1}
\end{proof}
For each $\vecn\in\domaine$ we define
\[
p_{\sK}(\vecn)=\vp_{\sK}(\vecn)\wedge 1.
\]

\begin{proof}[Proof of Theorem \ref{mainincon}]
Using the estimate \eqref{restep}  and Proposition \ref{normevp}, we
get for any $\vecn\in\domaine$
\begin{align*}
\left|\proba_{\vecn}\big(t<T_{\vecz}\big)
- \e^{-\lambda_{0}\,t}p_{\sK}(\vecn)\right| &\le 
\e^{-\lambda_{0}\,t}\; \e^{-\cte{normevp}\,K}
+\,\cte{maincor}\; \e^{-\omega\,t}\left(1+\e^{-\cte{normevp}\,K}\right)\\
&
\le \e^{-\cte{normevp}\,K} \e^{-\lambda_{0}\,t}+2\,\cte{maincor}\;\e^{-\omega\,t}.
\end{align*}
The result follows from Corollary \ref{maincor}
\end{proof}
\begin{proof}[Proof of Theorem \ref{ancien}]
Combining the estimates \eqref{maincorbis} and \eqref{restep} we
obtain
\[
\left\|\sg_{t}- \e^{-\omega\,t}\;\pi_{\sK}\right\|_{\linf}\le 
\cte{maincor} \e^{-\omega\,t}\;\left(2+\e^{-\cte{normevp}\,K}\right).
\]
Therefore if $\Re z>-\omega$ and $z\neq -\lambda_{0}(K)$ we have
\[
\int_{0}^{\infty}\sg_{t} \e^{-t\,z} \dd t=\frac{\pi_{\sK}}{z+\lambda_{0}(K)}+
\EuScript{M}_{z}
\]
where $\EuScript{M}_{z}$ is analytic  in $\Re z>-\omega$.
The result follows.
\end{proof}

\section{Proof of Lemma \ref{boulot}}\label{sec:pleindelemmes}

The proof of Lemma \ref{boulot} follows from a series of sublemmas.
We first estimate the contribution of $\Pi_q^2$ to $S_{q}(\vecn,\vecm)$.

\begin{sublemma}\label{moment}
For all $K$ large enough and  all $q\in[\,\Lambda_{*}-\sqrt K,\Lambda_{*}+\sqrt K\, ]$
\[
\sup_{0\le s\le q}\;\sup_{\substack{\vecn\in\Delta'\\ \vecm\in\Delta}}
\sum_{\substack{\vg\, : \,\vg(0)=\vecn,\;\vg(q)=\vecm\\ \sup_{1\le\ell\le q-1}|\vg(\ell)-\vecnf|<\sqrt K\,\log K}}
\prod_{\ell=0}^{q-1}p^{*}\big(\vg(\ell),\vg(\ell+1)\big)\;|\vg(s)-\vecnf|^{2}
\]
\[
\le \Oun\;K^{-d/2+1}.
\]
\end{sublemma}

\begin{proof}
We have
\begin{align*}
\MoveEqLeft
\sup_{0\le s\le q}\;\sup_{\substack{\vecn\in\Delta'\\ \vecm\in\Delta}}
\sum_{\substack{\vg\, :\,\vg(0)=\vecn,\;\vg(q)=\vecm\\ \sup_{1\le\ell\le q-1}|\vg(\ell)-\vecnf|<\sqrt K\,\log K}}
\prod_{\ell=0}^{q-1}p^{*}\big(\vg(\ell),\vg(\ell+1)\big)
\;|\vg(s)-\vecnf|^{2}\\
& \leq
\sup_{0\le s\le q}\;\sup_{\substack{\vecn\in\Delta'\\\vecm\in\Delta}}
\sum_{\vg\,:\,\vg(0)=\vecn\, :\,\vg(q)=\vecm}\; \prod_{\ell=0}^{q-1}p^{*}\big(\vg(\ell),\vg(\ell+1)\big)
\;|\vg(s)-\vecnf|^{2}.
\end{align*}
For $0 \leq s\leq q$ we have
\begin{align*}
\MoveEqLeft[4]
\sum_{\vg\,:\,\vg(0)=\vecn,\;\vg(q)=\vecm}\;
\prod_{\ell=0}^{q-1}p^{*}\big(\vg(\ell),\vg(\ell+1)\big)
\;|\vg(s)-\vecnf|^{2}\\
& =
\sum_{\vecu} \proba_{\vecn}(\plat_s=\vecu)\, |\vecu-\vecnf|^2\, \proba_{\vecu}(\plat_{q-s}=\vecm)\\
&=
\sum_{\vecu} \proba_{\vecn}(\plat_s=\vecu)\, |\vecu-\vecnf|^2\, \proba_{\vecm}(\plat_{q-s}=\vecu)
\end{align*}
where the second equality follows from the reversibility of the random walk $(\plat_{\ell})_{\ell\in \integers}$.
For $s\leq q/2$ and from the local  limit theorem \cite{Durrett} (applied to $\plat$) we have
\[
\proba_{\vecm}(\plat_{q-s}=\vecu) \le \frac{\Oun}{K^{d/2}}.
\]
It follows that 
\[
\sum_{\vg,\;\vg(0)=\vecn,\;\vg(q)=\vecm}\;
\prod_{\ell=0}^{q-1}p^{*}\big(\vg(\ell),\vg(\ell+1)\big)
\;|\vg(s)-\vecnf|^{2}\leq
\]
\[
\frac{\Oun}{K^{d/2}}\, \sum_{\vecu} \proba_{\vecn}(\plat_s=\vecu)\, |\vecu-\vecnf|^2
\leq \frac{\Oun}{K^{d/2}}\, \esperance_{\vecz}\left( |\plat_s+\vecn-\vecnf|^2 \right)\leq
\frac{\Oun}{K^{d/2-1}}.
\]
For $s>q/2$ the proof is similar (exchange the role of $\vecn$ and $\vecm$).
\end{proof}

\begin{sublemma}\label{boulot2}
There exists a constant $\cte{boulot2}>0$ independent of $K$ such that for
$K$ large enough and all $q\in[\Lambda_{*}-\sqrt K,\Lambda_{*}+\sqrt K]$
\[
K^{d/2}\;q!\;\left|\;\sum_{\substack{\vg,\;\vg(0)=\vecn,\;\vg(q)
=\vecm\\ \sup_{1\le\ell\le q-1}|\vg(\ell)-\vecnf|\le \sqrt K\,\log K}}
\int\cdots\int_{t_{0}+\cdots + t_{q-1}<1}\;
\prod_{\ell=0}^{q-1}p^{*}\big(\vg(\ell),\vg(\ell+1)\big)\right.
\]
\[
\left.
\log \Pi_q^2\big(\vg,t_{0},\ldots,t_{q-1}\big)
\prod_{\ell=0}^{q-1}\dd t_{\ell}
\vphantom{\sum_{\substack{\vg,\;\vg(0)=\vecn,\;\vg(q)=\vecm\\
\sup_{1\le\ell\le q-1}|\vg(\ell)-\vecnf|\le \sqrt K\,\log K}}
\int\cdots\int_{t_{0}+\cdots + t_{q-1}<1}}\;\right|\le \cte{boulot2}.
\]
\end{sublemma}
\begin{proof}
We first observe that for all $0\le s\le q-1$
\begin{align*}
& \int\cdots\int_{t_{0}+\cdots
  t_{q-1}<1}\;t_{s}\;\prod_{\ell=0}^{q-1}\dd t_{\ell}\\
& =\int\cdots\int_{t_{0}+\cdots+ t_{q-1}<1}\;t_{0}\;\prod_{\ell=0}^{q-1}\dd t_{\ell}\\
&=\int_{0}^{1}t_{0}\;\dd t_{0}\;\int\cdots\int_{t_{1}+\cdots +t_{q-1}<1-t_{0}}\prod_{\ell=1}^{q-1}\dd t_{\ell}\\
& =\frac{1}{(q-1)!}\;\int_{0}^{1}t_{0}(1-t_{0})^{q-1}\;\dd t_{0}\\
& =\frac{1}{(q-1)!}\;\left(\int_{0}^{1}(1-t_{0})^{q-1}\;\dd t_{0}-\int_{0}^{1}(1-t_{0})^{q}\;\dd t_{0}\right)\\
& =\frac{1}{(q-1)!}\;\left(\frac{1}{q}-\frac{1}{(q+1)}\right) =\frac{1}{(q+1)!}.
\end{align*}
Therefore we have to estimate
\[
K^{d/2}\sum_{\substack{\vg,\;\vg(0)=\vecn,\;\vg(q)=\vecm\\ \sup_{1\le\ell\le q-1}|\vg(\ell)-\vecnf|\le \sqrt K\,\log K}}
\;\prod_{\ell=0}^{q-1}p^{*}\big(\vg(\ell),\vg(\ell+1)\big)\;\times
\]
\[
\hskip 1cm\sum_{\ell=0}^{q-1}
\left[\log\left(\frac{\Lambda\big(\vg(\ell)\big)}{\Lambda_{*}}\right)
-\frac{1}{q+1}\,\big(\Lambda \big(\vg(\ell)\big)-\Lambda_{*}\big)\right].
\]
Since $|\vg(\ell)-\vecnf|\le \sqrt K\log K$ one has
\[
\log\left(\frac{\Lambda\big(\vg(\ell)\big)}{\Lambda_{*}}\right)=
\frac{\Lambda\big(\vg(\ell)\big)-\Lambda_{*}}{\Lambda_{*}}-
\frac{\big(\Lambda\big(\vg(\ell)\big)-\Lambda_{*}\big)^{2}}{2\,\Lambda_{*}^{2}}
+\mathcal{O}\big(K^{-3/2}\,(\log K)^{3}\big).
\]
Therefore,
\begin{align*}
\MoveEqLeft \log\left(\frac{\Lambda\big(\vg(\ell)\big)}{\Lambda_{*}}\right)
-\frac{1}{q+1}\,\big(\Lambda \big(\vg(\ell)\big)-\Lambda_{*}\big)\\
&=
\frac{\big(\Lambda\big(\vg(\ell)\big)-\Lambda_{*}\big)\,
\big(q+1-\Lambda_{*}\big)}{\Lambda_{*}\;(q+1)}-
\frac{\big(\Lambda\big(\vg(\ell)\big)-\Lambda_{*}\big)^{2}}{2\,\Lambda_{*}^{2}}
+\mathcal{O}\big(K^{-3/2}\,(\log K)^{3}\big).
\end{align*}
This implies using the binomial inequality that
\begin{align*}
\left|\log\left(\frac{\Lambda\big(\vg(\ell)\big)}{\Lambda_{*}}\right)
-\frac{1}{q+1}\,\big(\Lambda \big(\vg(\ell)\big)-\Lambda_{*}\big)\right|
&\le
\frac{\big(q+1-\Lambda_{*}\big)^{2}}{2(q+1)^{2}}+
\frac{\big(\Lambda\big(\vg(\ell)\big)-\Lambda_{*}\big)^{2}}{\Lambda_{*}^{2}}\\
&
\le \frac{\big(\Lambda\big(\vg(\ell)\big)-\Lambda_{*}\big)^{2}}{\Lambda_{*}^{2}}
+\mathcal{O}\big(K^{-1}\big)\,.
\end{align*}
The result follows from Sublemma \ref{moment}, the Lipschitz continuity  of
$\lambda$ and the fact that $\lambda(x_{*})$ is independent of $K$. (These quantities are defined
in \eqref{def-petit-lambda}.)
\end{proof}

Let us now estimate the contribution of $\Pi_q^3$ to $S_{q}(\vecn,\vecm)$.

\begin{sublemma}\label{lem:cancelation}
For all $K$ large enough
\[
\sup_{\substack{\vecw \in \Delta',\, \vecv\in\Delta,\\\,1\leq j\leq d,\, q/3\leq \ell \leq q}}
\sum_{\vecu} 
\proba_{\vecv}(\plat_{q-\ell-1}=\vecu)
\,\big|\proba_{\vecw}(\plat_\ell=\vecu)-\proba_{\vecw}(\plat_\ell=\vecu\pm \vece{j})\big|\,
\|\vecu-\vecnf\|
\]
\[
\leq \frac{\Oun}{K^{d/2}}.
\]
\end{sublemma}

\begin{proof}
The reader can check that
\[
\proba_{\vecz}(\plat_\ell=\vecn)=
\frac{1}{(2\pi)^d} \int_{[-\pi,\pi]^d} \e^{-\ii \langle \vecn,\vectheta\rangle}\, \e^{-\ell S(\vectheta)} \,
\dd \vectheta
\]
where
\[
S(\vectheta)=-\log \left(2 \sum_{m=1}^d r_m \cos(\theta_m)\right).
\]
Then we have
\[
\sup_{\vecn}|\proba_{\vecz}(\plat_\ell=\vecn)-\proba_{\vecz}(\plat_\ell=\vecn \pm \vece{j})|
\leq 
\frac{2}{(2\pi)^d} \int_{[-\pi,\pi]^d}  \e^{-\ell S(\vectheta)}\, \Big|\sin \frac{\theta_j}{2} \Big| \,
\dd \vectheta.
\]
Next using a Taylor expansion of $S$ around $\underline{0}$, we have
\begin{align*}
\MoveEqLeft
\sup_{q/3\leq \ell \leq q}\sup_{\vecn}|\proba_{\vecz}(\plat_\ell=\vecn)-\proba_{\vecz}(\plat_\ell=\vecn \pm \vece{j})|
\\
& \leq
\e^{-\Oun\, (\log K)^2} +
\frac{1}{(2\pi)^d} \e^{\frac{\Oun(\log K)^4}{K}} 
\int_{\|\vectheta\| \leq \frac{\log K}{\sqrt{K}}}  \e^{-\frac{q}{3} \sum_{m=1}^d r_m \theta_m^2}\,
\left|\theta_j \right| 
\dd \vectheta\\
& \leq \frac{\Oun}{K^{\frac{d+1}{2}}}.
\end{align*}
Therefore
\begin{align*}
\MoveEqLeft
\sup_{\substack{\vecw \in \Delta',\, \vecv\in\Delta,\\\,1\leq j\leq d,\, q/3\leq \ell \leq q}}
\sum_{\vecu} 
\proba_{\vecv}(\plat_{q-\ell-1}=\vecu)\,
\big|\proba_{\vecw}(\plat_\ell=\vecu)-\proba_{\vecw}(\plat_\ell=\vecu\pm \vece{j})\big|\,
\|\vecu-\vecnf\|\\
& \leq
 \frac{\Oun}{K^{\frac{d+1}{2}}}\,\sup_{\vecv\in\Delta}\; \sup_{q/3\leq \ell \leq q}
\sum_{\vecu} 
\proba_{\vecv}(\plat_{q-\ell-1}=\vecu)\,\|\vecu-\vecnf\|\\
& =
\frac{\Oun}{K^{\frac{d+1}{2}}}\,\sup_{\vecv\in\Delta}\; \sup_{q/3\leq \ell \leq q}
\sum_{\underline z} \proba_{\vecz}(\plat_{q-\ell-1}=\underline z)\,\|\underline z+\vecv-\vecnf\|\leq
\frac{\Oun}{K^{\frac{d}{2}}}
\end{align*}
where we used the triangle inequality and Cauchy-Schwarz inequality.
\end{proof}

\begin{sublemma}\label{compense}
For all $K$ large enough and  all $q\in[\Lambda_{*}-\sqrt K,\Lambda_{*}+\sqrt K]$
\begin{align*}
\MoveEqLeft[8] \sup_{1\le r\le d}\;\sup_{\substack{\vecn\in\Delta'\\ \vecm\in\Delta}}
\Bigg|
\sum_{\substack{\vg\,:\,\vg(0)=\vecn,\;\vg(q)=\vecm\\ \sup_{1\le\ell\le q-1}|\vg(\ell)-\vecnf|<\sqrt K\log K}}
\prod_{\ell'=0}^{q-1}p^{*}\big(\vg(\ell'),\vg(\ell'+1)\big)\\
&\! \sum_{\ell=0}^{q-1} \; \sum_{j=1}^d \; \sum_{r=1}^d
\frac{1}{K\,r_{j}}
\bigg[\,
\mb_{j,r}\,(\vg(\ell)_{r}-\vecnf_{r})
\; \delta_{\gamma_j(\ell+1)-\gamma_j(\ell), 1} \\
 &\qquad\qquad\qquad\quad  +
\md_{j,r}\,(\vg(\ell)_{r}-\vecnf_{r})
\; \delta_{\gamma_j(\ell+1)-\gamma_j(\ell), -1}
\bigg]
\Bigg|\\
 & \le \Oun\, K^{-d/2}.
\end{align*}
\end{sublemma}

\begin{proof}
For all $s\in\entiers$ we define
\begin{align*}
\qroba_{\vecn}(\plat_s=\vecu)
&=
\sum_{\substack{\vg\, :\,\vg(0)=\vecn,\;\vg(s)=\vecu\\ \sup_{1\le\ell\le s-1}|\vg(\ell)-\vecnf|<\sqrt K \log K}}
\prod_{\ell'=0}^{s-1}p^{*}\big(\vg(\ell'),\vg(\ell'+1)\big)\\
&= \esperance_{\vecn}
\left(\un_{\{\sup_{0\leq \ell\leq s-1} |\plat_{\ell}-\vecnf |<\sqrt{K} \log K\}}
\; \un_{\{\plat_s=\vecu\}}\right).
\end{align*}
After a simple cancellation we get
\begin{align*}
\MoveEqLeft[8] \sum_{\substack{\vg\,:\,\vg(0)=\vecn,\;\vg(q)=\vecm\\ \sup_{1\le\ell\le q-1}|\vg(\ell)-\vecnf|<\sqrt K\,\log K}}
\prod_{\ell'=0}^{q-1}p^{*}\big(\vg(\ell'),\vg(\ell'+1)\big)\\
& \!\!\!\sum_{\ell=0}^{q-1} \; \sum_{j=1}^d \; \sum_{r=1}^d \frac{1}{K r_{j}}
\bigg[\,
\mb_{j,r}\,(\vg(\ell)_{r}-\vecnf_{r})
\; \delta_{\gamma_j(\ell+1)-\gamma_j(\ell),1}\\
& \qquad\qquad\qquad\quad +
\md_{j,r}\,(\vg(\ell)_{r}-\vecnf_{r})
\; \delta_{\gamma_j(\ell+1)-\gamma_j(\ell),-1}
\bigg] \\
\MoveEqLeft[8]
=\sum_{\ell=0}^{q-1} \; \sum_{j=1}^d \; \sum_{r=1}^d \frac{1}{K} \sum_{|\vecu-\vecnf|<\sqrt{K}\log K}
\qroba_{\vecn} (\plat_\ell=\vecu)\; \times\\
\MoveEqLeft[6] \bigg[
\mb_{j,r}\,(\vecu_{r}-\vecnf_{r})\,  \qroba_{\vecu+\vece{j}}(\plat_{q-\ell-1}=\vecm)
+
\md_{j,r}\,(\vecu_{r}-\vecnf_{r})\,  \qroba_{\vecu-\vece{j}}(\plat_{q-\ell-1}=\vecm)
\bigg]\\
\MoveEqLeft[8] =\sum_{\ell=0}^{q-1} \; \sum_{j=1}^d \; \sum_{r=1}^d \frac{1}{K}\; \sum_{|\vecu-\vecnf|<\sqrt{K}\log K}
\qroba_{\vecn} (\plat_\ell=\vecu)\; \times
\end{align*}
\[
\left[
\mb_{j,r}\,(\vecu_{r}-\vecnf_{r})\,  \qroba_{\vecm}(\plat_{q-\ell-1}=\vecu+\vece{j})
+
\md_{j,r}\,(\vecu_{r}-\vecnf_{r})\,  \qroba_{\vecm}(\plat_{q-\ell-1}=\vecu-\vece{j})
\right],
\]
where we used the reversibility property of $(\plat_{\ell})_{\ell\in \integers}$ under $\qroba$.
Using  Sublemmas \ref{lem:qrobaproba} and \ref{lem:erreurproba} to bound the corrections,  it is enough
to estimate
\[
\sum_{\ell=0}^{q-1} \; \sum_{j=1}^d \; \sum_{r=1}^d \frac{1}{K}\; \sum_{\vecu}
\proba_{\vecn} (\plat_\ell=\vecu)\; \times
\]
\[
\left[
\mb_{j,r}\,(\vecu_{r}-\vecnf_{r})\,  \proba_{\vecm}(\plat_{q-\ell-1}=\vecu+\vece{j})
+
\md_{j,r}\,(\vecu_{r}-\vecnf_{r})\,  \proba_{\vecm}(\plat_{q-\ell-1}=\vecu-\vece{j})
\right].
\]
Using Sublemma \ref{lem:cancelation} we have
\begin{align*}
\MoveEqLeft
\sum_{\ell=0}^{2q/3-1} \; \sum_{j=1}^d \; \sum_{r=1}^d K^{-1}\; \sum_{\vecu}
\proba_{\vecn} (\plat_\ell=\vecu)\; \times \\
&
\left[
\mb_{j,r}\,(\vecu_{r}-\vecnf_{r})\,  \proba_{\vecm}(\plat_{q-\ell-1}=\vecu+\vece{j})
+
\md_{j,r}\,(\vecu_{r}-\vecnf_{r})\,  \proba_{\vecm}(\plat_{q-\ell-1}=\vecu-\vece{j})
\right]\\
 \MoveEqLeft =
\sum_{\ell=0}^{2q/3-1} \; \sum_{j=1}^d \; \sum_{r=1}^d K^{-1}\; \sum_{\vecu}
\proba_{\vecn} (\plat_\ell=\vecu)\; \times \\
&
\left[
\mb_{j,r}\,(\vecu_{r}-\vecnf_{r})\, +\md_{j,r}\,(\vecu_{r}-\vecnf_{r})\, \right] 
\proba_{\vecm}(\plat_{q-\ell-1}=\vecu)+\frac{\Oun}{K^{d/2}}\\
& = \frac{\Oun}{K^{d/2}}.
\end{align*}
Indeed, for all $1\leq r\leq d$, we have 
$\sum_{j=1}^d (\mb_{j,r}+\md_{j,r})=0$ since 
\[
\sum_{j=1}^d \frac{B_{j}(\vecx) +D_{j}(\vecx)}{\lambda(\vecx)}=1.
\]
We also have
\begin{align*}
\MoveEqLeft
\sum_{\ell=2q/3}^{q-1} \; \sum_{j=1}^d \; \sum_{r=1}^d K^{-1}\; \sum_{\vecu}
\proba_{\vecn} (\plat_\ell=\vecu)\; \times \\
& 
\left[
\mb_{j,r}\,(\vecu_{r}-\vecnf_{r})\,  \proba_{\vecm}(\plat_{q-\ell-1}=\vecu+\vece{j})
+
\md_{j,r}\,(\vecu_{r}-\vecnf_{r})\,  \proba_{\vecm}(\plat_{q-\ell-1}=\vecu-\vece{j})
\right]\\
\MoveEqLeft 
= \sum_{\ell=2q/3}^{q-1} \; \sum_{j=1}^d \; \sum_{r=1}^d K^{-1}\; \sum_{\vecv}
\proba_{\vecm} (\plat_{q-\ell-1}=\vecv)\; \times \\
&
\left[
\mb_{j,r}\,(\vecv_{r}-\vece{j}_r-\vecnf_{r})\,  \proba_{\vecn}(\plat_{\ell}=\vecv-\vece{j})
+
\md_{j,r}\,(\vecv_{r}+\vece{j}_r-\vecnf_{r})\,  \proba_{\vecn}(\plat_{\ell}=\vecv+\vece{j})
\right].
\end{align*}
Using again Sublemma \ref{lem:cancelation} and the same cancellation as before this is equal to
\begin{align*}
\MoveEqLeft
\frac{\Oun}{K^{d/2}} + 
\sum_{\ell=2q/3}^{q-1} \; \sum_{j=1}^d \; K^{-1}\, \sum_{\vecv}
\proba_{\vecm} (\plat_{q-\ell-1}=\vecv)\; \times\\
&
\left[
-\mb_{j,j}\,  \proba_{\vecn}(\plat_{\ell}=\vecv-\vece{j})
+
\md_{j,j}\,  \proba_{\vecn}(\plat_{\ell}=\vecv+\vece{j})
\right]=\frac{\Oun}{K^{d/2}}
\end{align*}
since
\[
\sup_{\vecu\in \zentiers^d}\sup_{\ell\geq q/3} \sup_{\vecn\in \Delta'} \proba_{\vecn}(\plat_\ell=\vecu)\leq
\frac{\Oun}{K^{d/2}},
\] using once again the local limit theorem.
\end{proof}

\begin{sublemma}\label{boulot3}
There exists a constant $\cte{boulot3}>0$ independent of $K$ such that for
$K$ large enough and all $q\in[\Lambda_{*}-\sqrt K,\Lambda_{*}+\sqrt K]$
\[
K^{d/2}\;q!\;\left|\;\sum_{\substack{\vg:\;\vg(0)
=\vecn,\;\vg(q)=\vecm\\ \sup_{1\le\ell\le q-1}|\vg(\ell)-\vecnf|\le \sqrt K\,\log K}}
\int\cdots\int_{t_{0}+\cdots + t_{q-1}<1}\;
\prod_{\ell=0}^{q-1}p^{*}\big(\vg(\ell),\vg(\ell+1)\big)\right.
\]
\[
\left.
\log \Pi_q^3\big(\vg,t_{0},\ldots,t_{q-1}\big)
\prod_{\ell=0}^{q-1}\dd t_{\ell}\vphantom{\sum_{\substack{\vg,\;\vg(0)=
\vecn,\;\vg(q)=\vecm\\ \sup_{1\le\ell\le q-1}|\vg(\ell)-\vecnf|\le \sqrt K\,\log K}}
\int\cdots\int_{t_{0}+\cdots
  t_{q-1}<1}}\;\right|\le \cte{boulot3}.
\]
\end{sublemma}
\begin{proof}
By a similar computation as in the proof of Sublemma \ref{boulot2} we get
\[
q!\times \!\!\!\sum_{\substack{\vg:\;\vg(0)= \vecn,\;\vg(q)=\vecm\\ \sup_{1\le\ell\le q-1}|\vg(\ell)-\vecnf|\le \sqrt K\,\log K}}
\int\cdots\int_{t_{0}+\cdots+ t_{q-1}<1}\;
\prod_{\ell=0}^{q-1}p^{*}\big(\vg(\ell),\vg(\ell+1)\big)\quad\times
\]
\[
\big(\Lambda\big(\vecm\big)-\Lambda_{*}\big)\,\Big(1-\sum_{\ell=0}^{q-1}t_{\ell}\Big)
\prod_{\ell=0}^{q-1}\dd t_{\ell}
\]
\[
=\sum_{\substack{\vg:\;\vg(0)=\vecn,\;\vg(q)=\vecm\\ \sup_{1\le\ell\le q-1}|\vg(\ell)-\vecnf|\le \sqrt K\,\log K}}
\big(\Lambda\big(\vecm\big)-\Lambda_{*}\big)\,\left(1-q\, q!\frac{1}{(q+1)!}\right)
\;\prod_{\ell=0}^{q-1}p^{*}\big(\vg(\ell),\vg(\ell+1)\big)
\]
\[
=\sum_{\substack{\vg:\;\vg(0)=\vecn,\;\vg(q)=\vecm \\ \sup_{1\le\ell\le q-1}
|\vg(\ell)-\vecnf|\le \sqrt K\,\log K}}\;\big(\Lambda\big(\vecm\big)-\Lambda_{*}\big)\; \frac{1}{q+1}\;
\prod_{\ell=0}^{q-1}p^{*}\big(\vg(\ell),\vg(\ell+1)\big)
.
\]
We have since $\vecm\in\Delta$
\[
K^{-1}\;\big(\Lambda\big(\vecm\big)-\Lambda_{*}\big)
=\big[\lambda\big(\vecm/K\big)-
\lambda\big(\vecnf/K\big)\big]=\mathcal{O}\big(K^{-1/2}\big).
\]
Since $q$ is of order $K$ we get
\[
\left|\frac{\Lambda\big(\vecm\big)-\Lambda_{*}}{q}\right|\le
\mathcal{O}\big(K^{-1/2}\big). 
\]
The result follows from  \eqref{coupe} and from the local limit theorem applied to $\plat$.
\end{proof}

Let us finally estimate the contribution of $\Pi_q^1$. 
\begin{sublemma}\label{boulot1}
There exists a constant $\cte{boulot1}>0$ independent of $K$ such that for
$K$ large enough and all $q\in[\Lambda_{*}-\sqrt K,\Lambda_{*}+\sqrt K]$
\[
K^{d/2}\;q!\;\left|\;\sum_{\substack{\vg:\;\vg(0)=\vecn,\;\vg(q)=\vecm\\ \sup_{1\le\ell\le 
    q-1}|\vg(\ell)-\vecnf|\le \sqrt K\,\log K}}
\int\cdots\int_{t_{0}+\cdots + t_{q-1}<1}\; \prod_{\ell=0}^{q-1}p^{*}\big(\vg(\ell),\vg(\ell+1)\big)\right.
\]
\[
\left.
\log \Pi_q^1\big(\vg,t_{0},\ldots,t_{q-1}\big)
\prod_{\ell=0}^{q-1}\dd t_{\ell}\vphantom{\sum_{\substack{\vg,\;\vg(0)=\vecn,\;\vg(q)=\vecm\\ \sup_{1\le\ell\le 
    q-1}|\vg(\ell)-\vecnf|\le \sqrt K\,\log K}}
\int\cdots\int_{t_{0}+\cdots +t_{q-1}<1}}\;\right|
\le \cte{boulot1}.
\]
\end{sublemma}
\begin{proof}
We have 
\begin{align*}
\MoveEqLeft
p\big(\vg(\ell),\vg(\ell)+\vece{j}\big) =\\ 
& p^{*}\big(\vg(\ell),\vg(\ell)+\vece{j}\big)
 +K^{-1}\;\sum_{r=1}^{d}\mb_{j,r}\,(\vg(\ell)_{r}-\vecnf_{r})
+K^{-2}\;\mathcal{O}\big(|\vg(\ell)-\vecnf|^{2}\big)
\end{align*}
where
\[
\mb_{j,r}=\partial_{x_r}\left(\frac{B_{j}(\vecx)}{\lambda(\vecx)}\right)\bigg|_{\vecx=\vecxf}.
\]
Similarly
\begin{align*}
\MoveEqLeft p\big(\vg(\ell),\vg(\ell)-\vece{j}\big)=\\
&p^{*}\big(\vg(\ell),\vg(\ell)+\vece{j}\big)
+K^{-1}\sum_{r=1}^{d}\md_{j,r}\,(\vg(\ell)_{r}-\vecnf_{r})
+K^{-2}\mathcal{O}\big(|\vg(\ell)-\vecnf|^{2}\big)
\end{align*}
where
\[
\md_{j,r}=\partial_{x_r}\left(\frac{D_{j}(\vecx)}{\lambda(\vecx)}\right)\bigg|_{\vecx=\vecxf}.
\]
We then have  
\[
\log\left( 
\frac{p\big(\vg(\ell),\vg(\ell)+\vece{j}\big)}
{p^{*}\big(\vg(\ell),\vg(\ell)+\vece{j}\big)}\right)=
\frac{1}{K\,r_{j}}\,
\sum_{r=1}^{d}\mb_{j,r}\,(\vg(\ell)_{r}-\vecnf_{r})
+K^{-2}\;\mathcal{O}\big(|\vg(\ell)-\vecnf|^{2}\big),
\]
with $r_{j}$ defined in \eqref{rj},
and
\[
\log\left( 
\frac{p\big(\vg(\ell),\vg(\ell)-\vece{j}\big)}
{p^{*}\big(\vg(\ell),\vg(\ell)-\vece{j}\big)}\right)=
\frac{1}{K\,r_{j}}\;
\sum_{r=1}^{d}\md_{j,r}\,(\vg(\ell)_{r}-\vecnf_{r})
+K^{-2}\;\mathcal{O}\big(|\vg(\ell)-\vecnf|^{2}\big).
\]

The result follows using Sublemma \ref{compense} and Sublemma \ref{moment}. 
\end{proof}

\begin{sublemma}\label{lem:qrobaproba}
For all $K$ large enough
\[
\sup_{\vecn \in\Delta'}\, \sup_{q\geq \ell> q/3} \;\sup_{\|\vecu-\vecnf\|<\sqrt{K}\log K}
|\qroba_{\vecn}(\plat_\ell=\vecu)-\proba_{\vecn}(\plat_\ell=\vecu)|\, \|\vecu-\vecnf\|
\leq \e^{-\Oun (\log K)^2}.
\]
\end{sublemma}

\begin{proof}
We have
\begin{align*}
\MoveEqLeft
\sup_{\vecn \in\Delta'} \sup_{q\geq \ell> q/3} \sup_{\|\vecu-\vecnf\|<\sqrt{K}\log K}
|\qroba_{\vecn}(\plat_\ell=\vecu)-\proba_{\vecn}(\plat_\ell=\vecu)|\, \|\vecu-\vecnf\| \\
& \leq 
\sqrt{K}\log K\, \sup_{\vecn \in\Delta'}\, \sup_{q\geq \ell> q/3}\,\, \sup_{\|\vecu-\vecnf\|<\sqrt{K}\log K}
|\qroba_{\vecn}(\plat_\ell=\vecu)-\proba_{\vecn}(\plat_\ell=\vecu)|.
\end{align*}
Using Bonferroni's inequality we have
\begin{align*}
\left|\qroba_{\vecn}(\plat_\ell=\vecu)-\proba_{\vecn}(\plat_\ell=\vecu)\right|
 & \leq \sum_{s=1}^\ell \esperance_{\vecn}
     \left[ \un_{\{|\plat_s-\vecnf|>\sqrt{K} \log K\}}\, \un_{\{\plat_\ell=\vecu\}}\right] \\
 & \leq  \sum_{s=1}^q \sum_{j=1}^d
     \, \proba_{\vecn}\bigg( |(\plat_s)_j-\vecnf_j|>\frac{\sqrt{K} \log K}{d}\bigg)\\
&=\sum_{s=1}^q \sum_{j=1}^d \, \proba_{\vecz}\bigg( |(\plat_s)_j+\vecn_j -\vecnf_j|>\frac{\sqrt{K} \log K}{d}\bigg)\\
& \leq \e^{-\Oun (\log K)^2}
\end{align*}
where the last inequality follows by applying Hoeffding's inequality.
\end{proof}

\begin{sublemma}\label{lem:erreurproba}
For all $K$ large enough
\[
\sup_{\vecn \in\Delta'} \sup_{q\geq \ell> q/3} 
\sum_{\|\vecu-\vecnf\| \geq \sqrt{K}\log K}
\proba_{\vecn}(\plat_\ell=\vecu)\, \|\vecu-\vecnf\|
\leq \e^{-\Oun (\log K)^2}.
\]
\end{sublemma}

\begin{proof}
We have 
\begin{align*}
\MoveEqLeft  \sup_{\vecn \in\Delta'} \sup_{q\geq \ell> q/3} 
\sum_{\|\vecu-\vecnf\| \geq \sqrt{K}\log K}
\proba_{\vecn}(\plat_\ell=\vecu)\, \|\vecu-\vecnf\|\\
& \leq
\sup_{\vecn \in\Delta'} \sup_{q\geq \ell> q/3} 
\sum_{j=1}^d\sum_{|\vecu_j-\vecnf_j| \geq \sqrt{K}\log K/\sqrt{d}}
\proba_{\vecn}(\plat_\ell=\vecu)\, \sum_{k=1}^d |\vecu_k-\vecnf_k |\\
&=
\sup_{\vecn \in\Delta'} \sup_{q\geq \ell> q/3} 
\sum_{j=1}^d \sum_{k=1}^d\sum_{|\vecu_j-\vecnf_j| \geq \sqrt{K}\log K/\sqrt{d}}
\proba_{\vecn}(\plat_\ell=\vecu)\,  |\vecu_k-\vecnf_k |\\
& \leq \e^{-\Oun (\log K)^2}
\end{align*}
where the last estimate can be easily obtained by using again Hoeffding's inequality since $\plat_\ell$ is a sum
of independent identically distributed random variables, distinguishing the cases $j=k$
and $j\neq k$.
\end{proof}

\appendix
\section{A difference between monotype and multitype birth-and-death processes}
\label{selfadjoint}

In the one-dimensional case \cite{ccm}, the method is based on the existence of a reference measure on
$\mathbb{N}^*$ and an associated $\ell^2$ space such that the generator $\gen$ is self-adjoint in the space. This is
not the case in dimension strictly larger than one for a generator $\gen$ defined by 
\[
\gen f(\vecn)= \sum_{j=1}^{d} \boldsymbol{\lambda}_j(\vecn)\big(f(\vecn+\vece{j})-f(\vecn)\big)+
\boldsymbol{\mu}_j(\vecn)\big(f(\vecn-\vece{j})-f(\vecn)\big),
\] 
where $\boldsymbol{\lambda}_j,\boldsymbol{\mu}_j:\domaine\to\real^+$, $\vecn\in \dintegers$ and $f:\dintegers\to\real$ has finite support. 

For a positive measure $\pi$ on $\domaine$ and two functions $f,g:\domaine\to\real$ with finite support, define
\[
\langle g,  f\rangle_{\pi} = \sum_{\vecn\in\domaine} f(\vecn)\, g(\vecn)\, \pi(\vecn).
\]
\begin{proposition}
A positive measure $\pi$ on $\domaine$ satisfies 
\begin{equation}
\label{auto}\langle g, \gen f\rangle_{\pi} = \langle \gen g, f\rangle_{\pi}
\end{equation}
for all $f, g$ with finite support, if and only if 
\[
\pi(\vecn)\,\boldsymbol{\lambda}_j(\vecn) = \boldsymbol{\mu}_j\big(\vecn+\vece{j}\big)\, \pi\big(\vecn+\vece{j}\big), \quad  \forall j\in \{1,\ldots,d\},
\forall \vecn \in \domaine.
\]
\end{proposition}

\begin{proof}
Equation \eqref{auto} will be satisfied for all functions with finite support if and only if the equality is true for
$f =\un_{\ushort{p}}$ and $g =\un_{\ushort{q}}$ for all $\ushort{p}$ and all $\ushort{q}$ in $\domaine$. The result follows immediately by direct computations.
\end{proof}

This proposition has a consequence which can be cast in terms of circuits in $\domaine$. A circuit
$\mathcal{C}$ is a path of the form $\big(\vecn^{(1)}, \ldots, \vecn^{(k)}, \vecn^{(k+1)}\big)$, for some positive integer $k$, such that $\vecn^{(k+1)}=\vecn^{(1)}$ and
$\vecn^{(\ell+1)}= \vecn^{(\ell)} +  \ushort{\varepsilon}^{(j_{\ell})}$, for $\ell = 1, \ldots, k$,
where $\ushort{\varepsilon}^{j}=\pm\, \vece{j}$, with the constraint that
$\ushort{\varepsilon}^{(j_{1})}+\cdots + \ushort{\varepsilon}^{(j_{k})}=\vecz$.
Now define
\[
\rho(\ushort{\varepsilon}^{(j)}, \vecn)=
\begin{cases}
\frac{\boldsymbol{\lambda}_{j}(\vecn)}{\boldsymbol{\mu}_{j}(\vecn+\vece{j})} &
\hbox{if} \quad \ushort{\varepsilon}^{(j)} = \vece{j}\\ \\ 
\frac{\boldsymbol{\mu}_{j}(\vecn)}{\boldsymbol{\lambda}_{j}(\vecn-\vece{j})} & \hbox{if}
\quad \ushort{\varepsilon}^{(j)} = - \vece{j}.
\end{cases}
\]

\begin{corollary}
There exists a positive measure $\pi$ on $\domaine$ such that \eqref{auto} holds if and only if for all circuits $\mathcal{C}$ contained in $\domaine$ one has
\begin{equation}
\label{autolib}
\prod_{\ell=1}^k \rho\big(\ushort{\varepsilon}^{(j_{\ell})}, \vecn^{(\ell)}\big)= 1.\end{equation}
\end{corollary}

\begin{proof}
The proof is elementary and left to the reader. It is enough to observe that for all $\vecn\in\domaine$ and all $1\leq j\leq d$,
\[
\frac{\pi\big(\vecn + \ushort{\varepsilon}^{(j)}\big)}{\pi(\vecn)} = \rho\big(\ushort{\varepsilon}^{(j)}, \vecn\big).
\]
\end{proof}

Condition \eqref{autolib} is always satisfied in dimension one, but it imposes very stringent conditions on the demographic parameters in higher dimensions. This is why we had to follow another route in the present work.
Let us illustrate this fact in dimension two with the following example modelling two populations with both intra- and inter-specific competition:
\begin{equation}
\label{case2}
\begin{cases}
\boldsymbol{\lambda}_1(n_{1}, n_{2}) = \lambda_{1} n_{1} ,& \boldsymbol{\mu}_1(n_{1}, n_{2}) = n_{1} (\mu_{1} + c_{11}n_{1}+ c_{12} n_{2})\\
\boldsymbol{\lambda}_2(n_{1}, n_{2}) = \lambda_{2} n_{2} ,& \boldsymbol{\mu}_2(n_{1}, n_{2}) = n_{2} (\mu_{2} + c_{21}n_{1}+ c_{22} n_{2})
\end{cases} \end{equation}
where $\lambda_{k}, \mu_{k}, c_{k\ell}$, $k, \ell = 1,2$, are nonnegative parameters. 
In this case, condition \eqref{autolib} reads
\[
\frac{\boldsymbol{\lambda}_{1}(n_{1}, n_{2})}{\boldsymbol{\mu}_{1}(n_{1}+1, n_{2})}
\frac{\boldsymbol{\lambda}_{2}(n_{1}+1, n_{2})}{\boldsymbol{\mu}_{2}(n_{1}+1, n_{2}+1) }
\frac{\boldsymbol{\mu}_{1}(n_{1}+1, n_{2}+1)}{\boldsymbol{\lambda}_{1}(n_{1}, n_{2}+1)}
\frac{\boldsymbol{\mu}_{2}(n_{1}, n_{2}+1)}{\boldsymbol{\lambda}_{2}(n_{1}, n_{2})} = 1
\]
for all $(n_{1},n_{2})\in \integers^2\backslash\{(0,0)\}$.
Using \eqref{case2} and assuming that $c_{21}\neq 0$ and $c_{12}\neq 0,$ (so we do have inter-specific interactions),  condition \eqref{autolib} is fulfilled if and only if
\[
c_{11} = c_{12}\,, c_{21} = c_{22}\,,  \mu_{1} c_{21} =  \mu_{2} c_{12}.
\]
If one of these conditions is violated, there is no positive measure satisfying \eqref{auto}.


\end{document}